\documentclass[11pt]{amsart}

\usepackage[T1]{fontenc}
\usepackage{lmodern}
\usepackage{amsmath,amssymb,amsthm,mathtools}
\usepackage{mathrsfs}
\usepackage{booktabs,array}
\usepackage{enumitem}
\usepackage{needspace}
\usepackage{microtype}
\usepackage[colorlinks=true,linkcolor=blue,citecolor=blue,urlcolor=blue]{hyperref}

\theoremstyle{plain}
\newtheorem{theorem}{Theorem}[section]
\newtheorem{corollary}[theorem]{Corollary}
\newtheorem{proposition}[theorem]{Proposition}
\newtheorem{lemma}[theorem]{Lemma}

\newtheorem{problem}[theorem]{Problem}
\theoremstyle{definition}
\newtheorem{definition}[theorem]{Definition}

\newtheorem{example}[theorem]{Example}
\theoremstyle{remark}
\newtheorem{remark}[theorem]{Remark}

\numberwithin{equation}{section}

\DeclareMathOperator{\Tr}{Tr}
\DeclareMathOperator{\Norm}{N}
\DeclareMathOperator{\ord}{ord}

\DeclareMathOperator{\Frob}{Frob}
\DeclareMathOperator{\disc}{disc}
\DeclareMathOperator{\Spec}{Spec}

\DeclareMathOperator{\Pic}{Pic}
\newcommand{\F}{\mathbb F}
\newcommand{\Z}{\mathbb Z}
\newcommand{\Q}{\mathbb Q}
\newcommand{\C}{\mathbb C}
\newcommand{\PP}{\mathbb P}

\newcommand{\cO}{\mathcal O}
\newcommand{\eps}{\varepsilon}
\newcommand{\Gm}{\mathbf G_m}
\newcommand{\cC}{\mathcal C}
\newcommand{\cE}{\mathcal E}

\title[Coset-refined trace statistics in cubic norm tori]{Coset-refined trace statistics, nodal characters, and affine branches in cubic norm tori}
\author{Henry Shin}

\begin{document}
\raggedbottom

\begin{abstract}
Prescribed trace/norm estimates and Soto-Andrade-type sums control whole
fibers or related global character sums.  We prove a coset-refined trace theorem
for cubic norm-one tori.  Let \(B/\F_q\) be finite \'etale cubic,
\(\operatorname{char}\F_q\ne2,3\), and let
\(T_B=\ker(\Norm_{B/\F_q}:\operatorname{Res}_{B/\F_q}\Gm\to\Gm)\).  For every
subgroup \(H\subset T_B(\F_q)\) of index \(m\), every coset \(gH\), every
\(\gamma\in B^\times\), and every smooth fiber
\(\Tr(\gamma h)=s\), \(s^3\ne27\Norm(\gamma)\), we prove
\[
\begin{aligned}
N_{gH,B}(s;\gamma)
&=m^{-1}N_B(s,\Norm\gamma)+E_{gH,B}(s;\gamma),\\
|E_{gH,B}(s;\gamma)|
&\le3(1-1/m)\sqrt q .
\end{aligned}
\]
The geometric input is a Picard--Kummer kernel calculation: no nontrivial torus
character becomes geometrically constant on a smooth trace/norm curve, so
nontrivial coset character sums have square-root cancellation.  On the nodal
boundary \(s^3=27\Norm(\gamma)\), the kernel degenerates exactly to a cyclic
cubic Kummer kernel.  Its Frobenius-fixed part is the sole source of order-\(q\)
bias; after removing that explicit projection, remaining characters again have
square-root cancellation up to bounded normalization/node correction.

The same geometry gives local branch theory for \(\Tr_A(\gamma\eta^n)=c\) over
finite \'etale cubic \(\Z_p\)-algebras, \(p\ge5\).  The logarithmic tangent and
trace-dual codifferent coordinates identify singular branches: nondegenerate
classes have quadratic Hensel models, while the genuinely affine degenerate
class has a cubic first-obstruction model; in full norm-fiber orbits singular
branch counting reduces to one cubic norm equation.
\end{abstract}

\maketitle

\section{Introduction}\label{sec:intro}

Let \(R\) be a base ring, let \(A/R\) be a finite \'etale cubic algebra,
let \(\eta\in A^\times\), let \(\gamma\in A\), and let \(c\in R\).  We write
\(\Tr_{A/R}\) for the trace, abbreviated to \(\Tr_A\) once the base is fixed.
This paper studies local solutions of
\[
        \Tr_{A/R}(\gamma\eta^n)=c,
        \qquad n\in\Z.
\]
The affine target is essential.  The case \(c=0\) is a homogeneous trace-zero
intersection, while \(c\ne0\) introduces a new codifferent class whose first
nonzero local model can be cubic rather than quadratic.  At a rational prime
\(p\), the reduction of a cubic algebra may be split, mixed, or inert, so the
natural local object is a finite \'etale cubic \(\Z_p\)-algebra rather than only
an unramified field extension.

The guiding principle is that trace/norm geometry over finite fields provides
the skeleton, but the paper's new output is the coset and branch theory built on
that skeleton.  The prescribed trace/norm problem itself has a substantial
literature, including Katz's work on Soto-Andrade sums, Moisio's cubic
elliptic-curve formulae, the Moisio--Wan refinement of Katz's bound,
Rojas-Le\'on's trace/norm \(L\)-functions, Rojas-Le\'on--Wan's toric
Calabi--Yau moment-zeta work, Rojas-Le\'on's estimates for curves with many
automorphisms, the finite-\'etale-algebra estimates of Lin--Wan, Wan's extension
to finite semisimple algebras, and related recent prescribed trace/norm
point-count work
\cite{KatzSotoAndrade1993,Moisio2008,MoisioWan2010,RojasLeon2012,RojasLeonWan2007,RojasLeon2013,LinWan2025,Wan2026,AlvarengaBorges2023}.
Here those results are background rather than the novelty claim.

Two nearby comparison points should be separated from the present result.
Soto-Andrade and related character-sum estimates give bounds for trace-type
sums on norm-one groups, while prescribed trace/norm estimates control whole
fibers.  Restricted-norm work studies trace conditions with the norm constrained
to cosets in the base-field group \(\F_q^\times\)
\cite{KononenMoisioRintaAhoVaananen2008}.  Here the coset condition is imposed
instead inside the two-dimensional cubic norm-one torus \(T_B(\F_q)\), uniformly
over arbitrary finite-index subgroup cosets, all cubic finite \'etale splitting
types, and all smooth trace/norm fibers.

In the inert trace-zero specialization, the trace hyperplane is related to the
classical Singer difference-set and Gauss-period literature.  For example,
M\'esz\'aros--R\'onyai--Szab\'o describe planar Singer difference sets using
norm-one elements in \(\F_{q^3}/\F_q\) \cite{MeszarosRonyaiSzabo2019}.  The
present theorem is different in scope: it treats arbitrary subgroup cosets in
every smooth cubic trace/norm fiber, not only the trace-zero inert case; it is
uniform across split, mixed, and inert cubic \'etale types; and on the nodal
boundary it identifies the exact Frobenius-fixed cubic Kummer projection
responsible for all order-\(q\) secondary terms.

The novelty claimed here is deliberately narrow.  We do not claim priority for
the underlying prescribed trace/norm counts, nor for the existence of
Soto-Andrade-type character-sum estimates on norm-one groups.  The new point is
the coset-refined problem inside the cubic norm-one torus itself: arbitrary
finite-index subgroup cosets are equidistributed on every smooth trace/norm
fiber, and the nodal boundary is resolved by an explicit Frobenius-fixed cubic
Kummer kernel.  The codifferent branch theory then translates the same geometry
into local affine branch singularities, explicit norm equations, and finite-jet
statistics.

\subsection*{Relation with local Skolem questions}
The problem of determining zeros of linear recurrence sequences is classically
connected with the Skolem--Mahler--Lech theorem
\cite{Skolem1935,Mahler1935,Lech1953}; see
\cite{EverestPoortenShparlinskiWard2003} for recurrence background.  Recent
work on local and structured recurrence zeros includes the \(p\)-adic zeros of
the Tribonacci sequence \cite{BiluLucaNieuwveldOuaknineWorrell2024}.  The
present paper does not claim a general Skolem decision theorem.  Instead, it
restricts the input class to cubic norm-torus trace recurrences and obtains a
much more explicit local description: branch derivatives, singular directions,
intersection multiplicities, and finite-field branch counts are all expressed
inside the original \'etale algebra by trace, norm, and codifferent linear
algebra.  For orientation on the current low-order Skolem landscape, see also
Bacik's order-four decidability result \cite{Bacik2025}.  For recent algorithms
for computing \(p\)-adic zeros of general linear recurrence sequences, with
conditional termination under the \(p\)-adic Schanuel conjecture, see
\cite{BacikOuakninePurserWorrell2026}.

\subsection*{Frobenius convention}
For finite fields, \(\operatorname{Frob}_q\) denotes arithmetic Frobenius
\(x\mapsto x^q\), while \(\operatorname{Fr}_q=\operatorname{Frob}_q^{-1}\)
denotes geometric Frobenius in cohomological trace formulae.  If \(s\in\mathbb A^1(F)\), then \(\operatorname{Fr}_s\) denotes the geometric Frobenius element at the rational point \(s\); under the standard identification of the stalk with compactly supported cohomology of the fiber, this is the same \(\operatorname{Fr}_q\)-action used in the displayed cohomological formulae.  Descent data for
ordering torsors are written using arithmetic Frobenius; replacing a Frobenius
permutation by its inverse does not change its sign, its fixed coordinate
labels, or the nodal cubic-kernel criterion used below.

\subsection*{Main theorem package and dependencies}
The finite-field subgroup theorem is the headline result.  The local branch and codifferent theory supplies the intrinsic coordinates that make its branch consequences explicit, and the statistics section records two natural distributional refinements.  The results are organized around three theorem packages.

\medskip
\noindent\textbf{Main Theorem 1: local affine branch and codifferent singular
classification.}
Let \(p\ge5\), let \(A/\Z_p\) be finite \'etale cubic, let
\(\eta\in A^\times\), \(\gamma\in A\), and \(c\in\Z_p\).  Put
\[
        P=\ord\bigl(\bar\eta\in(A/pA)^\times\bigr),
        \qquad \eta^P=1+pU,
        \qquad \omega=\bar U,
        \qquad s=\bar c .
\]
Choose the representative \(a\in\{0,\ldots,P-1\}\) for each period class.  On
such a branch \(n=a+Pt\), the affine target is governed by
\[
        F_{a,c}(t)=\Tr_A\bigl(\gamma\eta^a(\eta^P)^t\bigr)-c .
\]
The modulo-\(p\) affine target classes are
\[
        Z_{p,c}(1)=\{a\bmod P:
        \Tr_{A/pA/\F_p}(\bar\gamma\bar\eta^a)=s\}.
\]
For such a target class the normalized first-order coefficient of
\(F_{a,c}(t)/p\) modulo \(p\) is
\[
        d_a=\Tr_{A/pA/\F_p}(\bar\gamma\bar\eta^a\omega).
\]
Equivalently, the literal derivative of \(F_{a,c}\) is divisible by \(p\), and
\(d_a\) is the coefficient that appears after this first factor of \(p\) is
removed.  Classes outside \(Z_{p,c}(1)\) have no Hensel branch above the affine target.
For \(a\in Z_{p,c}(1)\), if \(d_a\ne0\), there is a unique simple Hensel
branch and \(v_p(F_{a,c}(t))=1+v_p(t-\tau)\).  The singular alternatives below are applied only after the denominator-clearing and primitive-reduction steps of Lemmas~\ref{lem:denominator-clearing} and \ref{lem:primitive-reduction}; in particular the quadratic and cubic normal forms require the reduced class
\[
        x_a=\bar\gamma\bar\eta^a
\]
to be nonzero.  If \(x_a=0\), the class is governed by primitive reduction and the finite digit recursion, not by the quadratic or cubic normal form.  If
\(a\in Z_{p,c}(1)\), \(d_a=0\), \(x_a\ne0\), and \(1,\omega,\omega^2\) is a basis of
\(A/pA\), the singular directions are identified by the trace-dual basis
\(z_0,z_1,z_2\).  Nondegenerate singular classes have a quadratic Hensel
polynomial; the affine degenerate class \(x=s z_0\) has a cubic
first-obstruction polynomial only under the additional hypotheses \(s\ne0\),
\(\omega\in(A/pA)^\times\), and the required lower obstructions vanish.  Under the explicit unit-coefficient and full norm-fiber hypotheses used in
Section~\ref{sec:codifferent-census}, singular classes are counted by one cubic
norm equation.  Here, whenever \(B/F\) is a finite \'etale cubic algebra and \(\omega\)
generates \(B\), let \(m_\omega:B\to B\) be multiplication by \(\omega\), and write
\[
        f_\omega(T)=\det(T\cdot\operatorname{id}_B-m_\omega)
\]
for its characteristic polynomial.  Equivalently, because \(\omega\) generates
\(B\), \(f_\omega\) is the monic squarefree generator polynomial and
\(B\simeq F[T]/(f_\omega)\).
In the reduced algebra \(B=A/pA\), when \(\bar\gamma\in B^\times\) and the norm
fiber is \(\Norm(h)=\nu\) with \(\nu\in\F_p^\times\), the homogeneous case is governed by the number of
\(u\in\F_p^\times\) satisfying
\[
        u^3=-\Norm_{B/\F_p}(\bar\gamma)\nu\disc(f_\omega)
        \quad\text{in }\F_p .
\]

\medskip
\noindent\textbf{Main Theorem 2: subgroup-coset equidistribution and nodal secondary
terms.}
Let \(B/\F_q\) be a finite \'etale cubic algebra of characteristic different
from \(2\) and \(3\), and let
\[
        T_B=\ker\bigl(\Norm_{B/\F_q}:\operatorname{Res}_{B/\F_q}\Gm\to\Gm\bigr)
\]
be the norm-one torus, so
\(T_B(\F_q)=\{h\in B^\times:\Norm_{B/\F_q}(h)=1\}\).  For every subgroup
\(H\subset T_B(\F_q)\), every coset \(gH\), every \(\gamma\in B^\times\), and every \(s\in\F_q\) with smooth
trace/norm fiber \(s^3\ne27\Norm(\gamma)\), define
\[
        N_{gH,B}(s;\gamma)
        =\#\{h\in gH:\Tr_{B/\F_q}(\gamma h)=s\},
\]
and
\[
        N_B(s,n)=\#\{x\in B^\times:
        \Tr_{B/\F_q}(x)=s,
        \Norm_{B/\F_q}(x)=n\}.
\]
Then
\[
        N_{gH,B}(s;\gamma)
        =\frac1{[T_B(\F_q):H]}N_B(s,\Norm\gamma)+E_{gH,B}(s;\gamma),
\]
where \(E_{gH,B}(s;\gamma)\) is defined by this identity, \(m=[T_B(\F_q):H]\), and
\[
        |E_{gH,B}(s;\gamma)|\le 3(1-1/m)\sqrt q<3\sqrt q .
\]
The proof is organized through Lang/Kummer character sheaves and the relative
trace complexes \(R\tau_!\mathcal L_\chi\), whose stalks recover the subgroup
character sums.  The key geometric
input is a Picard-group calculation at the three points at infinity, which shows
that no nontrivial torus character becomes geometrically constant on a smooth
trace/norm fiber.  At the nodal boundary \(s^3=27\Norm(\gamma)\), this obstruction
collapses exactly for one cyclic cubic-character group.  The resulting
exceptional projection is computed explicitly and gives all order-\(q\) secondary
terms; the remaining characters satisfy square-root cancellation.

\medskip
\noindent\textbf{Main Theorem 3: branch statistics.}
The codifferent singular census becomes a family statistic.  In the finite-field
setting, let \(F=\F_q\), let \(B/F\) be finite \'etale cubic, let
\(\gamma\in B^\times\), and fix a norm fiber \(\Norm(h)=\nu\) with \(\nu\in F^\times\).  For a generator
\(\omega\) of \(B\), let \(m_\omega:B\to B\) be multiplication by \(\omega\), and
write
\[
        f_\omega(T)=\det(T\cdot\operatorname{id}_B-m_\omega)
\]
for its characteristic polynomial.  As
\(\omega\) varies over generators of \(B\), the relevant cube class is
\[
        -\Norm_{B/F}(\gamma)\nu\disc(f_\omega)
        \in F^\times/(F^\times)^3 .
\]
When \(q\equiv1\pmod3\), the three cube classes are equidistributed with an
\(O(q^{5/2})\) error in the three-dimensional family, with an absolute implied
constant uniform in the cubic \'etale type, the algebra \(B\), the coefficient
\(\gamma\), the chosen norm fiber \(\nu\), and the cube class; when
\(q\equiv2\pmod3\), the cube map on \(F^\times\) is bijective and
there is only one cube class.  The proof reduces the nontrivial cubic-character
sums to
a rank-one Kummer sheaf on a projective line with three punctures.  Conditional
on a nondegenerate singular disk, the two lower quadratic Hensel coefficients
are uniformly distributed in the natural full finite-jet lift family, giving exact
frequencies for the no-lift, two-simple-branch, and double-root alternatives.
These finite-jet frequencies are not asserted for the jets arising from an
arbitrary fixed recurrence unless a separate sampling or parametrization argument
shows that the recurrence-induced jets run through that lift family with the
required distribution.

\subsection*{Roadmap}
Section~\ref{sec:finitefield} records the cubic trace/norm curve and its
\(S_3\)-twisted point counts; this is the finite-field infrastructure.  Section
\ref{sec:general-local-skolem} proves the local branch theorem, the quadratic
and cubic singular models, and the Weierstrass multiplicity bounds.  Section
\ref{sec:codifferent-census} gives the trace-dual singular line and exact branch
census.  Section~\ref{sec:proper-subgroups} proves the subgroup-coset theorem and the
nodal cubic-character formula.  Section~\ref{sec:singular-statistics} proves the
finite-field and finite-jet branch statistics.  Appendix~\ref{sec:higher-rank}
records the sharp higher-rank bounds, and Appendix~\ref{sec:inert-wieferich}
records the inert norm-one Wieferich interpretation.

\medskip
\noindent\textbf{Dependency map.}
The finite-field trace/norm geometry of Section~\ref{sec:finitefield} supplies
both the mod-\(p\) branch skeleton and the compactified curves used in the
subgroup estimates.  Main Theorem~1 uses this geometry through the codifferent
line of Section~\ref{sec:codifferent-census}; its branch-census identities use
\(N_B(s,n)\) as the actual affine trace/norm count, with closed nodal evaluations
explicitly deferred to Proposition~\ref{prop:full-nodal-count} in Section~\ref{sec:proper-subgroups}.  Main Theorem~2 uses the same
curves through the Lang/Kummer trace complexes of Section~\ref{sec:proper-subgroups}.
Main Theorem~3 is then a statistical refinement of the codifferent census and of
the quadratic Hensel step.  The rank-\(d\) and inert-Wieferich appendices record optimality and
exceptional-prime behavior; they are auxiliary to the three main theorem
packages and do not introduce further unproved hypotheses.

\section{Prescribed trace/norm curves and their \texorpdfstring{\(S_3\)}{S3}-twists}\label{sec:finitefield}

This section isolates the geometric input behind the mod-\(p\) affine trace-target classes.  We use standard finite-field trace and norm notation as in, for example, \cite{LidlNiederreiter1997}.  The prescribed trace/norm problem over finite fields is classical; see Katz's Soto-Andrade estimates, Moisio's trace/norm formulae and elliptic-curve interpretation, the Moisio--Wan refinement, Rojas-Le\'on's trace/norm \(L\)-functions, and the recent finite-\'etale and finite-semisimple algebra extensions of Lin--Wan and Wan \cite{KatzSotoAndrade1993,Moisio2008,MoisioWan2010,RojasLeon2012,LinWan2025,Wan2026,AlvarengaBorges2023}.  The primitive trace/norm literature, including work of Cohen--Huczynska and Cohen, is also part of the surrounding context \cite{CohenHuczynska2003,Cohen2012}.  Here the trace-zero and prescribed trace/norm cubic slices are used as geometric input, including their three \'etale splitting types, because those are exactly the reductions that occur in the local branch theorem below.

The quotient below is a concrete trace/norm form of the classical Hesse--Fermat cubic and \(j=0\) isogeny geometry.  For the Hesse pencil and its \(3\)-torsion interpretation, see \cite{ArtebaniDolgachev2009,AnemaTopTuijp2018}; for standard elliptic-curve and isogeny background, see \cite{Silverman2009}.  Let \(F\) be a field of characteristic different from \(2\) and \(3\), and let \(\varepsilon\in F^\times\).  In the split cubic algebra \(F^3\), write
\[
        x_1+x_2+x_3=0,\qquad x_1x_2x_3=\varepsilon .
\]
After setting \(x=x_1\), \(y=x_2\), and \(x_3=-x-y\), this is the affine plane cubic
\[
        xy(x+y)=-\varepsilon .
\]
Let \(\cC_\varepsilon\) be its smooth projective closure.

Throughout this section, and in later references to the same notation, every
displayed Weierstrass equation
\[
        E_\bullet:\quad V^2=\text{a cubic polynomial in }U
\]
denotes the smooth projective Weierstrass curve over the relevant base field,
with the displayed equation as its affine chart and with the usual point at
infinity included.  Thus point counts such as \(\#E_\bullet(F)\) are projective
point counts.

\begin{theorem}[Discriminant quotient of the trace-zero norm curve]\label{thm:universal-discriminant-isogeny}
Let
\[
        \cE_\varepsilon: V^2=-4U^3-27\varepsilon^2 .
\]
The cyclic permutation
\[
        \rho(x_1,x_2,x_3)=(x_2,x_3,x_1)
\]
acts freely on \(\cC_\varepsilon\).  The invariant functions
\[
        U=x_1x_2+x_1x_3+x_2x_3,
        \qquad
        V=(x_1-x_2)(x_2-x_3)(x_3-x_1)
\]
define an \(F\)-morphism
\[
        \pi_\varepsilon:\cC_\varepsilon\longrightarrow \cE_\varepsilon
\]
which is invariant under \(\rho\).  After base change to \(\overline F\), the morphism
\[
        \pi_{\varepsilon,\overline F}:\cC_{\varepsilon,\overline F}\longrightarrow \cE_{\varepsilon,\overline F}
\]
is the finite \'etale quotient of \(\cC_{\varepsilon,\overline F}\) by the cyclic group generated by \(\rho\).  Equivalently, after choosing any one of the three points at infinity of \(\cC_{\varepsilon,\overline F}\) as the origin, \(\pi_{\varepsilon,\overline F}\) is an \'etale \(3\)-isogeny of elliptic curves.
\end{theorem}

\begin{proof}
Write the projective closure as
\[
        XY(X+Y)+\varepsilon Z^3=0\subset\PP^2_F,
\]
so that
\[
        (x_1,x_2,x_3)=\left(\frac XZ,\frac YZ,-\frac{X+Y}{Z}\right)
\]
on the affine chart \(Z\ne0\).  The three points at infinity are
\[
        P_1=[1:0:0],\qquad P_2=[0:1:0],\qquad P_3=[1:-1:0].
\]
The source is smooth.  On the affine chart, simultaneous vanishing of the two partial derivatives gives
\[
        Y(2X+Y)=0,\qquad X(X+2Y)=0,
\]
which forces \(X=Y=0\), because \(3\) is invertible; this is incompatible with \(XY(X+Y)+\varepsilon=0\).  At infinity one has \(Z=0\), but the \(X\)- and \(Y\)-partials do not vanish simultaneously at any of \(P_1,P_2,P_3\).  The target curve is smooth as well: for \(V^2+4U^3+27\varepsilon^2=0\), simultaneous vanishing of the affine partials forces \(U=V=0\), contradicting \(27\varepsilon^2\ne0\), and the usual projective point at infinity is smooth.

The cyclic permutation is induced by
\[
        [X:Y:Z]\longmapsto [Y:-X-Y:Z],
\]
and this map cycles \(P_1,P_2,P_3\).  On the affine chart, a fixed point would have \(x_1=x_2=x_3\); since \(3\) is invertible, the trace-zero equation would force all three coordinates to be zero, contradicting \(x_1x_2x_3=\varepsilon\).  Thus the action is free over \(\overline F\).

The homogeneous numerators of the two invariant functions are
\[
        U_h=-X^2-XY-Y^2,
        \qquad
        V_h=-(X-Y)(X+2Y)(2X+Y).
\]
The minus sign in \(V_h\) is forced by the convention
\[
        V=(x_1-x_2)(x_2-x_3)(x_3-x_1).
\]
Indeed, with \(x_3=-(X+Y)/Z\), the numerator of \(V\) is exactly \(V_h\).  The projective map to the Weierstrass model
\[
        \cE_\varepsilon: \mathsf V^2\mathsf W=-4\mathsf U^3-27\varepsilon^2\mathsf W^3
\]
is
\[
        [X:Y:Z]\longmapsto [\mathsf U:\mathsf V:\mathsf W]
        =[U_hZ:V_h:Z^3].
\]
At the three points at infinity one has \(Z=0\) and \(V_h=\pm2\), so each maps to the point at infinity of \(\cE_\varepsilon\).  Hence the affine discriminant functions extend to a morphism on \(\cC_\varepsilon\).

The identity
\[
        V^2=-4U^3-27\varepsilon^2
\]
is the discriminant formula for the depressed cubic
\[
        T^3+UT-\varepsilon=(T-x_1)(T-x_2)(T-x_3),
\]
and its homogeneous form is
\[
        V_h^2=-4U_h^3-27\varepsilon^2Z^6
\]
on the plane cubic.  Therefore the morphism lands on \(\cE_\varepsilon\).

It remains to identify the quotient.  It is enough to check this after base change to an algebraically closed field \(F\) of characteristic different from \(2\) and \(3\).  Let \(F(\cC_\varepsilon)\) be the function field of the fiber.  Since \(U\) and \(V\) are invariant, we have
\[
        F(U,V)\subseteq F(\cC_\varepsilon)^{\langle\rho\rangle}.
\]
The element \(x_1\) satisfies
\[
        T^3+UT-\varepsilon=0.
\]
Moreover the remaining two coordinates are rational over \(F(U,V,x_1)\): in the
function field,
\[
        x_2+x_3=-x_1,
        \qquad
        x_2-x_3=-\frac{V}{3x_1^2+U},
\]
because \(3x_1^2+U=(x_1-x_2)(x_1-x_3)\) and
\(V=-(3x_1^2+U)(x_2-x_3)\).  The denominator is not the zero rational function
on the generic fiber, since the cubic has nonzero discriminant
\(V^2=-4U^3-27\varepsilon^2\).  Thus
\(F(\cC_\varepsilon)=F(U,V,x_1)\), and the displayed cubic gives
\([F(\cC_\varepsilon):F(U,V)]\le3\).  Conversely, the generic orbit of the free order-three action has size \(3\), and hence
\[
        [F(\cC_\varepsilon):F(\cC_\varepsilon)^{\langle\rho\rangle}]=3.
\]
Thus
\[
        F(\cC_\varepsilon)^{\langle\rho\rangle}=F(U,V),
\]
and the geometric quotient is the displayed Weierstrass curve.  Because the action is free and the group order is prime to the characteristic, this quotient morphism is finite \'etale of degree \(3\) over \(\overline F\).  Both source and target are smooth projective genus-one curves, so after choosing an origin on the source, the morphism is an \'etale \(3\)-isogeny.
\end{proof}

Throughout the remainder of this section \(q\) is an odd prime power with \(\operatorname{char}\F_q\ne3\), and \(\chi\) denotes the quadratic character of \(\F_q\), extended by \(\chi(0)=0\).  For \(\varepsilon\in\F_q^\times\), put
\[
        E_\varepsilon: V^2=-4U^3-27\varepsilon^2,
        \qquad
        C_\varepsilon(q)=\sum_{u\in\F_q}\chi(-4u^3-27\varepsilon^2),
\]
so that \(\#E_\varepsilon(\F_q)=q+1+C_\varepsilon(q)\).

\begin{lemma}[Descent from the ordering torsor]\label{lem:ordering-torsor-descent}
Let \(F\) be a field of characteristic different from \(2\) and \(3\), let
\(\varepsilon\in F^\times\), and let \(B\) be a finite \'etale cubic
\(F\)-algebra.  Let
\[
        \mathscr P_B=\operatorname{Isom}_{\overline F\text{-alg}}
        (B\otimes_F\overline F,\overline F^3)
\]
be the right \(S_3\)-torsor of orderings of the three geometric embeddings of
\(B\).  If \(\overline C_\varepsilon\) denotes the split compactified curve over
\(\overline F\), with \(S_3\) acting by permutation of the three coordinates,
then the compactified trace-zero norm curve
\[
        C_{\varepsilon,B}:\quad \Tr_{B/F}(x)=0,
        \qquad \Norm_{B/F}(x)=\varepsilon
\]
is canonically the twist
\[
        \mathscr P_B\times^{S_3}\overline C_\varepsilon .
\]
Under this identification, the three points at infinity are the corresponding
\(S_3\)-twist of the standard three-point set, and Frobenius acts on them by the
same permutation by which it acts on the ordered geometric embeddings of \(B\).
\end{lemma}

\begin{proof}
After base change to \(\overline F\), any ordering of the three embeddings
identifies \(B\otimes_F\overline F\) with \(\overline F^3\).  Under such an
identification the trace and norm are exactly
\[
        x_1+x_2+x_3,
        \qquad x_1x_2x_3,
\]
so the base change of \(C_{\varepsilon,B}\) is the split projective curve
\(\overline C_\varepsilon\).  Replacing the chosen ordering by another one
composes this identification with the corresponding permutation of the three
coordinates.  Thus the descent datum on the split curve is precisely the descent
datum obtained from the \(S_3\)-torsor \(\mathscr P_B\), which is the definition
of the contracted product twist.  The projective closure and its three points at
infinity are defined by the same homogeneous equations, so the same descent
datum twists the three-point set at infinity.  For a finite field, this says
that Frobenius acts on the points at infinity by the Frobenius permutation of
the geometric embeddings of \(B\).
\end{proof}

\begin{theorem}[\texorpdfstring{\(S_3\)}{S3}-twists and cohomological point counts]
\label{thm:s3-twist-cohomological-count}
Let \(r\) be a prime different from \(\operatorname{char}\F_q\), and let
\(\overline C_\varepsilon\) be the base change of \(\cC_\varepsilon\) to
\(\overline{\F}_q\).  The natural \(S_3\)-action on
\(\overline C_\varepsilon\), obtained by permuting the three coordinates,
acts on
\[
        H^1_{\mathrm{\acute et}}(\overline C_\varepsilon,\Q_r)
\]
through the sign character.  Equivalently, the cyclic subgroup \(A_3\) acts
trivially on \(H^1\), while every transposition acts by \(-1\).

Let \(B\) be a finite \'etale cubic \(\F_q\)-algebra, and let
\(\tau_B\in S_3\) be the Frobenius permutation of the three geometric
embeddings of \(B\), well defined up to conjugacy.  Let
\(C_{\varepsilon,B}\) be the smooth projective compactification of the affine
trace-zero norm curve
\[
        \Tr_{B/\F_q}(x)=0,
        \qquad
        \Norm_{B/\F_q}(x)=\varepsilon .
\]
Put
\[
        A_q(\varepsilon)=\#E_\varepsilon(\F_q)-q-1 .
\]
Then
\[
        \#C_{\varepsilon,B}(\F_q)
        =q+1+\operatorname{sgn}(\tau_B)A_q(\varepsilon).
\]
The number of \(\F_q\)-rational points at infinity on \(C_{\varepsilon,B}\)
is the number of fixed points of \(\tau_B\) on the three coordinate labels.
Consequently the affine trace-zero norm count is
\[
        N_B(\varepsilon)
        =q+1+\operatorname{sgn}(\tau_B)A_q(\varepsilon)
        -\#\operatorname{Fix}_{\{1,2,3\}}(\tau_B).
\]
Thus
\[
\begin{array}{c|c|c|c}
\toprule
B & \tau_B & \#\operatorname{Fix}(\tau_B) & N_B(\varepsilon)\\
\midrule
\F_q^3 & 1 & 3 & \#E_\varepsilon(\F_q)-3\\
\F_q\times\F_{q^2} & \text{a transposition} & 1
& 2q+1-\#E_\varepsilon(\F_q)\\
\F_{q^3} & \text{a three-cycle} & 0 & \#E_\varepsilon(\F_q)\\
\bottomrule
\end{array}
\]
\end{theorem}

\begin{proof}
The cyclic quotient map
\(\pi_\varepsilon:\overline C_\varepsilon\to\overline E_\varepsilon\) of
Theorem~\ref{thm:universal-discriminant-isogeny} is finite \'etale of degree
\(3\).  On \(r\)-adic cohomology, the trace map gives
\[
        (\pi_\varepsilon)_*\pi_\varepsilon^*=3
\]
on \(H^1\).  Since \(3\) is invertible in \(\Q_r\), \(\pi_\varepsilon^*\) is
injective.  Both curves have genus one, so their \(H^1\)-spaces have the same
\(\Q_r\)-dimension \(2\), and \(\pi_\varepsilon^*\) is an isomorphism.
Because \(\pi_\varepsilon\) is invariant under \(A_3=\langle\rho\rangle\),
the subgroup \(A_3\) acts trivially on \(H^1(\overline C_\varepsilon,\Q_r)\).

A transposition fixes
\[
        U=x_1x_2+x_1x_3+x_2x_3
\]
and changes the sign of
\[
        V=(x_1-x_2)(x_2-x_3)(x_3-x_1).
\]
It therefore descends to the involution \((U,V)\mapsto(U,-V)\) on
\(E_\varepsilon\), i.e. multiplication by \(-1\) with respect to the point at
infinity.  This involution acts as \(-1\) on \(H^1\) of an elliptic curve.
Transporting through \(\pi_\varepsilon^*\), every transposition acts by
\(-1\) on \(H^1(\overline C_\varepsilon,\Q_r)\).  Since \(S_3\) is generated
by \(A_3\) and a transposition, the action on \(H^1\) is the sign character.

By Lemma~\ref{lem:ordering-torsor-descent}, the compactified curve
\(C_{\varepsilon,B}\) is the \(S_3\)-twist of the split compactified curve by the
ordering torsor of \(B\).  After base change to \(\overline{\F}_q\), this twist is
identified with \(\overline C_\varepsilon\), while Frobenius acts as ordinary
Frobenius followed, according to the chosen convention, by an element in the
conjugacy class of \(\tau_B\) or \(\tau_B^{-1}\).  The sign and the number of
fixed coordinate labels are unchanged by inversion.  The Grothendieck--Lefschetz
trace formula
\cite[Chapter~VI]{MilneEtale1980} gives
\[
        \#C_{\varepsilon,B}(\F_q)
        =q+1-
        \operatorname{Tr}\bigl(\operatorname{Fr}_q\tau_B
        \mid H^1(\overline C_\varepsilon,\Q_r)\bigr).
\]
Using the sign action and the isomorphism with \(H^1(\overline E_\varepsilon)\),
this becomes
\[
        q+1+\operatorname{sgn}(\tau_B)
        (\#E_\varepsilon(\F_q)-q-1).
\]
Finally, the three points at infinity on the split curve are permuted in the
same way as the coordinate labels.  After twisting, their rational points are
therefore the fixed points of \(\tau_B\).  Subtracting them gives the affine
count.  The three rows are the identity, transposition, and three-cycle cases.
\end{proof}

\begin{remark}[Cohomological refinement of the trace/norm count]
\label{rem:s3-twist-refinement}
Theorem~\ref{thm:s3-twist-cohomological-count} is a structural refinement of
the elementary factorization count below.  It does not assert priority for the
classical prescribed trace/norm problem.  Its role here is to show that the
three cubic \'etale rows are not separate coincidences: they are the three
Frobenius cycle types of a single \(S_3\)-twisted genus-one curve, with the
sign character on \(H^1\) producing the change from \(\#E_\varepsilon\) to
\(2q+1-\#E_\varepsilon\) in the mixed case and the Frobenius-fixed points at
infinity accounting for the affine constants.
\end{remark}

\begin{lemma}[Depressed cubics by factorization type]\label{lem:factorization-type-counts}
Let \(I,S,L,R\) be the numbers of parameters \(u\in\F_q\) for which
\[
        g_u(T)=T^3+uT-\varepsilon
\]
is respectively irreducible, split with three distinct roots, a product of a linear factor and an irreducible quadratic, or ramified.  Then
\[
        3I=\#E_\varepsilon(\F_q),\qquad
        6S+3R=\#E_\varepsilon(\F_q)-3,
\]
and
\[
        2L+R=2q+1-\#E_\varepsilon(\F_q).
\]
\end{lemma}

\begin{proof}
The four factorization types partition the \(q\) values of \(u\):
\begin{equation}\label{eq:factor-partition-v2}
        I+S+L+R=q.
\end{equation}
The equation \(g_u(r)=0\) with \(r\in\F_q\) has \(r\ne0\), and for each \(r\in\F_q^\times\) determines the unique value \(u=(\varepsilon-r^3)/r\).  A split squarefree cubic contributes three rational roots, a linear-quadratic cubic contributes one, and a ramified cubic of this form has one double root and one distinct simple root, both rational.  Hence
\begin{equation}\label{eq:factor-rootcount-v2}
        3S+L+2R=q-1.
\end{equation}
For squarefree cubics, the discriminant
\[
        \Delta(u)=-4u^3-27\varepsilon^2
\]
is a square exactly when the Frobenius permutation of the three roots is even.  Thus the split and irreducible cases contribute \(+1\) to \(\chi(\Delta(u))\), the linear-quadratic case contributes \(-1\), and the ramified cases contribute \(0\).  Therefore
\begin{equation}\label{eq:factor-character-v2}
        I+S-L=C_\varepsilon(q).
\end{equation}
Solving \eqref{eq:factor-partition-v2}, \eqref{eq:factor-rootcount-v2}, and \eqref{eq:factor-character-v2} gives the three displayed identities, using \(\#E_\varepsilon(\F_q)=q+1+C_\varepsilon(q)\).
\end{proof}

The preceding cohomological theorem already gives the table of counts.  We keep the following theorem in the concrete affine form used later, and include the elementary factorization proof as an independent check of the same formulae.

\begin{corollary}[Affine trace-zero norm counts in all cubic \'etale types]\label{thm:all-etale-counts}
Let \(B\) be a finite \'etale cubic \(\F_q\)-algebra, and set
\[
        N_B(\varepsilon)=\#\{x\in B^\times:\Tr_{B/\F_q}(x)=0,
        \Norm_{B/\F_q}(x)=\varepsilon\}.
\]
Then
\[
\begin{array}{c|c}
\toprule
B & N_B(\varepsilon)\\
\midrule
\F_q^3 & \#E_\varepsilon(\F_q)-3\\
\F_q\times\F_{q^2} & 2q+1-\#E_\varepsilon(\F_q)\\
\F_{q^3} & \#E_\varepsilon(\F_q)\\
\bottomrule
\end{array}
\]
In particular, for \(\varepsilon=\pm1\) the inert field case gives
\[
\#\{x\in\F_{q^3}^{\times}:\Tr(x)=0,\Norm(x)=\varepsilon\}=\#E(\F_q),
\qquad E:V^2=-4U^3-27.
\]
\end{corollary}

\begin{proof}
This is the affine part of Theorem~\ref{thm:s3-twist-cohomological-count} in
the three Frobenius cycle types.  In the split case the compactification has
three rational points at infinity; in the mixed case it has one; in the inert
case it has none.  Subtracting those points from the corresponding compact
cohomological counts gives the three displayed affine formulae.  Lemma~\ref{lem:factorization-type-counts}
provides an independent elementary factorization check of the same table.
\end{proof}

\begin{definition}[Prescribed trace/norm count]
\label{def:NB-global}
For every finite field \(F\) of characteristic different from \(2\) and \(3\),
every finite \'etale cubic \(F\)-algebra \(B\), and every \(s\in F\),
\(n\in F^\times\), set
\[
        N_B(s,n)=\#\{x\in B^\times:
        \Tr_{B/F}(x)=s,\ \Norm_{B/F}(x)=n\}.
\]
The smooth formula below applies when \(s^3\ne27n\); the nodal formula in
Proposition~\ref{prop:full-nodal-count} applies when \(s^3=27n\).  We write
\[
        N_B(\varepsilon):=N_B(0,\varepsilon)
\]
for the trace-zero specialization used in the earlier trace-zero count tables
and in later branch-census formulae.

\end{definition}

\begin{theorem}[Prescribed trace/norm \texorpdfstring{\(S_3\)}{S3}-twists]
\label{thm:prescribed-trace-norm-s3}
Let \(F=\F_q\) have characteristic different from \(2\) and \(3\).  Let
\(s\in F\) and \(n\in F^\times\) satisfy
\[
        s^3\ne 27n.
\]
For a finite \'etale cubic \(F\)-algebra \(B\), let
\(C_{s,n,B}\) be the smooth projective compactification of
\[
        \Tr_{B/F}(x)=s,
        \qquad
        \Norm_{B/F}(x)=n .
\]
Put
\[
        E_{s,n}:\quad
        V^2=s^2U^2-4U^3-4s^3n-27n^2+18sUn
\]
and
\[
        A_q(s,n)=\#E_{s,n}(F)-q-1.
\]
If \(\tau_B\in S_3\) is the Frobenius permutation of the three geometric
embeddings of \(B\), then
\[
        \#C_{s,n,B}(F)=q+1+\operatorname{sgn}(\tau_B)A_q(s,n).
\]
The number of affine solutions
\[
        N_B(s,n)=\#\{x\in B^\times:\Tr_{B/F}(x)=s,\ \Norm_{B/F}(x)=n\}
\]
is therefore
\[
        N_B(s,n)=q+1+\operatorname{sgn}(\tau_B)A_q(s,n)
        -\#\operatorname{Fix}_{\{1,2,3\}}(\tau_B).
\]
Equivalently,
\[
\begin{array}{c|c}
\toprule
B & N_B(s,n)\\
\midrule
F^3 & \#E_{s,n}(F)-3\\
F\times F_{q^2} & 2q+1-\#E_{s,n}(F)\\
F_{q^3} & \#E_{s,n}(F)\\
\bottomrule
\end{array}
\]
The trace-zero theorem above is the specialization \(s=0\), \(n=\varepsilon\).
\end{theorem}

\begin{proof}
Over a separable closure write the three coordinates as
\(x_1,x_2,x_3\).  The equations are
\[
        x_1+x_2+x_3=s,
        \qquad x_1x_2x_3=n .
\]
Set
\[
        U=x_1x_2+x_1x_3+x_2x_3,
        \qquad
        V=(x_1-x_2)(x_2-x_3)(x_3-x_1).
\]
Then \(x_1,x_2,x_3\) are the roots of
\[
        T^3-sT^2+UT-n,
\]
and the discriminant identity gives
\[
        V^2=s^2U^2-4U^3-4s^3n-27n^2+18sUn .
\]
The discriminant of the cubic polynomial in \(U\) on the right-hand side is
\[
        16n(s^3-27n)^3,
\]
so the target Weierstrass curve is smooth under the stated hypotheses.
In homogeneous coordinates on the split source, with
\(X_3=sZ-X-Y\), put
\[
        U_h=XY+XX_3+YX_3,
        \qquad
        V_h=(X-Y)(X+2Y-sZ)(sZ-2X-Y).
\]
The affine functions \(U=U_h/Z^2\) and \(V=V_h/Z^3\) therefore extend to the
projective morphism
\[
        [X:Y:Z]\longmapsto[\mathsf U:\mathsf V:\mathsf W]
        =[U_hZ:V_h:Z^3]
\]
to the projective Weierstrass model
\[
        \mathsf V^2\mathsf W=s^2\mathsf U^2\mathsf W-4\mathsf U^3
        -4s^3n\mathsf W^3-27n^2\mathsf W^3+18sn\mathsf U\mathsf W^2.
\]
At each of the three points at infinity on the source, \(Z=0\) and
\(V_h\ne0\), so all three map to the Weierstrass point at infinity.  Thus the
invariant affine functions define the quotient morphism on the smooth
projective compactification, not only on the affine chart.
The projective closure of the split source is
\[
        XY(sZ-X-Y)-nZ^3=0\subset\PP^2.
\]
Its only possible affine singularity would have
\(x_1=x_2=x_3=s/3\), which would force \(n=s^3/27\).  The three points at
infinity are
\[
        [1:0:0],\qquad [0:1:0],\qquad [1:-1:0].
\]
For
\(G(X,Y,Z)=XY(sZ-X-Y)-nZ^3\), one has
\[
        G_X=Y(sZ-2X-Y),\qquad
        G_Y=X(sZ-X-2Y),\qquad
        G_Z=sXY-3nZ^2 .
\]
At these three points, respectively, \(G_Y=-1\), \(G_X=-1\), and
\(G_X=G_Y=1\).  Thus all points at infinity are smooth.  Hence the split source
is a smooth genus-one curve.

The cyclic permutation \((123)\) has no fixed point on the split source: an
affine fixed point would again force \(x_1=x_2=x_3=s/3\), and the points at
infinity are cyclically permuted.  The functions \(U\) and \(V\) are invariant
under the alternating subgroup, so the displayed map factors through the finite
quotient by \(A_3\).  The quotient is identified with \(E_{s,n}\) at the
function-field level, which also handles the points with \(V=0\).  Let \(K\)
be the function field of the split source over an algebraic closure.  The
coordinate \(x_1\) satisfies
\[
        T^3-sT^2+UT-n=0.
\]
The whole function field is generated by this one root over the invariant
field \(\overline F(U,V)\).  Indeed, in \(K\),
\[
        x_2+x_3=s-x_1,
        \qquad
        x_2-x_3=-\frac{V}{3x_1^2-2sx_1+U},
\]
since \(3x_1^2-2sx_1+U=(x_1-x_2)(x_1-x_3)\) and
\(V=-(3x_1^2-2sx_1+U)(x_2-x_3)\).  This denominator is not identically zero on
the generic smooth fiber.  Hence \(K=\overline F(U,V,x_1)\), and the displayed
cubic gives \([K:\overline F(U,V)]\le3\).  Conversely, the generic \(A_3\)-orbit has
three points, and hence \([K:K^{A_3}]=3\).  Since \(U\) and \(V\) are
\(A_3\)-invariant, this gives
\[
        K^{A_3}=\overline F(U,V),
\]
where \(U,V\) satisfy the displayed equation of \(E_{s,n}\).  The quotient
curve and \(E_{s,n}\) are smooth projective curves, and the induced finite
birational morphism between them is therefore an isomorphism.  Thus, over
\(\overline F\), the split \(A_3\)-quotient is the displayed curve
\(E_{s,n}\).  Because the \(A_3\)-action is free and its order is prime to the
characteristic, the quotient map from the split source to \(E_{s,n}\) is finite
\'etale of degree \(3\).  After choosing an origin on the source, this map is an
\'etale \(3\)-isogeny of genus-one curves.  For a nonsplit \(B\), the
ordering-torsor descent retains this geometric quotient identification and lets
Frobenius act on its cohomology through the residual \(S_3/A_3\) action described
next; no untwisted quotient identification over \(F\) is being asserted before
this descent datum is applied.  A transposition fixes \(U\) and sends \(V\) to
\(-V\); it therefore acts on the quotient elliptic curve as the elliptic
involution and hence by \(-1\) on \(H^1\).  The cyclic subgroup acts trivially on
\(H^1\) through the pullback isomorphism induced by the finite \'etale
\(3\)-isogeny.  Thus the \(S_3\)-action on \(H^1\) again factors through the sign
character.

The descent from the ordering torsor of \(B\) is identical to
Lemma~\ref{lem:ordering-torsor-descent}: choosing an ordering of the three
geometric embeddings identifies the curve with the split model, and changing
the ordering composes with the corresponding permutation of the coordinates.
The Grothendieck--Lefschetz trace formula then gives
\[
        \#C_{s,n,B}(F)=q+1+\operatorname{sgn}(\tau_B)A_q(s,n).
\]
The points at infinity are the same three coordinate-label points as before;
after twisting, the rational ones are exactly the fixed labels of \(\tau_B\).
Subtracting them gives the affine formula and the three rows.
\end{proof}

\begin{definition}[Compactified prescribed trace/norm curve in all fibers]\label{def:compactified-prescribed-trace-norm}
Let \(F\) be a field of characteristic different from \(2\) and \(3\), let \(B/F\) be a finite \'etale cubic algebra, let \(s\in F\), and let \(n\in F^\times\).  Over a separable closure and after an ordering of the three embeddings of \(B\), define the split projective cubic as the complete intersection in
\(\PP^3_{\overline F}\) with coordinates \([X_1:X_2:X_3:Z]\):
\[
        X_1+X_2+X_3=sZ,
        \qquad X_1X_2X_3=nZ^3 .
\]
Equivalently, inside the trace plane \(X_1+X_2+X_3=sZ\cong\PP^2\), after writing \(X=X_1\), \(Y=X_2\), and \(X_3=sZ-X-Y\), this is the plane cubic
\[
        XY(sZ-X-Y)-nZ^3=0\subset\PP^2_{[X:Y:Z]}.
\]
The curve \(C_{s,n,B}\) is the \(S_3\)-twist of this projective cubic by the ordering torsor of \(B\).  This definition includes both the smooth case \(s^3\ne27n\) and the nodal case \(s^3=27n\).  Its affine part is the prescribed trace/norm scheme
\[
        \Tr_{B/F}(x)=s,\qquad \Norm_{B/F}(x)=n .
\]
For finite fields, \(N_B(s,n)\) will always denote the actual affine count of this scheme.  When \(s^3\ne27n\) it is evaluated by Theorem~\ref{thm:prescribed-trace-norm-s3}; when \(s^3=27n\) we also write \(N_B^{\rm nod}(s,n)\) for the value given by Proposition~\ref{prop:full-nodal-count} below.  Until that proposition is proved, statements involving \(N_B(s,n)\) are to be read as formal identities in the actual affine count; the closed nodal evaluation is deferred to that proposition.
\end{definition}

\begin{proposition}[Nodal prescribed trace/norm fibers]
\label{prop:nodal-prescribed-trace-norm}
Let \(F\) be a field of characteristic different from \(2\) and \(3\), let
\(B/F\) be a finite \'etale cubic algebra, and let \(s\in F\),
\(n\in F^\times\).  The compactified prescribed trace/norm curve
\[
        C_{s,n,B}:\quad \Tr_{B/F}(x)=s,
        \qquad \Norm_{B/F}(x)=n
\]
is smooth of genus one if \(s^3\ne27n\).  If \(s^3=27n\), then after base
change to a separable closure it is a rational nodal cubic.  Its unique
singular point is the diagonal point
\[
        x_1=x_2=x_3=s/3
\]
in split coordinates.  Consequently the discriminant condition
\(s^3\ne27n\) in Theorem~\ref{thm:prescribed-trace-norm-s3} is exactly the
smoothness condition; the excluded fiber is nodal, not cuspidal or worse.
\end{proposition}

\begin{proof}
The assertion is geometric, so it suffices to work over a separable closure and
write the split projective model as
\[
        XY(sZ-X-Y)-nZ^3=0\subset\PP^2 .
\]
The three points at infinity are
\([1:0:0]\), \([0:1:0]\), and \([1:-1:0]\).  With
\(G(X,Y,Z)=XY(sZ-X-Y)-nZ^3\), one has
\[
        G_X=Y(sZ-2X-Y),\qquad
        G_Y=X(sZ-X-2Y),\qquad
        G_Z=sXY-3nZ^2 .
\]
At these three points, respectively, \(G_Y=-1\), \(G_X=-1\), and
\(G_X=G_Y=1\).  Hence all points at infinity are smooth.  On the affine chart
\(Z=1\), write \(x_1=X\), \(x_2=Y\), and \(x_3=s-X-Y\).  A point of the curve
has
\[
        x_1x_2x_3=n\ne0,
\]
so all three coordinates are nonzero.  The affine singular equations are
\[
        x_2(s-2x_1-x_2)=0,
        \qquad
        x_1(s-x_1-2x_2)=0.
\]
Since \(x_1,x_2\ne0\), these imply \(s-2x_1-x_2=0\) and
\(s-x_1-2x_2=0\), hence \(x_1=x_2\) and then
\[
        x_1=x_2=x_3=s/3 .
\]
This point lies on the norm fiber exactly when \(n=(s/3)^3\), i.e.
\(s^3=27n\).  Thus the curve is smooth when \(s^3\ne27n\), as used in
Theorem~\ref{thm:prescribed-trace-norm-s3}.

Assume now that \(s^3=27n\), and put \(a=s/3\ne0\).  Write
\(x_i=a+y_i\), with \(y_1+y_2+y_3=0\).  In the trace-zero tangent plane the norm
condition becomes
\[
        (a+y_1)(a+y_2)(a+y_3)=a^3,
\]
that is
\[
        a(y_1y_2+y_1y_3+y_2y_3)+y_1y_2y_3=0 .
\]
The quadratic part
\[
        a(y_1y_2+y_1y_3+y_2y_3)
\]
is nondegenerate on the two-dimensional plane \(y_1+y_2+y_3=0\), because
\(a\ne0\) and the characteristic is not \(2\) or \(3\).  Therefore the singularity
is an ordinary double point.  The same explicit equation over an algebraic closure shows that the cubic is geometrically irreducible.  Indeed,
if it had a line component, that line would not be the line at infinity, since
\(G(X,Y,0)=-XY(X+Y)\) is not identically zero on \(Z=0\).  Its affine part would
then be an affine line \((x(t),y(t))\) on which
\[
        x(t)y(t)(s-x(t)-y(t))=n\ne0
\]
identically.  The three factors are linear polynomials in \(t\), and their
product is a nonzero constant; hence each factor is constant, contradicting
that an affine line is one-dimensional.  Thus the nodal cubic is irreducible.
An irreducible plane cubic with one ordinary double point has arithmetic genus
one and geometric genus zero, hence its normalization is \(\PP^1\).  The
discriminant quotient has a node at \((U,V)=(s^2/3,0)\).  For an
arbitrary finite \'etale cubic algebra \(B/F\), the same ordering-torsor descent
as in Lemma~\ref{lem:ordering-torsor-descent} twists the nodal split fiber, its
normalization, the three infinity preimages, and the unordered pair of branches
above the node.  Geometrically, the nodal curve is obtained from this twisted
normalization by identifying the two geometric preimages of the rational node.
No rational point count is being asserted here: over a nonsplit finite-field
form, those two branch preimages may be rational or conjugate.  The finite-field
rational point accounting is carried out later in
Lemma~\ref{lem:rational-node-branches} and
Proposition~\ref{prop:full-nodal-count}.
This proves the proposition.
\end{proof}

\begin{corollary}[Supersingular simplification]\label{cor:supersingular}
If \(q\equiv2\pmod3\), then \(\#E_\varepsilon(\F_q)=q+1\).  Hence
\[
        N_{\F_q^3}(\varepsilon)=q-2,
        \qquad
        N_{\F_q\times\F_{q^2}}(\varepsilon)=q,
        \qquad
        N_{\F_{q^3}}(\varepsilon)=q+1.
\]
\end{corollary}

\begin{proof}
When \(q\equiv2\pmod3\), the cube map on \(\F_q\) is bijective.  Hence \(-4u^3-27\varepsilon^2\) ranges over \(\F_q\) as \(u\) ranges over \(\F_q\), and \(C_\varepsilon(q)=0\).  The result follows from Theorem \ref{thm:all-etale-counts}.
\end{proof}

\begin{example}[The three splitting types at \(q=5\)]\label{ex:three-splitting-types}
Since \(5\equiv2\pmod3\), Corollary~\ref{cor:supersingular} gives, for each \(\varepsilon=\pm1\),
\[
        N_{\F_5^3}(\varepsilon)=3,
        \qquad
        N_{\F_5\times\F_{25}}(\varepsilon)=5,
        \qquad
        N_{\F_{125}}(\varepsilon)=6.
\]
Thus the same trace-zero/norm-fixed problem has visibly different affine counts in the split, mixed, and field cases, even though all three are governed by the same discriminant quotient.  This is the finite-field reason the local theorem is formulated over arbitrary cubic \'etale \(\Z_p\)-algebras rather than only over inert field algebras.
\end{example}

\begin{corollary}[Inert trace-zero norm-sign count]\label{thm:toruscount}
For each \(\eps\in\{\pm1\}\),
\[
\#\{x\in\F_{q^3}^{\times}:\Tr_{\F_{q^3}/\F_q}(x)=0,
        \Norm_{\F_{q^3}/\F_q}(x)=\eps\}=\#E(\F_q),
\]
where \(E:V^2=-4U^3-27\).  Consequently,
\[
\#\{x\in\F_{q^3}^{\times}:\Tr(x)=0,
        \Norm(x)=\pm1\}=2\#E(\F_q).
\]
\end{corollary}

\begin{proof}
This is the inert-field row of Theorem \ref{thm:all-etale-counts}, applied to \(\varepsilon=1\) and \(\varepsilon=-1\).  Since \(\varepsilon^2=1\), both signs are governed by the same curve \(E\).
\end{proof}

\section{The effective local congruence-branch theorem at good primes}\label{sec:general-local-skolem}

Let \(p\ge5\), and let \(A\) be a finite \'etale cubic \(\Z_p\)-algebra.  Thus \(A/pA\) is one of \(\F_{p^3}\), \(\F_p\times\F_{p^2}\), or \(\F_p^3\).  We write
\[
        \Tr_A:A\longrightarrow \Z_p,
        \qquad
        \overline A=A/pA,
\]
for the trace and the reduced \'etale algebra.  Let \(\eta\in A^\times\), let \(\gamma\in A\), and consider
\[
        T_n=\Tr_A(\gamma\eta^n)\qquad(n\in\Z).
\]
Let \(P\) be the order of \(\bar\eta\in\overline A^\times\).  Since \(\overline A\) is a product of finite fields, \(\overline A^\times\) has order prime to \(p\); hence \(p\nmid P\).  Thus the coordinates \(a\bmod P\) and \(t\bmod p^{k-1}\) below are independent modulo \(Pp^{k-1}\).  The local zero space used in this paper is the disjoint branch space
\[
        \mathscr B_p=\Z/P\Z\times\Z_p
        \simeq \varprojlim_k \Z/(Pp^{k-1})\Z,
        \qquad (a,t)\longmapsto a+Pt,
\]
where the inverse limit is taken with the evident reduction maps and the final arrow is compatible at finite level by the Chinese remainder theorem.  Since \(P\) is a unit in \(\Z_p\), the subsets \(a+P\Z_p\) are not disjoint as subsets of ordinary \(\Z_p\); all branch counts below are therefore counts in \(\mathscr B_p\), not counts of a subset of \(\Z_p\) with different branches identified.  We have
\[
        \eta^P=1+pU,
        \qquad U\in A,
\]
and we define the logarithmic tangent
\[
        \omega=\bar U\in\overline A .
\]
Fix once and for all the representative section
\[
        \{0,1,\ldots,P-1\}\longrightarrow \Z.
\]
For each chosen representative \(a\), put
\[
        x_a=\bar\gamma\bar\eta^a\in\overline A,
\]
and define the analytic branch functions
\[
        F_a(t)=\Tr_A\bigl(\gamma\eta^a(\eta^P)^t\bigr),
        \qquad
        F_{a,c}(t)=F_a(t)-c
        \quad(c\in\Z_p),
        \qquad t\in\Z_p .
\]
For \(m\ge0\), set
\[
        C_{a,m}=\Tr_A(\gamma\eta^aU^m).
\]
When later Hensel coefficients such as \(A_a\) and \(B_a\) are attached to a
class \(a\bmod P\), they are computed in this chosen branch coordinate; \(T_a\)
means the value of \(T_n\) at the chosen integer representative.  This is only a
coordinate convention.  Replacing \(a\) by \(a+Pj\) changes the branch parameter
by translation:
\[
        F_{a+Pj}(t)=F_a(t+j),\qquad
        F_{a+Pj,c}(t)=F_{a,c}(t+j).
\]
Thus the zero set and multiplicity statements are invariant under the
replacement, while numerical Hensel coefficients are those of the translated
coordinate.
For integer \(t\), one has \(F_a(t)=T_{a+Pt}\).  For \(k\ge1\), define the local congruence-zero set
\[
        Z_p(k)=\{n\bmod Pp^{k-1}:T_n\equiv0\pmod {p^k}\}.
\]
Equivalently, over a fixed class \(a\bmod P\), membership in \(Z_p(k)\) is determined by the residue of \(t\bmod p^{k-1}\) in \(n=a+Pt\).  Throughout the local sections we use the convention \(v_p(0)=+\infty\).  We use only standard Hensel, Strassmann, and Weierstrass preparation facts for restricted \(p\)-adic power series; see, for example, \cite{Koblitz1984,BoschGuntzerRemmert1984}.

\begin{lemma}[Denominator clearing]\label{lem:denominator-clearing}
Let \(A_{\Q_p}=A\otimes_{\Z_p}\Q_p\), let \(\eta\in A^\times\), and let \(\gamma\in A_{\Q_p}\).  Choose the least integer \(e\ge0\) such that
\[
        \gamma'=p^e\gamma\in A.
\]
Put \(T_n=\Tr_{A_{\Q_p}/\Q_p}(\gamma\eta^n)\) and \(T'_n=\Tr_A(\gamma'\eta^n)\).  Then
\[
        T'_n=p^eT_n,
        \qquad
        v_p(T_n)=v_p(T'_n)-e
\]
for every \(n\) with \(T_n\ne0\), and the same identity holds for the analytic branch functions.  Thus, for \(k\in\Z\), the condition \(v_p(T_n)\ge k\) is equivalent to \(v_p(T'_n)\ge k+e\).  In particular, homogeneous local congruence and valuation questions with \(\gamma\in A_{\Q_p}\) reduce to the integral case \(\gamma'\in A\), at the cost of shifting the required precision by \(e\).

For an affine target \(c\in\Q_p\), the target must be scaled at the same time:
\[
        \Tr_{A_{\Q_p}/\Q_p}(\gamma\eta^n)=c
        \quad\Longleftrightarrow\quad
        \Tr_A(\gamma'\eta^n)=p^e c.
\]
Equivalently, for branch functions,
\[
        \Tr_A\bigl(\gamma'\eta^a(\eta^P)^t\bigr)-p^e c
        =p^e\left(\Tr_{A_{\Q_p}/\Q_p}\bigl(\gamma\eta^a(\eta^P)^t\bigr)-c\right).
\]
Thus affine valuation inequalities are shifted by \(e\) only after replacing
\(c\) by \(p^e c\).  This least denominator-clearing exponent for
\(\gamma\) does not by itself guarantee that the scaled affine target is
integral.  To apply the integral affine branch theorem, one should choose an
exponent
\[
        e_{\rm aff}\ge e
        \quad\text{and}\quad
        e_{\rm aff}\ge -v_p(c)\quad(c\ne0).
\]
Equivalently, with the convention \(-v_p(0):=-\infty\), this condition may be
written as
\[
        e_{\rm aff}\ge \max\{e,\,-v_p(c)\}.
\]
Then replace the pair \((\gamma,c)\) by
\((p^{e_{\rm aff}}\gamma,p^{e_{\rm aff}}c)\); if the least exponent
\(e\) is used, the affine reduction enters the integral theory only when
\(p^e c\in\Z_p\).
\end{lemma}

\begin{proof}
The trace is \(\Q_p\)-linear, and \(\eta^n\in A^\times\) for every integer \(n\).  Multiplication of the coefficient by \(p^e\) therefore multiplies every homogeneous recurrence value, and every homogeneous value of every branch function, by \(p^e\).  Taking valuations gives the homogeneous assertion.  Subtracting an affine target after the same scaling gives
\[
        \Tr_A(\gamma'\eta^n)-p^ec
        =p^e\bigl(\Tr_{A_{\Q_p}/\Q_p}(\gamma\eta^n)-c\bigr),
\]
and the branch identity is identical with \(\eta^n\) replaced by
\(\eta^a(\eta^P)^t\).
\end{proof}

\begin{lemma}[Primitive reduction]\label{lem:primitive-reduction}
Let \(s\in\Z_{\ge0}\cup\{\infty\}\) be the largest integer for which \(\gamma\in p^sA\), with \(s=\infty\) only when \(\gamma=0\).  If \(s<\infty\), write \(\gamma=p^s\gamma_0\) with \(\gamma_0\notin pA\), and put
\[
        T_n^{(0)}=\Tr_A(\gamma_0\eta^n),
        \qquad
        F_a^{(0)}(t)=\Tr_A(\gamma_0\eta^a(\eta^P)^t).
\]
Then
\[
        T_n=p^sT_n^{(0)},\qquad F_a(t)=p^sF_a^{(0)}(t).
\]
Consequently, for the homogeneous congruence \(T_n\equiv0\pmod {p^k}\), if
\(k>s\) the congruence is equivalent to
\(T_n^{(0)}\equiv0\pmod {p^{k-s}}\), while if \(k\le s\) every residue class
is a solution.  Thus all nontrivial homogeneous branch questions reduce to the
primitive case \(\gamma\notin pA\).  If \(\gamma=0\), all homogeneous branch
functions are identically zero.

For an affine target \(c\in\Z_p\), the corresponding branch is
\[
        F_{a,c}(t)=p^sF_a^{(0)}(t)-c.
\]
For congruences \(\Tr_A(\gamma\eta^n)\equiv c\pmod {p^k}\), the following
cases are exhaustive when \(s<\infty\):
\begin{enumerate}[label=\textup{(\alph*)}]
\item if \(k\le s\), every residue class is a solution when
\(c\equiv0\pmod {p^k}\), and no residue class is a solution otherwise;
\item if \(k>s\) and \(c\notin p^s\Z_p\), no residue class is a solution;
\item if \(k>s\) and \(c=p^sc_0\), the congruence is equivalent to
\[
        T_n^{(0)}\equiv c_0\pmod {p^{k-s}}.
\]
\end{enumerate}
For exact affine equations, the same reduction says that
\(p^sT_n^{(0)}=c\) has no solution unless \(c\in p^s\Z_p\), and, when
\(c=p^sc_0\), is equivalent to \(T_n^{(0)}=c_0\).  If \(\gamma=0\), the affine
congruence has all residue classes exactly when \(c\equiv0\pmod {p^k}\), and
none otherwise; the exact affine equation has every branch parameter as a solution when
\(c=0\), and no solution when \(c\ne0\).
\end{lemma}

\begin{proof}
Since \(A\) is finite free over \(\Z_p\), multiplication by \(p^s\) commutes with the trace and with the analytic branch construction.  This proves
\(T_n=p^sT_n^{(0)}\) and \(F_a=p^sF_a^{(0)}\).  The homogeneous congruence
statements are immediate.  For an affine congruence, the left side
\(p^sT_n^{(0)}\) is always divisible by \(p^s\).  If \(k\le s\), it is zero
modulo \(p^k\), giving the all-or-none criterion according to
\(c\bmod p^k\).  If \(k>s\), divisibility of \(c\) by \(p^s\) is necessary;
when \(c=p^sc_0\), division by \(p^s\) gives the congruence modulo
\(p^{k-s}\).  The exact equation follows by the same division when
\(s<\infty\), and for \(\gamma=0\) it is simply the equation \(0=c\).
\end{proof}

\begin{theorem}[Uniform local congruence-branch theorem for cubic \'etale trace recurrences]\label{thm:general-local-skolem}
With the notation above, the zero classes modulo \(p\) in one period \(P\) are
\[
        Z_p(1)=\{a\bmod P:\Tr_{\overline A/\F_p}(x_a)=0\}.
\]
For every \(a\in Z_p(1)\), the function \(F_a\) is a restricted \(p\)-adic analytic function, the quotient \(F_a(X)/p\) lies in
\(\Z_p\langle X\rangle\), and it satisfies the first-order expansion
\begin{equation}\label{eq:general-first-order}
        \frac{F_a(t)}p\equiv
        \frac{T_a}p+t\Tr_{\overline A/\F_p}(x_a\omega)
        \pmod p .
\end{equation}
Consequently:
\begin{enumerate}[label=\textup{(\roman*)}]
\item If
\[
        d_a:=\Tr_{\overline A/\F_p}(x_a\omega)\ne0,
\]
then there is a unique \(\tau_a\in\Z_p\) with \(F_a(\tau_a)=0\), and
\begin{equation}\label{eq:general-valuation}
        v_p(F_a(t))=1+v_p(t-\tau_a)\qquad(t\in\Z_p).
\end{equation}
\item If \(d_a=0\), then the class \(a\) lifts from modulo \(p\) to modulo \(p^2\) if and only if \(T_a/p\equiv0\pmod p\).
\item For every \(k\ge1\), the complete set of classes above \(a\) modulo \(p^k\) is obtained by the finite recursion
\begin{align}
        R_a(1)&=\{0\}\subset\Z/p^0\Z,\notag\\
        R_a(k+1)&=\{r+p^{k-1}j\bmod p^k:
        r\in R_a(k),\ j\in\{0,\ldots,p-1\},\notag\\
        &\hspace{38mm}F_a(r+p^{k-1}j)\equiv0\pmod {p^{k+1}}\}.       \label{eq:newton-hensel-recursion}
\end{align}
Thus the local branch-zero set in \(\mathscr B_p\) is effectively determined to any prescribed precision by exact arithmetic in the corresponding quotients of \(A\); the step producing lifts modulo \(p^{k+1}\) uses arithmetic in \(A/p^{k+1}A\).  This statement is independent of the splitting type of \(p\) and does not require \(\bar\eta\) to generate a full norm torus.
\end{enumerate}
\end{theorem}

\begin{proof}
Since \(\eta^P\in1+pA\), the binomial expansion
\[
        (\eta^P)^t=\sum_{m\ge0}\binom{t}{m}(\eta^P-1)^m
\]
converges in \(A\) for every \(t\in\Z_p\).  More precisely, if \(\eta^P=1+pU\), then the \(m\)-th term is divisible by \(p^m/m!\), and \(v_p(p^m/m!)\to\infty\); equivalently one may write \(\exp(t\log(1+pU))\).  The binomial identity agrees with ordinary powers for nonnegative integers, and for negative integers by the convergent expansion of \((1+pU)^{-m}\).  Hence \(F_a\) is a restricted analytic function.  The modulus used in the definition of \(Z_p(k)\) is compatible with this interpolation: for every \(k\ge1\),
\[
        (\eta^P)^{p^{k-1}}=(1+pU)^{p^{k-1}}\equiv1\pmod {p^k},
\]
with the case \(k=1\) understood modulo \(p\).  Thus \(F_a(t+p^{k-1})\equiv F_a(t)\pmod {p^k}\), and the parameter \(t\) is naturally taken modulo \(p^{k-1}\) at precision \(p^k\).  Modulo \(p^2\),
\[
        (\eta^P)^t\equiv1+tpU\pmod {p^2},
\]
and taking traces gives \eqref{eq:general-first-order}.  For
\(a\in Z_p(1)\), \(T_a\in p\Z_p\).  Put
\(C_{a,m}=\Tr_A(\gamma\eta^aU^m)\).  In \(\Q_p\langle X\rangle\),
\[
        \frac{F_a(X)}p=\frac{T_a}p+
        \sum_{m\ge1}p^{m-1}\binom Xm C_{a,m}.
\]
This quotient is coefficientwise integral and restricted.  Indeed, every
coefficient of the \(m\)-th summand has valuation at least
\[
        m-1-v_p(m!)\ge0,
\]
with equality possible at \(m=1\); for \(m\ge2\), the inequality follows from
\(v_p(m!)\le(m-1)/(p-1)<m-1\), since \(p\ge5\).  The lower bounds tend to
\(+\infty\), so \(F_a(X)/p\in\Z_p\langle X\rangle\).  After division by \(p\),
all terms with \(m\ge2\) vanish modulo \(p\), giving the displayed first-order
restricted-series reduction.  If \(d_a\ne0\), Hensel's lemma applied to the
restricted series \(F_a(X)/p\) gives a unique zero \(\tau_a\).  Division by
\(X-\tau_a\) in \(\Z_p\langle X\rangle\) gives a unit quotient, which proves
\eqref{eq:general-valuation}.  If \(d_a=0\), the linear congruence
\eqref{eq:general-first-order} is independent of \(t\), giving the stated
modulo \(p^2\) criterion.
The recursion is the ordinary digit-by-digit lifting criterion applied to \(F_a\).
\end{proof}

\begin{theorem}[Affine trace-target branch theorem]
\label{thm:affine-local-branch}
Let \(c\in\Z_p\).  With the notation of
Theorem~\ref{thm:general-local-skolem}, define
\[
        F_{a,c}(t)=\Tr_A\bigl(\gamma\eta^a(\eta^P)^t\bigr)-c,
        \qquad t\in\Z_p,
\]
and
\[
        Z_{p,c}(k)=\{n\bmod Pp^{k-1}:
        \Tr_A(\gamma\eta^n)\equiv c\pmod {p^k}\}.
\]
Then the target classes modulo \(p\) in one period are
\[
        Z_{p,c}(1)=\{a\bmod P:
        \Tr_{\overline A/\F_p}(x_a)=\bar c\}.
\]
For every \(a\in Z_{p,c}(1)\), the quotient \(F_{a,c}(X)/p\) lies in \(\Z_p\langle X\rangle\), and the affine branch satisfies
\[
        \frac{F_{a,c}(t)}p\equiv
        \frac{\Tr_A(\gamma\eta^a)-c}{p}
        +t\Tr_{\overline A/\F_p}(x_a\omega)
        \pmod p .
\]
Consequently, if
\[
        d_a=\Tr_{\overline A/\F_p}(x_a\omega)\ne0,
\]
then there is a unique \(\tau_{a,c}\in\Z_p\) with
\(F_{a,c}(\tau_{a,c})=0\), and
\[
        v_p(F_{a,c}(t))=1+v_p(t-\tau_{a,c})
        \qquad(t\in\Z_p).
\]
If \(d_a=0\), then the class lifts from modulo \(p\) to modulo \(p^2\) if and
only if
\[
        \frac{\Tr_A(\gamma\eta^a)-c}{p}\equiv0\pmod p.
\]
All higher lifts are obtained by the digit recursion of
Theorem~\ref{thm:general-local-skolem} with \(F_a\) replaced by \(F_{a,c}\).
\end{theorem}

\begin{proof}
The proof is the proof of Theorem~\ref{thm:general-local-skolem} after
subtracting the constant \(c\), with the coefficientwise integrality made
explicit.  For \(a\in Z_{p,c}(1)\), the constant term
\(\Tr_A(\gamma\eta^a)-c\) lies in \(p\Z_p\), and
\[
        \frac{F_{a,c}(X)}p=\frac{\Tr_A(\gamma\eta^a)-c}{p}+
        \sum_{m\ge1}p^{m-1}\binom Xm
        \Tr_A(\gamma\eta^aU^m).
\]
The same estimate \(m-1-v_p(m!)\ge0\) for \(m\ge1\) shows that
\(F_{a,c}(X)/p\in\Z_p\langle X\rangle\), and after division by \(p\) only the
constant and \(m=1\) terms survive modulo \(p\).  The derivative is unchanged,
and the constant term in the first-order expansion is shifted from \(T_a/p\) to
\((\Tr_A(\gamma\eta^a)-c)/p\).  The compatibility of the parameter modulo
\(p^{k-1}\), Hensel lifting in the transverse case, and the digit recursion are
therefore identical.
\end{proof}

\begin{corollary}[Scalar tangents are transverse for nonzero affine targets]
\label{cor:scalar-tangent-affine-transverse}
In the situation of Theorem~\ref{thm:affine-local-branch}, suppose the reduced
logarithmic tangent is scalar:
\[
        \omega=\lambda\cdot 1\in\overline A,
        \qquad \lambda\in\F_p.
\]
Let \(a\in Z_{p,c}(1)\), and write \(s=\bar c\).  Then
\[
        \Tr_{\overline A/\F_p}(x_a\omega)=\lambda s.
\]
Consequently, if \(\lambda s\ne0\), every affine target class above the target
\(s\) is transverse and lifts to a unique simple branch.  In particular, scalar
nonzero tangents are obstructive only for the homogeneous target \(s=0\); for
nonzero affine targets they force transversality.
\end{corollary}

\begin{proof}
For \(a\in Z_{p,c}(1)\), one has \(\Tr(x_a)=s\).  If
\(\omega=\lambda\cdot1\), then
\[
        \Tr(x_a\omega)=\lambda\Tr(x_a)=\lambda s .
\]
The final assertion is the transverse case of
Theorem~\ref{thm:affine-local-branch}.
\end{proof}

\begin{proposition}[Orbit-preimage form of the affine mod-\(p\) target set]
\label{prop:affine-orbit-intersection}
Let \(H=\langle\bar\eta\rangle\subset\overline A^\times\).  The affine target
classes modulo \(p\) are identified with
\[
        \{h\in H:\Tr_{\overline A/\F_p}(\bar\gamma h)=\bar c\}.
\]
If \(\bar\gamma\in\overline A^\times\), multiplication by \(\bar\gamma\) gives
a bijection with
\[
        \bar\gamma H\cap\{x\in\overline A:\Tr(x)=\bar c\}.
\]
When \(H\) is a full union of norm fibers and \(\bar\gamma\) is a unit, the
number of affine target classes is
\[
        \sum_{\delta\in C} N_{\overline A}
        \bigl(\bar c,\Norm(\bar\gamma)\delta\bigr),
\]
where \(C\subset\F_p^\times\) is the set of norms occurring in \(H\), and
\(N_B(s,n)\) denotes the actual affine prescribed trace/norm count of
Definition~\ref{def:NB-global}.
\end{proposition}

\begin{proof}
The map \(a\bmod P\mapsto\bar\eta^a\) identifies period classes with \(H\), and
the congruence
\(\Tr_A(\gamma\eta^a)\equiv c\pmod p\) is exactly
\(\Tr(\bar\gamma\bar\eta^a)=\bar c\).  If \(\bar\gamma\) is a unit, multiplication
by \(\bar\gamma\) is a bijection on \(\overline A^\times\) and sends the norm
fiber \(\Norm(h)=\delta\) to the norm fiber
\(\Norm(x)=\Norm(\bar\gamma)\delta\).  Decomposing the image by norm gives the
sum of prescribed trace/norm counts.
\end{proof}

\begin{proposition}[Orbit-preimage form of the mod-\(p\) zero set]\label{prop:orbit-intersection}
With the notation of Theorem~\ref{thm:general-local-skolem}, let
\[
        H=\langle\bar\eta\rangle\subset\overline A^{\times},
        \qquad
        m_{\bar\gamma}:H\longrightarrow \overline A,
        \quad h\longmapsto \bar\gamma h .
\]
Then the zero classes modulo \(p\) in one period are canonically identified with the preimage
\[
        m_{\bar\gamma}^{-1}
        \bigl(\ker(\Tr_{\overline A/\F_p})\bigr)
        =\{h\in H:
        \Tr_{\overline A/\F_p}(\bar\gamma h)=0\}.
\]
Equivalently,
\[
        \#Z_p(1)=\#\{h\in H:
        \Tr_{\overline A/\F_p}(\bar\gamma h)=0\}.
\]
If multiplication by \(\bar\gamma\) is injective on \(H\)---in particular, if \(\bar\gamma\in\overline A^\times\)---then this preimage is also canonically bijective with
\[
        \bar\gamma H\cap\ker\bigl(\Tr_{\overline A/\F_p}\bigr).
\]
If \(\bar\gamma=0\), every class modulo \(P\) is a zero class modulo \(p\).  Thus the maximal-order assumptions used later are not part of the local theory; they are only a way to replace the subgroup preimage by a closed trace/norm count in special full-orbit cases.
\end{proposition}

\begin{proof}
The map \(a\bmod P\mapsto\bar\eta^a\) is a bijection from \(\Z/P\Z\) onto \(H\).  The congruence \(T_a\equiv0\pmod p\) is exactly
\[
        \Tr_{\overline A/\F_p}(\bar\gamma\bar\eta^a)=0.
\]
This proves the preimage description and the count.  If \(m_{\bar\gamma}\) is injective on \(H\), then restricting \(m_{\bar\gamma}\) to the displayed preimage gives a bijection onto \(\bar\gamma H\cap\ker(\Tr)\).  If \(\bar\gamma=0\), the displayed trace is identically zero.
\end{proof}

\begin{corollary}[Unit-coefficient full-orbit counts]\label{cor:unit-coefficient-full-orbit-counts}
Let \(B\) be a finite \'etale cubic algebra over \(\F_p\), and put
\[
        N_B(\varepsilon)=\#\{x\in B^\times:\Tr_{B/\F_p}(x)=0,
        \Norm_{B/\F_p}(x)=\varepsilon\}.
\]
Let \(\gamma_0\in B^\times\), and let \(C\subset\F_p^\times\) be nonempty.  Suppose a subset \(H\subset B^\times\) is the full union of norm fibers
\[
        H=\{h\in B^\times:\Norm_{B/\F_p}(h)\in C\}.
\]
Then
\[
\#\{h\in H:\Tr_{B/\F_p}(\gamma_0h)=0\}
        =\sum_{\varepsilon\in C}N_B\bigl(\Norm(\gamma_0)\varepsilon\bigr).
\]
In particular, if \(B=\F_{p^3}\) and the set under consideration is the full norm-\(\varepsilon\) fiber (a coset of the norm-one torus, not generally a subgroup unless \(\varepsilon=1\)), this number is
\[
        \#E_{\Norm(\gamma_0)\varepsilon}(\F_p),
        \qquad
        E_\delta:V^2=-4U^3-27\delta^2.
\]
Still assuming \(B=\F_{p^3}\), if \(H\) is the union of the two norm-sign fibers
\(\Norm(h)=\pm1\), then the number is
\[
        \#E_{\Norm(\gamma_0)}(\F_p)+\#E_{-\Norm(\gamma_0)}(\F_p)
        =2\#E_{\Norm(\gamma_0)}(\F_p).
\]
Thus, in full-orbit situations with a unit coefficient, the mod-\(p\) zero count depends on the coefficient only through its norm.
\end{corollary}

\begin{proof}
Multiplication by \(\gamma_0\) is a bijection from \(H\) onto
\[
        \{x\in B^\times:\Norm(x)\in \Norm(\gamma_0)C\}.
\]
Under this bijection, the condition \(\Tr(\gamma_0h)=0\) becomes \(\Tr(x)=0\).  Decomposing the image according to its norm gives the displayed sum.  The inert-field formulas are the corresponding row of Theorem~\ref{thm:all-etale-counts}.  Since \(E_\delta\) and \(E_{-\delta}\) have the same Weierstrass equation, the two norm signs give equal counts.
\end{proof}

\begin{example}[Why the image formulation needs injectivity]\label{ex:nonunit-coefficient-caveat}
Let \(p=5\), \(\overline A=\F_5^3\), \(H=\langle(1,1,2)\rangle\), and \(\bar\gamma=(1,-1,0)\).  Then \(H\) has order \(4\), while
\[
        \bar\gamma(1,1,2)^j=(1,-1,0)
        \qquad(0\le j<4).
\]
The trace of \((1,-1,0)\) is zero, so all four elements of \(H\) lie in the preimage of the trace hyperplane.  However
\[
        \bar\gamma H\cap\ker(\Tr)=\{(1,-1,0)\}
\]
has only one element.  Thus the preimage statement in Proposition~\ref{prop:orbit-intersection} is the correct general statement; a bijection with the image intersection requires an injectivity hypothesis such as \(\bar\gamma\in\overline A^\times\).
\end{example}

\begin{corollary}[Stability in the all-transverse case]\label{cor:all-transverse-stability}
Assume the hypotheses of Theorem~\ref{thm:general-local-skolem}.  If every class \(a\in Z_p(1)\) is transverse, i.e.
\[
        \Tr_{\overline A/\F_p}(x_a\omega)\ne0
        \qquad(a\in Z_p(1)),
\]
then for every \(k\ge1\)
\[
        Z_p(k)=\{a+Pt:
        a\in Z_p(1),
        \ t\equiv\tau_a\pmod {p^{k-1}}\}
        \pmod {Pp^{k-1}},
\]
where \(\tau_a\) is the unique zero of \(F_a\) from Theorem~\ref{thm:general-local-skolem}.  In particular,
\[
        \#Z_p(k)=\#Z_p(1)
        \qquad(k\ge1).
\]
\end{corollary}

\begin{proof}
For a transverse class, Theorem~\ref{thm:general-local-skolem} gives
\[
        v_p(F_a(t))=1+v_p(t-\tau_a).
\]
Therefore \(F_a(t)\equiv0\pmod {p^k}\) if and only if \(t\equiv\tau_a\pmod {p^{k-1}}\).  Classes outside \(Z_p(1)\) do not contribute even modulo \(p\), so the displayed union is complete and disjoint modulo \(Pp^{k-1}\).
\end{proof}

\begin{theorem}[Finite jet stratification]\label{thm:finite-jet-stratification}
The following statement is valid for a finite \'etale \(\Z_p\)-algebra \(A\) of arbitrary rank.  Let \(p\ge5\), let \(\eta\in A^\times\), let \(\gamma\in A\), let \(P\) be the order of \(\bar\eta\in(A/pA)^\times\), write \(\eta^P=1+pU\), set \(\omega=\bar U\in A/pA\), and define
\[
        F_a(t)=\Tr_A\bigl(\gamma\eta^a(\eta^P)^t\bigr).
\]
Put
\[
        C_{a,m}=\Tr_A(\gamma\eta^aU^m)\qquad(m\ge0).
\]
The statement is first written in the initial branch coordinate \(t\).  Shifted
residue disks \(t=t_0+p^jY\) are handled by the separate shifted-disk form in
Lemma~\ref{lem:shifted-finite-jet-stratification}; the logarithmic tangent scale
and the coefficients then change from \(U\) to \(((\eta^P)^{p^j}-1)/p^{j+1}\).
Fix an integer \(r\) with \(1\le r<p\).  Suppose that
\[
        C_{a,m}\in p^{r-m}\Z_p\quad(0\le m<r),
        \qquad C_{a,r}\not\equiv0\pmod p .
\]
Then, as integer-valued functions of \(t\in\Z_p\) and not as a
coefficientwise congruence in \(\Z_p\langle X\rangle\),
\[
        \frac{F_a(t)}{p^r}\equiv Q_{a,r}(t)\pmod p,
\]
where
\[
        Q_{a,r}(X)=
        \sum_{m=0}^{r-1}\binom{X}{m}\left(\frac{C_{a,m}}{p^{r-m}}\bmod p\right)
        +\binom{X}{r}(C_{a,r}\bmod p)
        \in\F_p[X].
\]
The restriction \(r<p\) ensures that every binomial polynomial \(\binom{X}{m}\) appearing in \(Q_{a,r}\) is integral modulo \(p\), so its value modulo \(p\) depends only on \(X\bmod p\).  Consequently, in the initial coordinate, a residue \(\rho\in\F_p\) occurs as \(t\bmod p\) for a class satisfying \(F_a(t)\equiv0\pmod {p^{r+1}}\) if and only if \(Q_{a,r}(\rho)=0\).  Equivalently, the surviving residue disks at this step are exactly \(t\in\tilde\rho+p\Z_p\), where \(\rho\) runs through the roots of \(Q_{a,r}\) and \(\tilde\rho\in\{0,\ldots,p-1\}\) is its least representative.

No Weierstrass factorization is asserted from root multiplicities of \(Q_{a,r}\) alone.  After a root \(\rho\in\F_p\) is chosen and lifted to \(\tilde\rho\in\{0,\ldots,p-1\}\), higher lifting in the disk
\[
        t=\tilde\rho+pY,\qquad Y\in\Z_p,
\]
is governed by the exact recursion of Theorem \ref{thm:general-local-skolem}.  Any later use of Weierstrass preparation must be justified from the actual reduced restricted power series in the chosen disk, not merely from the multiplicity of \(\rho\) as a root of \(Q_{a,r}\).  This distinction is harmless in the cubic quadratic theorem below, where the shifted reduction is checked directly.

If \(\overline A\) has rank \(d\) as an \(\F_p\)-vector space, if
\(1,\omega,\ldots,\omega^{d-1}\) is an \(\F_p\)-basis, and if
\(x_a:=\overline{\gamma\eta^a}\ne0\), then \(x_a\) cannot be orthogonal to all
of these powers.  Equivalently,
\[
        \Tr_{\overline A/\F_p}(x_a\omega^m)\ne0
\]
for some \(0\le m\le d-1\).
\end{theorem}

\begin{proof}
The binomial expansion gives
\[
        F_a(t)=\sum_{m\ge0}\binom{t}{m}p^m C_{a,m}.
\]
Modulo \(p^{r+1}\), as functions on \(\Z_p\), the terms with \(m>r\) vanish because they contain \(p^m\) and \(\binom tm\in\Z_p\) for every \(t\in\Z_p\).  The hypotheses make the terms with \(m<r\) divisible by \(p^r\).  Dividing by \(p^r\) gives the displayed polynomial \(Q_{a,r}\) modulo \(p\).  Thus this congruence is a congruence of integer-valued functions on \(\Z_p\), not a coefficientwise assertion in \(\Z_p[X]\) for all \(m\).  Therefore, in the displayed coordinate, the next residues modulo \(p\) that lift one further precision are exactly the roots of \(Q_{a,r}\); each such residue specifies the corresponding disk for the subsequent recursion.

The warning about Weierstrass preparation is included because the polynomial \(Q_{a,r}\) controls the values modulo \(p\) at the next digit, not the full reduced restricted power series on a residue disk.  When the latter reduction is directly of the form \(cY^e\), the Weierstrass preparation theorem for restricted power series applies; see, for example, \cite[Section 5.2]{BoschGuntzerRemmert1984}.  The final assertion follows from nondegeneracy of the trace pairing on the finite \'etale \(\F_p\)-algebra \(\overline A\): the nonzero element \(x_a\) cannot be orthogonal to the displayed \(\F_p\)-basis
\(1,\omega,\ldots,\omega^{d-1}\).
\end{proof}

\begin{lemma}[Shifted finite-jet stratification]
\label{lem:shifted-finite-jet-stratification}
Keep the notation of Theorem~\ref{thm:finite-jet-stratification}.  Fix
\(t_0\in\Z_p\) and \(j\ge0\), and write
\[
        (\eta^P)^{p^j}=1+p^{j+1}U_j,
        \qquad
        U_j=\frac{(\eta^P)^{p^j}-1}{p^{j+1}}\in A .
\]
Then \(\bar U_j=\omega\).  Put
\[
        y_{a,t_0}=\gamma\eta^a(\eta^P)^{t_0},
        \qquad
        C^{(j,t_0)}_{a,m}=\Tr_A(y_{a,t_0}U_j^m).
\]
Let \(R\ge1\), set \(M=\lfloor R/(j+1)\rfloor\), and assume \(M<p\).  Suppose
\[
        C^{(j,t_0)}_{a,m}\in p^{R-(j+1)m}\Z_p
        \qquad(0\le m\le M).
\]
Then, as functions of \(Y\in\Z_p\),
\[
        \frac{F_a(t_0+p^jY)}{p^R}\equiv
        Q^{(j,t_0)}_{a,R}(Y)\pmod p,
\]
where
\[
        Q^{(j,t_0)}_{a,R}(Y)=
        \sum_{m=0}^{M}\binom Ym
        \left(\frac{C^{(j,t_0)}_{a,m}}{p^{R-(j+1)m}}\bmod p\right)
        \in\F_p[Y].
\]
Consequently, within the residue disk \(t=t_0+p^jY\), the residues
\(Y\bmod p\) that lift the congruence from precision \(p^R\) to precision
\(p^{R+1}\) are exactly the roots of \(Q^{(j,t_0)}_{a,R}\).  If this polynomial
is identically zero, every value of \(Y\bmod p\) survives one more digit.
\end{lemma}

\begin{proof}
Since \(\eta^P=1+pU\), the binomial expansion gives
\[
        (\eta^P)^{p^j}=(1+pU)^{p^j}=1+p^{j+1}U+O(p^{j+2}),
\]
so \(U_j\equiv U\pmod p\), and therefore \(\bar U_j=\omega\).
The identity
\[
        (\eta^P)^{t_0+p^jY}=(\eta^P)^{t_0}\bigl((\eta^P)^{p^j}\bigr)^Y
\]
and the definition of \(U_j\) give the expansion
\[
        F_a(t_0+p^jY)=
        \sum_{m\ge0}\binom Ym p^{(j+1)m}C^{(j,t_0)}_{a,m}.
\]
For \(m>M\), one has \((j+1)m\ge R+1\), so those terms vanish modulo
\(p^{R+1}\).  The displayed divisibility hypotheses make the terms with
\(0\le m\le M\) divisible by \(p^R\).  Dividing by \(p^R\) and reducing modulo
\(p\) gives the stated polynomial.  Since \(M<p\), the binomial polynomials
\(\binom Ym\) are integral modulo \(p\) for all terms that occur.  The root
criterion is the resulting congruence modulo \(p\).
\end{proof}

\begin{theorem}[Rank-\(d\) Weierstrass bound for primitive toric branches]
\label{thm:rank-d-weierstrass-bound}
Let \(A\) be a finite \'etale \(\Z_p\)-algebra of rank \(d\), with \(p>d\).  Let \(\eta\in A^\times\), let \(\gamma\in A\), let \(P\) be the order of \(\bar\eta\in(A/pA)^\times\), and write
\[
        \eta^P=1+pU,\qquad \omega=\bar U\in A/pA .
\]
Fix a residue class \(a\bmod P\), put
\[
        y_a=\gamma\eta^a,\qquad x_a=\bar y_a,
\]
and assume \(x_a\ne0\).  Suppose that
\[
        1,\omega,\omega^2,\ldots,\omega^{d-1}
\]
is an \(\F_p\)-basis of \(A/pA\).  Define
\[
        C_{a,m}=\Tr_A(y_aU^m)\qquad(m\ge0),
\]
and
\[
        F_a(t)=\Tr_A\bigl(y_a(\eta^P)^t\bigr),\qquad t\in\Z_p .
\]
Then
\[
        s_a:=\min_{0\le m\le d-1}\bigl(m+v_p(C_{a,m})\bigr)
\]
is finite and satisfies \(0\le s_a\le d-1\).  Moreover
\[
        \mathcal F_a(X):=p^{-s_a}F_a(X)\in\Z_p\langle X\rangle
\]
has nonzero reduction
\[
        \Phi_a(X)=
        \sum_{\substack{0\le m\le d-1\\ m+v_p(C_{a,m})=s_a}}
        \overline{\,C_{a,m}/p^{s_a-m}\,}\binom{X}{m}
        \in\F_p[X].
\]
If \(e_a=\deg\Phi_a\), then \(e_a\le d-1\), and there are a distinguished polynomial
\[
        W_a(X)\in\Z_p[X],\qquad \deg W_a=e_a,
\]
and a unit \(V_a(X)\in\Z_p\langle X\rangle^\times\) such that
\[
        F_a(X)=p^{s_a}W_a(X)V_a(X).
\]
Consequently \(F_a\) has at most \(d-1\) zeros in \(\Z_p\), counted with Weierstrass multiplicity.  If \(\rho\in\F_p\) is a simple root of \(\Phi_a\), then there is a unique zero \(\tau_\rho\in\rho+p\Z_p\), and
\[
        v_p(F_a(t))=s_a+v_p(t-\tau_\rho)
        \qquad(t\in\rho+p\Z_p).
\]
\end{theorem}

\begin{proof}
The trace pairing on the finite \'etale \(\F_p\)-algebra \(A/pA\) is nondegenerate.  Since \(1,\omega,\ldots,\omega^{d-1}\) is a basis and \(x_a\ne0\), the element \(x_a\) cannot be orthogonal to every \(\omega^m\) with \(0\le m\le d-1\).  Hence some
\[
        C_{a,m}\equiv \Tr_{A/pA/\F_p}(x_a\omega^m)\pmod p
\]
with \(0\le m\le d-1\) is a \(p\)-adic unit.  This proves that \(s_a\) is finite and at most \(d-1\).

The binomial expansion gives
\[
        F_a(X)=\sum_{m\ge0}\binom Xm p^mC_{a,m}
\]
as a restricted \(p\)-adic power series.  By definition of \(s_a\), after division by \(p^{s_a}\) the terms with \(m<d\) have integral coefficients, and exactly the terms with \(m+v_p(C_{a,m})=s_a\) survive modulo \(p\).  If \(m\ge d\), then
\[
        m-s_a-v_p(m!)\ge m-(d-1)-\frac{m-1}{p-1}>0,
\]
because \(p>d\); the left side is an integer, so it is at least \(1\).  This is a coefficientwise estimate: since
\(\binom Xm=m!^{-1}\prod_{i=0}^{m-1}(X-i)\), every coefficient of
\(p^{m-s_a}\binom Xm C_{a,m}\) has valuation at least
\(m-s_a-v_p(m!)\).  Thus every term with \(m\ge d\) is divisible by \(p\) after division by \(p^{s_a}\).  The displayed reduction \(\Phi_a\) follows.  It is nonzero by the definition of \(s_a\).  Since \(m!\) is a unit for \(m\le d-1<p\), the degree of \(\Phi_a\) is the largest \(m\) occurring with nonzero coefficient, hence is at most \(d-1\).

The Weierstrass preparation theorem applied to \(\mathcal F_a\) gives
\[
        \mathcal F_a(X)=W_a(X)V_a(X),
\]
where \(W_a\) is distinguished of degree \(\deg\Phi_a\) and \(V_a\) is a unit.  The zero bound follows immediately.  If \(\rho\) is a simple root of \(\Phi_a\), Hensel's lemma gives a unique zero \(\tau_\rho\in\rho+p\Z_p\), and division by \(X-\tau_\rho\) in \(\Z_p\langle X\rangle\) gives the stated valuation formula on the residue disk.

\end{proof}

\begin{proposition}[Sharpness of the rank-\(d\) Weierstrass bound in split algebras]
\label{prop:rank-d-sharpness}
Let \(p>d\), and choose \(d\) elements
\[
        \Omega_1,\ldots,\Omega_d\in\Z_p
\]
whose reductions modulo \(p\) are pairwise distinct.  Put
\[
        A=\Z_p^d,
        \qquad
        q_i=1+p\Omega_i,
        \qquad
        \eta=(q_1,\ldots,q_d)\in A^\times .
\]
Then \(P=1\), the logarithmic tangent is
\[
        \omega=(\bar\Omega_1,\ldots,\bar\Omega_d)\in\F_p^d,
\]
and \(1,\omega,\ldots,\omega^{d-1}\) is an \(\F_p\)-basis of \(A/pA\).  Define
\[
        \gamma_i=\frac{p^{d-1}}{\prod_{j\ne i}(q_i-q_j)}\in\Z_p^\times,
        \qquad
        \gamma=(\gamma_1,\ldots,\gamma_d)\in A^\times .
\]
For the primitive branch
\[
        F(t)=\Tr_A(\gamma\eta^t)=\sum_{i=1}^d\gamma_iq_i^t,
\]
one has
\[
        F(0)=F(1)=\cdots=F(d-2)=0,
        \qquad
        F(d-1)=p^{d-1}\ne0.
\]
Consequently the upper bound \(d-1\) in Theorem~\ref{thm:rank-d-weierstrass-bound} is sharp, even in the split finite \'etale algebra \(\Z_p^d\).
\end{proposition}

\begin{proof}
Because \(\bar\eta=(1,\ldots,1)\), the period \(P\) is \(1\), and
\[
        \frac{\eta^P-1}{p}\equiv(\bar\Omega_1,\ldots,\bar\Omega_d)=\omega\pmod p.
\]
The Vandermonde determinant
\[
        \det(\bar\Omega_i^{\,j})_{1\le i\le d,\,0\le j\le d-1}
        =\prod_{i<j}(\bar\Omega_j-\bar\Omega_i)
\]
is nonzero in \(\F_p\), so \(1,\omega,\ldots,\omega^{d-1}\) is a basis of \(\F_p^d\).  Each difference \(q_i-q_j=p(\Omega_i-\Omega_j)\) has valuation \(1\), so the denominator in the definition of \(\gamma_i\) has valuation \(d-1\), and \(\gamma_i\in\Z_p^\times\).

It remains only to prove the vanishing identities.  Let
\[
        D_i=\prod_{j\ne i}(q_i-q_j).
\]
For every polynomial \(R(X)\) of degree at most \(d-1\), the coefficient of \(X^{d-1}\) in the Lagrange interpolation formula
\[
        R(X)=\sum_{i=1}^d R(q_i)\prod_{j\ne i}\frac{X-q_j}{q_i-q_j}
\]
is
\[
        \sum_{i=1}^d \frac{R(q_i)}{D_i}.
\]
Taking \(R(X)=X^r\) gives
\[
        \sum_{i=1}^d\frac{q_i^r}{D_i}=0\quad(0\le r\le d-2),
        \qquad
        \sum_{i=1}^d\frac{q_i^{d-1}}{D_i}=1.
\]
Multiplication by \(p^{d-1}\) gives the displayed values of \(F(r)\).  Theorem~\ref{thm:rank-d-weierstrass-bound} gives at most \(d-1\) zeros, counted with multiplicity, so these \(d-1\) distinct integral zeros show that the bound cannot be improved.
\end{proof}

\begin{theorem}[Tangent-subalgebra Weierstrass bound]
\label{thm:tangent-subalgebra-bound}
Let \(A\) be a finite \'etale \(\Z_p\)-algebra, with \(p\ge5\), let
\(\eta\in A^\times\), let \(P=\ord(\bar\eta)\), and write
\[
        \eta^P=1+pU,\qquad \omega=\bar U\in A/pA .
\]
Let
\[
        E_\omega=\F_p[\omega]\subset A/pA,
        \qquad r=\dim_{\F_p}E_\omega,
\]
and assume \(p>r\).  Fix a branch coefficient
\[
        y_a=\gamma\eta^a,\qquad x_a=\bar y_a\in A/pA,
\]
and suppose
\[
        x_a\notin E_\omega^\perp
        :=\{x\in A/pA:\Tr_{A/pA/\F_p}(xe)=0
        \text{ for every }e\in E_\omega\}.
\]
Then the branch
\[
        F_a(t)=\Tr_A\bigl(y_a(\eta^P)^t\bigr)
\]
has a Weierstrass factor of degree at most \(r-1\).  In particular, \(F_a\)
has at most \(r-1\) zeros in \(\Z_p\), counted with Weierstrass multiplicity.
When \(E_\omega=A/pA\), this recovers Theorem~\ref{thm:rank-d-weierstrass-bound}.
\end{theorem}

\begin{proof}
Put
\[
        C_m=\Tr_A(y_aU^m)\qquad(m\ge0).
\]
Since \(E_\omega=\F_p[\omega]\) has \(\F_p\)-dimension \(r\), the minimal
polynomial of \(\omega\) over \(\F_p\) has degree \(r\).  Hence
\(1,\omega,\ldots,\omega^{r-1}\) is an \(\F_p\)-basis of the subalgebra
\(E_\omega\), and the hypothesis \(x_a\notin E_\omega^\perp\) implies that
\[
        \Tr_{A/pA/\F_p}(x_a\omega^m)\ne0
\]
for at least one \(m\) with \(0\le m\le r-1\).  Hence
\[
        s=\min_{0\le m\le r-1}\bigl(m+v_p(C_m)\bigr)
\]
is finite and satisfies \(s\le r-1\).  The binomial expansion gives
\[
        F_a(X)=\sum_{m\ge0}\binom Xm p^mC_m .
\]
After division by \(p^s\), the terms with \(0\le m\le r-1\) are integral, and
at least one of them survives modulo \(p\).  If \(m\ge r\), then Legendre's
estimate \(v_p(m!)\le (m-1)/(p-1)\) gives
\[
        m-s-v_p(m!)\ge m-(r-1)-v_p(m!)
        \ge m-(r-1)-\frac{m-1}{p-1}>0.
\]
The last expression is minimized at \(m=r\), where it is
\(1-(r-1)/(p-1)>0\), since \(p>r\).  The left side is an integer, hence it is
at least \(1\).  Again the estimate is coefficientwise through the identity
\[
        \binom Xm=m!^{-1}\prod_{i=0}^{m-1}(X-i).
\]
Thus all terms with \(m\ge r\) vanish modulo \(p\) after division by \(p^s\).
The reduced restricted series is therefore a nonzero
polynomial in the binomial basis of degree at most \(r-1\).  Weierstrass
preparation gives a distinguished factor of degree at most \(r-1\), and the
zero bound follows.
\end{proof}

\begin{lemma}[When the primitive-tangent basis condition holds]\label{lem:tangent-basis-criterion}
Let \(B\) be a finite \'etale cubic algebra over \(\F_p\), and let \(\omega\in B\).  Then \(1,\omega,\omega^2\) is an \(\F_p\)-basis of \(B\) if and only if the subalgebra \(\F_p[\omega]\) is all of \(B\).  In the three splitting types this condition is as follows.
\begin{enumerate}[label=\textup{(\roman*)}]
\item If \(B=\F_{p^3}\), it is equivalent to \(\omega\notin\F_p\).
\item If \(B=\F_p\times\F_{p^2}\) and \(\omega=(a,b)\), it is equivalent to \(b\notin\F_p\).
\item If \(B=\F_p^3\) and \(\omega=(\omega_1,\omega_2,\omega_3)\), it is equivalent to the three coordinates \(\omega_1,\omega_2,\omega_3\) being pairwise distinct.
\end{enumerate}
\end{lemma}

\begin{proof}
The span of \(1,\omega,\omega^2\) is the image of the polynomials of degree at most \(2\) under the evaluation map \(\F_p[T]\to B\), \(T\mapsto\omega\).  Since \(B\) has dimension \(3\), these three elements form a basis exactly when the generated subalgebra \(\F_p[\omega]\) has dimension \(3\), equivalently when it is all of \(B\).

If \(B=\F_{p^3}\), the only proper subfield is \(\F_p\), because \(3\) is prime.  This proves (i).  If \(B=\F_p\times\F_{p^2}\) and \(\omega=(a,b)\), then the image has dimension equal to the degree of the least common multiple of the minimal polynomials of \(a\) and \(b\) over \(\F_p\).  This degree is \(3\) exactly when \(b\) has degree \(2\), i.e. when \(b\notin\F_p\); the linear factor \(T-a\) is then coprime to the irreducible quadratic minimal polynomial of \(b\).  Finally, in \(\F_p^3\), the same least-common-multiple description gives degree equal to the number of distinct coordinates of \(\omega\), proving (iii).
\end{proof}

The preceding two theorems give the complete finite recursion at every good prime.  In the cubic applications below, \(p\ge5\) and the primitive-tangent basis condition forces the first nonzero trace jet of a nonzero reduced class to occur with \(r\le2<p\), so the finite-jet statement applies exactly in the singular situations used here.  In cubic dimension one can say more because the second-order shifted reduction is a genuine quadratic.  The following theorem is the higher-order singular-branch classification used in the rest of the paper.

\begin{theorem}[Quadratic classification of primitive singular disks]\label{thm:quadratic-singular-classification}
Assume the hypotheses of Theorem \ref{thm:general-local-skolem}, and assume in addition that
\[
        1,\omega,\omega^2
\]
form an \(\F_p\)-basis of \(\overline A\).  Let \(a\in Z_p(1)\) be a primitive first-order singular class, meaning that
\[
        x_a\ne0,
        \qquad
        \Tr_{\overline A/\F_p}(x_a)=0,
        \qquad
        \Tr_{\overline A/\F_p}(x_a\omega)=0.
\]
Then
\begin{equation}\label{eq:second-tangent-nonzero}
        \Delta_a:=\Tr_{\overline A/\F_p}(x_a\omega^2)\ne0.
\end{equation}
If \(T_a/p\not\equiv0\pmod p\), the singular class dies modulo \(p^2\).  If \(T_a/p\equiv0\pmod p\), define
\[
        A_a\equiv \frac{T_a}{p^2}\pmod p,
        \qquad
        B_a\equiv \frac{\Tr_A(\gamma\eta^aU)}p\pmod p,
\]
and the quadratic Hensel polynomial
\begin{equation}\label{eq:quadratic-hensel-polynomial}
        Q_a(X)=A_a+B_aX+\binom{X}{2}\Delta_a\in\F_p[X].
\end{equation}
Then
\begin{equation}\label{eq:second-order-expansion}
        \frac{F_a(t)}{p^2}\equiv Q_a(t)\pmod p.
\end{equation}
Consequently, the possible first \(t\)-digits of a lift from the singular class to precision \(p^3\) are exactly the roots of \(Q_a\) in \(\F_p\): for \(\rho\in\F_p\), all \(t\in\tilde\rho+p\Z_p\) satisfy \(F_a(t)\equiv0\pmod {p^3}\) if and only if \(Q_a(\rho)=0\), where \(\tilde\rho\in\{0,\ldots,p-1\}\) denotes the least representative.

If \(r\in\F_p\) is a simple root of \(Q_a\), and \(\tilde r\in\{0,\ldots,p-1\}\) is its least representative, then there is a unique zero \(\tau_{a,r}\in \tilde r+p\Z_p\) with \(F_a(\tau_{a,r})=0\), and
\begin{equation}\label{eq:second-order-valuation}
        v_p(F_a(t))=2+v_p(t-\tau_{a,r})
        \qquad(t\in \tilde r+p\Z_p).
\end{equation}
If \(r\) is a double root of \(Q_a\), again let \(\tilde r\in\{0,\ldots,p-1\}\) be its least representative.  The survival and singularity hypotheses imply
\(F_a(\tilde r+Y)\in p^2\Z_p\langle Y\rangle\) coefficientwise after this
integral shift, so define the shifted restricted series
\[
        H_{a,r}(Y)=\frac{F_a(\tilde r+Y)}{p^2}\in\Z_p\langle Y\rangle .
\]
Then there is a uniquely determined distinguished quadratic
\[
        W_{a,r}(Y)=Y^2+b_{a,r}Y+c_{a,r}\in\Z_p[Y],
        \qquad b_{a,r},c_{a,r}\in p\Z_p,
\]
and a unit \(V_{a,r}(Y)\in\Z_p\langle Y\rangle^\times\) such that
\begin{equation}\label{eq:quadratic-weierstrass}
        H_{a,r}(Y)=W_{a,r}(Y)V_{a,r}(Y).
\end{equation}
The roots of \(W_{a,r}\), when they exist in \(\Q_p\), are integral and reduce to \(0\) modulo \(p\), hence automatically satisfy \(Y\in p\Z_p\).  They therefore correspond to zeros of \(F_a\) in the residue disk \(t\in\tilde r+p\Z_p\).  Let \(\delta_{a,r}=b_{a,r}^2-4c_{a,r}\in\Z_p\).  This disk contains no \(p\)-adic zero, one double \(p\)-adic zero, or two simple \(p\)-adic zeros according as \(\delta_{a,r}\) is not a square in \(\Q_p\), is zero, or is a nonzero square in \(\Q_p\).  For every \(k\ge3\), the congruence classes in this disk satisfying \(F_a(t)\equiv0\pmod {p^k}\) are exactly the solutions of
\[
        W_{a,r}(Y)\equiv0\pmod {p^{k-2}},
        \qquad t=\tilde r+Y,
        \qquad Y\in p\Z_p.
\]
\end{theorem}

\begin{proof}
The trace pairing on the finite \'etale \(\F_p\)-algebra \(\overline A\) is nondegenerate.  Since \(1,\omega,\omega^2\) is a basis, an element orthogonal to all three basis vectors must be zero.  The primitive hypothesis gives \(x_a\ne0\), while first-order singularity gives orthogonality to \(1\) and \(\omega\).  Therefore \(x_a\) cannot also be orthogonal to \(\omega^2\).  This proves \eqref{eq:second-tangent-nonzero}.

Assume now \(T_a/p\equiv0\pmod p\).  Expanding \((1+pU)^t\) modulo \(p^3\) gives
\[
        (1+pU)^t\equiv 1+tpU+\binom{t}{2}p^2U^2\pmod {p^3}.
\]
Taking traces and using the two singularity congruences gives \eqref{eq:second-order-expansion}.  Since the right side depends only on \(t\bmod p\), the roots of \(Q_a\) are precisely the first \(t\)-digits of the residue disks surviving modulo \(p^3\).

The same hypotheses also justify the coefficientwise division by \(p^2\).  In
\[
        F_a(X)=T_a+pX\Tr_A(\gamma\eta^aU)+
        \sum_{m\ge2}p^m\binom Xm\Tr_A(\gamma\eta^aU^m),
\]
the first term lies in \(p^2\Z_p\) and the second term lies in
\(p^2\Z_p[X]\).  For \(m=2\), \(2!\) is a unit; for \(m\ge3\), every coefficient
of \(p^{m-2}\binom Xm\) has valuation at least
\[
        m-2-v_p(m!)\ge0,
\]
because
\[
        m-2-v_p(m!)\ge m-2-\frac{m-1}{p-1}>0
\]
for \(m\ge3\) and \(p\ge5\).  Hence
\(F_a(X)\in p^2\Z_p\langle X\rangle\) coefficientwise, and the restricted
series \(F_a(X)/p^2\) has reduction \(Q_a(X)\).  This coefficientwise
divisibility remains true after every integral translation of \(X\).

For a simple root \(r\), let \(\tilde r\) be its least integer representative.  The derivative of \(F_a(X)/p^2\) at \(X=\tilde r\) is nonzero modulo \(p\).  Hensel's lemma gives the unique zero in \(\tilde r+p\Z_p\), and division by \(X-\tau_{a,r}\) gives the valuation formula \eqref{eq:second-order-valuation}.  If \(r\) is a double root, then with
\(Y=t-\tilde r\) the shifted series \(H_{a,r}(Y)=F_a(\tilde r+Y)/p^2\) is
congruent to \((\Delta_a/2)Y^2\) modulo \(p\).  Since \(p\ge5\) and
\(\Delta_a\ne0\), the Weierstrass preparation theorem for restricted power
series \cite[Section 5.2]{BoschGuntzerRemmert1984} gives
\eqref{eq:quadratic-weierstrass} with a distinguished quadratic reducing to
\(Y^2\) modulo \(p\).  Its coefficients lie in \(p\Z_p\), so every zero of
\(W_{a,r}\) lies in \(p\Z_p\).  The final assertions follow because multiplication
by a unit does not change valuations or congruence-zero classes.

\end{proof}

\begin{corollary}[Practical alternatives for a primitive singular class]\label{cor:primitive-singular-alternatives}
In the situation of Theorem~\ref{thm:quadratic-singular-classification}, assume \(T_a/p\equiv0\pmod p\).  Write
\[
        D_a=\left(B_a-\frac{\Delta_a}{2}\right)^2-2\Delta_a A_a\in\F_p
\]
for the discriminant of the quadratic \(Q_a(X)=A_a+B_aX+\binom X2\Delta_a\).  Then:
\begin{enumerate}[label=\textup{(\roman*)}]
\item if \(D_a\) is a nonsquare in \(\F_p\), the singular class has no lift to a zero modulo \(p^3\);
\item if \(D_a\ne0\) is a square in \(\F_p\), the singular class splits into two simple \(p\)-adic branches, each with valuation formula \eqref{eq:second-order-valuation};
\item if \(D_a=0\), the class has one double first digit, and all further lifting in that disk is controlled by the distinguished quadratic \(W_{a,r}\) of Theorem~\ref{thm:quadratic-singular-classification}.
\end{enumerate}
\end{corollary}

\begin{proof}
The polynomial \(Q_a\) has leading coefficient \(\Delta_a/2\ne0\), because \(p\ge5\) and \(\Delta_a\ne0\).  Its roots in \(\F_p\) are therefore classified by the displayed discriminant.  The three alternatives are exactly the no-root, two-simple-root, and double-root cases of Theorem~\ref{thm:quadratic-singular-classification}.
\end{proof}

\begin{corollary}[A sharp two-zero bound above a primitive cubic class]\label{cor:two-zero-bound}
Assume the hypotheses of Theorem~\ref{thm:general-local-skolem}, and assume in
addition that
\[
        1,\omega,\omega^2
\]
form an \(\F_p\)-basis of \(\overline A\).  Let \(a\in Z_p(1)\) be primitive, so
\(x_a\ne0\), and put
\[
        d_a=\Tr_{\overline A/\F_p}(x_a\omega).
\]
Then the branch function
\[
        F_a(t)=\Tr_A\bigl(\gamma\eta^a(\eta^P)^t\bigr)
\]
has at most two zeros in \(\Z_p\), counted with Weierstrass multiplicity.  More
precisely, if \(d_a\ne0\), then the class is transverse and has exactly one
simple zero; if \(d_a=0\), then the class is a primitive first-order singular
class and has no zero, one double zero, or two simple zeros according to the
lower-obstruction and quadratic alternatives in
Theorem~\ref{thm:quadratic-singular-classification}.
\end{corollary}

\begin{proof}
If \(d_a\ne0\), Theorem~\ref{thm:general-local-skolem} gives one simple Hensel
zero and the valuation formula \eqref{eq:general-valuation}.  If \(d_a=0\), then
\(a\in Z_p(1)\), \(x_a\ne0\), and the basis hypothesis is exactly the additional
hypothesis needed to apply Theorem~\ref{thm:quadratic-singular-classification} to
this primitive first-order singular class.  If \(T_a/p\not\equiv0\pmod p\), the class dies modulo \(p^2\) and hence has no
zero.  If \(T_a/p\equiv0\pmod p\), the theorem defines \(Q_a\) and excludes any
surviving disk whose first digit is not a root of \(Q_a\).  A simple root of
\(Q_a\) contributes one simple Hensel zero.  If \(Q_a\) has a double root, the
corresponding shifted branch is a unit times a distinguished quadratic
\(W_{a,r}\), so it contributes at most two zeros counted with multiplicity.
Since \(Q_a\) is a nonzero quadratic, these are the only singular possibilities
after the lower obstruction has vanished.

\end{proof}

\begin{corollary}[Nondegenerate affine singular disks]
\label{cor:affine-nondegenerate-singular}
Assume the hypotheses of Theorem~\ref{thm:affine-local-branch}, and assume
\(1,\omega,\omega^2\) is an \(\F_p\)-basis of \(\overline A\).  Let
\(a\bmod P\) be an affine target class for \(c\), and put
\(x_a=\bar\gamma\bar\eta^a\).  Suppose
\[
        x_a\ne0,
        \qquad
        \Tr(x_a)=\bar c,
        \qquad
        \Tr(x_a\omega)=0,
        \qquad
        \Delta_a:=\Tr(x_a\omega^2)\ne0 .
\]
Then the same quadratic Hensel classification as in
Theorem~\ref{thm:quadratic-singular-classification} holds with \(F_a\) replaced
by \(F_{a,c}\).  In particular, if
\[
        \frac{\Tr_A(\gamma\eta^a)-c}{p}\not\equiv0\pmod p,
\]
the class dies modulo \(p^2\).  If it survives to modulo \(p^2\), define
\[
        A_{a,c}\equiv \frac{\Tr_A(\gamma\eta^a)-c}{p^2}\pmod p,
        \qquad
        B_a\equiv \frac{\Tr_A(\gamma\eta^aU)}p\pmod p,
\]
and
\[
        Q_{a,c}(X)=A_{a,c}+B_aX+\binom X2\Delta_a\in\F_p[X].
\]
The first digits of lifts to precision \(p^3\) are exactly the roots of
\(Q_{a,c}\).  A simple root gives one simple Hensel branch with valuation
\[
        v_p(F_{a,c}(t))=2+v_p(t-\tau),
\]
while a double root gives a distinguished quadratic Weierstrass disk.  Thus a
nondegenerate affine singular class contributes no zero, one double zero, or two
simple zeros, counted with Weierstrass multiplicity.
\end{corollary}

\begin{proof}
Subtracting \(c\) changes only the constant coefficient in the binomial
expansion.  The expansion modulo \(p^3\) is
\[
        F_{a,c}(t)=(\Tr_A(\gamma\eta^a)-c)
        +tp\Tr_A(\gamma\eta^aU)
        +\binom t2 p^2\Tr_A(\gamma\eta^aU^2)
        \pmod {p^3}.
\]
The affine target condition makes the first term divisible by \(p\), the
singularity condition makes the coefficient of \(t\) divisible by \(p\), and the
hypothesis \(\Delta_a\ne0\) makes the quadratic coefficient a unit after
division by \(p^2\).  If the class survives to modulo \(p^2\), then
\(\Tr_A(\gamma\eta^a)-c\in p^2\Z_p\), while
\(\Tr_A(\gamma\eta^aU)\in p\Z_p\); the same coefficientwise estimate
\(m=2\) with \(2!\in\Z_p^\times\), and
\[
        m-2-v_p(m!)\ge m-2-\frac{m-1}{p-1}>0
\]
for \(m\ge3\), shows that
\(F_{a,c}(X)\in p^2\Z_p\langle X\rangle\).  Thus the proof of
Theorem~\ref{thm:quadratic-singular-classification} applies verbatim, with the
constant term replaced by \(\Tr_A(\gamma\eta^a)-c\).
\end{proof}

\begin{theorem}[Cubic resolution of affine-degenerate singular disks]
\label{thm:cubic-affine-degenerate-singular}
Assume the hypotheses of Theorem~\ref{thm:affine-local-branch}, and assume
\(\overline A=A/pA\) is cubic and generated by \(\omega\), so that
\(1,\omega,\omega^2\) is an \(\F_p\)-basis.  Let
\(f_\omega(T)=T^3+\alpha T^2+\beta T+\delta_0\) be the characteristic
polynomial of multiplication by \(\omega\) on \(\overline A\).  Equivalently,
because \(\omega\) generates \(\overline A\), \(f_\omega\) is the monic
squarefree generator polynomial with
\(\overline A\simeq\F_p[T]/(f_\omega)\).  Assume \(\omega\) is a unit in
\(\overline A\), equivalently \(\Norm(\omega)=-\delta_0\ne0\).  Let
\[
        y_a=\gamma\eta^a,
        \qquad C_m=\Tr_A(y_aU^m)\quad(m\ge0),
\]
where \(\eta^P=1+pU\).  Suppose that the reduced affine singular class is the
trace-dual degenerate class
\[
        x_a=\bar y_a=s z_0,
        \qquad s=\bar c\ne0,
\]
where \(z_0,z_1,z_2\) is the trace-dual basis to
\(1,\omega,\omega^2\).  Then
\[
        \Tr(x_a\omega)=0,
        \qquad \Tr(x_a\omega^2)=0,
        \qquad \Tr(x_a\omega^3)=s\Norm(\omega)\ne0 .
\]
If the lower congruence obstructions satisfy
\[
        C_0-c\in p^3\Z_p,
        \qquad C_1\in p^2\Z_p,
        \qquad C_2\in p\Z_p,
\]
then \(F_{a,c}(X)\in p^3\Z_p\langle X\rangle\) coefficientwise, and after any
integral shift it remains divisible by \(p^3\) in the restricted power-series
ring.  Moreover
\[
        \frac{F_{a,c}(X)}{p^3}\equiv
        A^{(3)}+B^{(3)}X+D^{(3)}\binom X2+s\Norm(\omega)\binom X3
        \pmod p,
\]
where
\[
        A^{(3)}\equiv\frac{C_0-c}{p^3},
        \qquad
        B^{(3)}\equiv\frac{C_1}{p^2},
        \qquad
        D^{(3)}\equiv\frac{C_2}{p}
        \pmod p .
\]
Consequently the first digits of lifts from this degenerate affine singular
class to precision \(p^4\) are exactly the roots in \(\F_p\) of the cubic
Hensel polynomial
\[
        R_{a,c}(X)=A^{(3)}+B^{(3)}X+D^{(3)}\binom X2+s\Norm(\omega)\binom X3 .
\]
A simple root gives one simple Hensel branch with valuation
\[
        v_p(F_{a,c}(t))=3+v_p(t-\tau)
\]
in the corresponding residue disk.  If \(r\) is a multiple root of multiplicity
\(e\ge2\), then the relevant disk is governed only by the local Weierstrass
factor at \(Y=0\).  The lower divisibility hypotheses imply coefficientwise
\[
        F_{a,c}(\tilde r+Y)\in p^3\Z_p\langle Y\rangle,
\]
so, with
\[
        G_r(Y)=p^{-3}F_{a,c}(\tilde r+Y)\in\Z_p\langle Y\rangle,
\]
one has \(\overline G_r(Y)=R_{a,c}(\tilde r+Y)=Y^eQ(Y)\) with \(Q(0)\ne0\).
Thus \(Y^e\) and \(Q(Y)\) are coprime in \(\F_p[Y]\).  Weierstrass preparation
and Hensel factorization split off a unique local factor whose reduction is a
unit times \(Y^e\); the zeros in the original residue disk
\(t\in\tilde r+p\Z_p\) are exactly the zeros of this local factor with
\(Y\in p\Z_p\).  Factors with nonzero reduction roots correspond to other
first-digit residue classes and are excluded.
\end{theorem}

\begin{proof}
The trace-dual basis satisfies \(\Tr(z_i\omega^j)=\delta_{ij}\) for
\(0\le i,j\le2\).  Since
\[
        \omega^3=-\alpha\omega^2-\beta\omega-\delta_0,
\]
we have
\[
        \Tr(z_0\omega^3)=-\delta_0=\Norm(\omega).
\]
Multiplying by \(s\) gives the three displayed trace identities for
\(x_a=s z_0\).  The last coefficient is nonzero by the unit hypothesis on
\(\omega\) and by \(s\ne0\).

The binomial expansion modulo \(p^4\) is
\[
\begin{aligned}
        F_{a,c}(X)
        &=(C_0-c)+pX C_1+p^2\binom X2 C_2
          +p^3\binom X3 C_3^{\rm jet} \pmod {p^4},
\end{aligned}
\]
where \(C_3^{\rm jet}=\Tr_A(y_aU^3)\).  The hypotheses make the first three terms
divisible by \(p^3\).  For the remaining binomial terms, coefficientwise
integrality follows from
\[
        v_p\left(\frac{p^m}{m!}\right)=m-v_p(m!)\ge3
        \qquad(m\ge3,\ p\ge5),
\]
with equality possible only at \(m=3\); for \(m\ge4\) this follows, for example,
from \(v_p(m!)\le(m-1)/(p-1)\).  Since \(C_m=\Tr_A(y_aU^m)\in\Z_p\), every term
\(p^m\binom Xm C_m\) with \(m\ge3\) lies coefficientwise in
\(p^3\Z_p[X]\).  Hence \(F_{a,c}(X)\in p^3\Z_p\langle X\rangle\)
coefficientwise, and the same is true after any integral translation of \(X\).
Reducing after division by \(p^3\), and using
\(U\equiv\omega\pmod p\), gives
\[
        C_3^{\rm jet}\equiv\Tr(x_a\omega^3)=s\Norm(\omega)\pmod p.
\]
This proves the cubic reduction.  The criterion for first lift digits is the
affine-shifted instance of the finite-jet criterion of Theorem~\ref{thm:finite-jet-stratification}: one applies the same calculation to the modified coefficients \(C'_0=C_0-c\) and \(C'_m=C_m\) for \(m\ge1\), with \(r=3\).  If a root is simple, Hensel's lemma applied to
\(p^{-3}F_{a,c}\) in the corresponding disk gives the simple branch and the
valuation formula.  If the root is multiple of multiplicity \(e\), then
\(\overline G_r(Y)=R_{a,c}(\tilde r+Y)=Y^eQ(Y)\) with \(Q(0)\ne0\).  Hence
\(Y^e\) and \(Q(Y)\) are coprime in \(\F_p[Y]\), so Hensel factorization
separates the factor reducing to \(Q(Y)\), which is a unit on the disk
\(Y\in p\Z_p\), from the unique local factor reducing to a unit multiple of
\(Y^e\).  Weierstrass preparation applied to this local factor gives the
distinguished polynomial governing the zeros in \(Y\in p\Z_p\).  The other
factors have nonzero reduction roots and correspond to other residue disks, not
to the original disk \(t\in\tilde r+p\Z_p\).
\end{proof}

\begin{remark}[How the affine-degenerate case fits the quadratic theory]
\label{rem:affine-degenerate-singular}
For \(c=0\), a primitive singular class is automatically nondegenerate under
the basis hypothesis, because it is orthogonal to \(1\) and \(\omega\) but not
to all of \(1,\omega,\omega^2\).  For a nonzero affine target, the parameter
value \(u=0\) in the codifferent singular line is the only possible class with
zero second trace tangent.  Theorem~\ref{thm:cubic-affine-degenerate-singular}
shows that, when \(\bar c\Norm(\omega)\ne0\) and the lower obstructions vanish,
this class is not an unresolved exception: the first surviving cubic model is
controlled by an explicit cubic Hensel polynomial.

If one of the displayed divisibility hypotheses fails, the cubic normal form at
that level simply does not apply.  The class is then governed by the earlier
finite digit recursion of Theorem~\ref{thm:affine-local-branch}: depending on
the first nonzero lower obstruction, it may die, or it may resolve by a lower
linear or quadratic Hensel polynomial.  In particular, failure of the cubic
hypotheses is not, by itself, a death criterion.
\end{remark}

\begin{proposition}[Analytic local intersection multiplicity]
\label{prop:intersection-multiplicity}
Let \(a\bmod P\) and \(c\in\Z_p\) be fixed, and assume
\(F_{a,c}\) is not the zero restricted power series.  Let
\[
        \iota_a:\Z_p\longrightarrow A^\times,
        \qquad T\longmapsto \eta^a(\eta^P)^T,
\]
be the one-variable analytic torus branch, and let
\[
        L_{\gamma,c}(X)=\Tr_A(\gamma X)-c
\]
be the affine linear function on the ambient affine \(\Z_p\)-module \(A\).  The pullback
\(L_{\gamma,c}\circ\iota_a\) is exactly \(F_{a,c}(T)\).  We use local intersection multiplicity
in the one-variable \(p\)-adic analytic sense: after writing
\(F_{a,c}=p^sWV\) by Weierstrass preparation, with \(W\) distinguished and
\(V\) a unit, the multiplicity is \(\deg W\).  Thus vertical powers of \(p\)
are removed; this is not a claim about the full scheme-theoretic length over
\(\Z_p\).

With this convention, the Weierstrass degree of \(F_{a,c}\) is the local
intersection multiplicity of the formal orbit with the affine trace hyperplane
on the branch \(a+P\Z_p\).  The following degree assertions are recorded here for orientation and follow from the preceding quadratic and cubic singular classifications; they are not used in the proofs of those classifications.  Transverse classes have multiplicity \(1\).
Homogeneous primitive singular classes, and nondegenerate affine singular
classes with \(\Tr(x_a\omega^2)\ne0\), have multiplicity at most \(2\) under
the basis hypothesis \(1,\omega,\omega^2\).  A genuinely affine degenerate
singular class satisfying the hypotheses of
Theorem~\ref{thm:cubic-affine-degenerate-singular} has first nonzero model of
degree at most \(3\); if the hypotheses fail, the finite digit recursion still
determines whether the class dies or resolves at a lower level.
\end{proposition}

\begin{proof}
The first assertion is the definition of \(F_{a,c}\): evaluating the affine
linear function \(L_{\gamma,c}(X)=\Tr_A(\gamma X)-c\) on
\(X=\eta^a(\eta^P)^T\) gives the restricted series \(F_{a,c}(T)\).  For a
nonzero restricted series over \(\Z_p\), Weierstrass preparation writes it as a
power of \(p\), a distinguished polynomial, and a unit.  By the convention in the statement, the degree of the distinguished polynomial
is the corresponding one-dimensional analytic intersection multiplicity.  The final
claims are Theorem~\ref{thm:affine-local-branch} in the transverse case,
Corollary~\ref{cor:two-zero-bound} in the homogeneous primitive singular case,
Corollary~\ref{cor:affine-nondegenerate-singular} in the nondegenerate
affine singular case, and Theorem~\ref{thm:cubic-affine-degenerate-singular} in
the cubic affine-degenerate case.
\end{proof}

\begin{corollary}[Homogeneous period-wise bound for primitive cubic branch zeros]\label{cor:period-wise-zero-bound}
In the homogeneous target case \(c=0\), assume \(p\ge5\), \(A/\Z_p\) is finite \'etale cubic, \(\eta\in A^\times\), \(P=\ord(\bar\eta)\), \(\eta^P=1+pU\), \(\omega=\bar U\), and \(1,\omega,\omega^2\) is an \(\F_p\)-basis of \(A/pA\).  Assume also that \(\gamma\notin pA\).  Then every branch class \(a\bmod P\) has at most two zeros of
\[
        F_a(t)=\Tr_A\bigl(\gamma\eta^a(\eta^P)^t\bigr)
\]
in its own parameter \(t\in\Z_p\), counted with Weierstrass multiplicity.  Consequently the total number of branch zeros in the disjoint local branch space \(\mathscr B_p=\Z/P\Z\times\Z_p\) is at most \(2\#Z_p(1)\).  More precisely, each transverse mod-\(p\) zero class contributes exactly one simple zero, each primitive singular class contributes no zero, one double zero, or two simple zeros, and every class outside \(Z_p(1)\) contributes none.
\end{corollary}

\begin{proof}
Since \(\gamma\notin pA\) and \(\eta\) is a unit, \(x_a=\bar\gamma\bar\eta^a\) is nonzero for every \(a\).  If \(a\notin Z_p(1)\), then \(F_a(t)\not\equiv0\pmod p\) for all \(t\in\Z_p\), so there is no zero in that branch.  If \(a\in Z_p(1)\), the class is either transverse or primitive first-order singular.  The transverse case is Theorem~\ref{thm:general-local-skolem}, and the primitive singular case is Corollary~\ref{cor:two-zero-bound}.  Summing the branchwise bounds over the mod-\(p\) zero classes gives the stated period-wise bound.
\end{proof}

\begin{example}[A singular disk that splits into two simple branches]\label{ex:singular-splits}
The quadratic singular theorem is not only a non-lifting criterion.  It can produce several genuine branches.  Work at \(p=5\) in the split algebra \(A=\Z_5^3\), take
\[
        \Omega=(0,1,2),
        \qquad
        \eta=1+5\Omega=(1,6,11),
        \qquad
        \gamma=(1,-2,1).
\]
Then \(P=1\), \(\omega=\overline\Omega=(0,1,2)\), and \(1,\omega,\omega^2\) is a Vandermonde basis of \(\F_5^3\).  The unique residue class \(a=0\) is primitive singular, because
\[
        \Tr(\bar\gamma)=0,
        \qquad
        \Tr(\bar\gamma\omega)=0,
        \qquad
        \Tr(\bar\gamma\omega^2)=2\ne0
        \quad(\bmod 5).
\]
Moreover \(T_0=\Tr(\gamma)=0\) and \(\Tr(\gamma\Omega)=0\), so the quadratic Hensel polynomial is
\[
        Q_0(X)=2\binom X2=X(X-1)\in\F_5[X].
\]
Thus the single first-order singular class has two surviving first lift digits to precision \(5^3\), namely \(t\equiv0\) and \(t\equiv1\) modulo \(5\), and each residue disk continues to a unique simple \(5\)-adic branch.  This example also shows why singular classes should be resolved rather than discarded.
\end{example}

\begin{theorem}[Versality of primitive quadratic singular disks]\label{thm:versal-quadratic-singular-disks}
Let \(p\ge5\).  Work in the split cubic algebra \(A=\Z_p^3\), put
\[
        \Omega=(0,1,2),\qquad \eta=1+p\Omega,
        \qquad P=1,
\]
and let \(\omega=\bar\Omega\in\F_p^3\).  For every quadratic polynomial in binomial form
\[
        Q(X)=A_0+B_0X+C_0\binom X2\in\F_p[X],
        \qquad C_0\ne0,
\]
there exists a primitive coefficient \(\gamma\in A\) such that the unique mod-\(p\) zero class \(a=0\) is primitive first-order singular and the quadratic Hensel polynomial of Theorem~\ref{thm:quadratic-singular-classification} is exactly \(Q(X)\).
\end{theorem}

\begin{proof}
Choose integer representatives for \(A_0,B_0,C_0\).  Put
\[
        x=\frac{C_0}{2}(1,-2,1)\in\Z_p^3,
        \qquad
        y=(-B_0,B_0,0),
        \qquad
        z=(A_0,0,0),
\]
where \(1/2\) denotes the inverse of \(2\) in \(\Z_p\), and set
\[
        \gamma=x+py+p^2z.
\]
The elements \(1,\omega,\omega^2\) form a Vandermonde basis of \(\F_p^3\), because the coordinates \(0,1,2\) are pairwise distinct.  Direct computation gives
\[
        \Tr(x)=0,
        \qquad
        \Tr(x\Omega)=0,
        \qquad
        \Tr(x\Omega^2)=C_0.
\]
Since \(C_0\ne0\), the reduction of \(\gamma\) is nonzero.  Thus \(a=0\) is primitive and first-order singular.  Moreover
\[
        \Tr(\gamma)=p^2A_0,
        \qquad
        \Tr(\gamma\Omega)=pB_0,
        \qquad
        \Tr(\bar\gamma\omega^2)=C_0.
\]
In the notation of Theorem~\ref{thm:quadratic-singular-classification}, the three coefficients of the singular Hensel polynomial are therefore \(A_0\), \(B_0\), and \(\Delta=C_0\).  Hence
\[
        Q_0(X)=A_0+B_0X+C_0\binom X2=Q(X),
\]
as required.
\end{proof}

\begin{corollary}[All quadratic singular alternatives occur]\label{cor:all-singular-alternatives}
The alternatives in Corollary~\ref{cor:primitive-singular-alternatives} are all sharp.  Already at \(p=5\) in \(\Z_5^3\), primitive singular disks may split into two simple branches, die before precision \(5^3\), or enter a double-root disk governed by a distinguished quadratic Weierstrass factor.
\end{corollary}

\begin{proof}
Apply Theorem~\ref{thm:versal-quadratic-singular-disks} at \(p=5\) with \(C_0=2\).  The three choices
\[
\begin{array}{c|c|c|c}
\hline
(A_0,B_0) & Q(X) & \text{root pattern in }\F_5 & \text{local consequence}\\
\hline
(0,0) & X(X-1) & \text{two simple roots} & \text{two simple branches}\\
(2,1) & X^2+2 & \text{no root} & \text{no lift to }5^3\\
(0,1) & X^2 & \text{one double root} & \text{quadratic Weierstrass disk}\\
\hline
\end{array}
\]
realize the three cases.  In the double-root row, Theorem~\ref{thm:quadratic-singular-classification} supplies the distinguished quadratic Weierstrass factor governing all further lifts in the disk \(t\in5\Z_5\).
\end{proof}

\begin{example}[The basis hypothesis is necessary]\label{ex:basis-hypothesis-necessary}
Again let \(A=\Z_p^3\), now with \(p\ge5\), \(\eta=(1+p,1+p,1+p)\), and \(\gamma=(1,-1,0)\).  Then \(P=1\) and \(\omega=(1,1,1)\) is scalar modulo \(p\).  The reduced class \(x=\bar\gamma\) satisfies
\[
        \Tr(x)=\Tr(x\omega)=\Tr(x\omega^2)=0.
\]
Hence the second tangent \(\Delta=\Tr(x\omega^2)\) vanishes.  This does not contradict Theorem~\ref{thm:quadratic-singular-classification}; it shows that the hypothesis that \(1,\omega,\omega^2\) span the reduced cubic algebra is essential.
\end{example}

\begin{corollary}[Certified finite local branch algorithm]\label{cor:certified-local-algorithm}
Fix a precision \(k\ge1\).  The algorithmic assertion in this corollary assumes
an effective input model for \(A\): for example, a \(\Z_p\)-basis with structure
constants known modulo \(p^k\), together with representatives of \(\eta\) and
\(\gamma\) modulo \(p^k\), so that multiplication, trace, reduction, and equality
in every quotient \(A/p^jA\) with \(j\le k\) are exact finite operations.
Equivalently, the statement may be read as a certified finite procedure once
exact arithmetic in these quotients is available.  From the data
\((A,\eta,\gamma,p,k)\) in this sense, with \(A\) a finite \'etale cubic
\(\Z_p\)-algebra and \(p\ge5\), the recursion
\eqref{eq:newton-hensel-recursion} gives all residue classes \(n\) modulo
\(Pp^{k-1}\) for which \(T_n\equiv0\pmod {p^k}\) when \(\gamma\) is primitive.
If \(\gamma\) is nonprimitive, Lemma \ref{lem:primitive-reduction} is applied
first: when \(\gamma=0\), or when \(\gamma\in p^sA\) for the divisibility
exponent \(s\) of Lemma~\ref{lem:primitive-reduction} and \(k\le s\), the
homogeneous problem has an all-solution output.  When \(k>s\), write
\(\gamma=p^s\gamma_0\) with \(\gamma_0\notin pA\), set
\[
        k_0=k-s,
\]
and run the primitive recursion for \(\gamma_0\) only to precision \(k_0\).  A
primitive residue
\[
        t\equiv t_0\pmod {p^{k_0-1}}
\]
then represents, in the original modulus \(p^{k-1}\), the inflated family
\[
        t=t_0+p^{k_0-1}w\pmod {p^{k-1}},
        \qquad w\in\Z/p^s\Z .
\]
Thus each primitive residue class at precision \(k_0\) gives \(p^s\) residue
classes modulo \(p^{k-1}\), unless \(s=0\), in which case this is the original
primitive output.  In the primitive case we use the same notation with
\(s=0\), \(\gamma_0=\gamma\), \(T_a^{(0)}=T_a\), and \(F_a^{(0)}=F_a\).  In the
reduced primitive problem write
\[
\begin{aligned}
        x_a^{(0)}&=\bar\gamma_0\bar\eta^a,\\
        Z_p^{(0)}(1)&=\{a\bmod P:\Tr_{\overline A/\F_p}(x_a^{(0)})=0\},\\
        d_a&=\Tr_{\overline A/\F_p}(x_a^{(0)}\omega).
\end{aligned}
\]

If the primitive-tangent hypothesis \(1,\omega,\omega^2\) holds, the output can
be recorded by finitely many branch descriptors of the following types.  For a
primitive input put \(k_0=k\).  In the nonprimitive case \(k>s\), every non-all
descriptor below is interpreted at the reduced precision \(k_0=k-s\) for
\(\gamma_0\) and then inflated by the rule above.  The singular descriptors are
precision-indexed: at reduced precision \(k_0=1\), every primitive class
\(a\in Z_p^{(0)}(1)\) is retained modulo \(p\) and no obstruction or quadratic
root test is applied; at reduced precision \(k_0=2\), a singular class with
\(T_a^{(0)}/p\equiv0\pmod p\) represents all values \(t\bmod p\); only for
\(k_0\ge3\) do the roots of \(Q_a\) restrict the first \(t\)-digit.
\begin{enumerate}[label=\textup{(\alph*)}]
\item an all-solution descriptor, occurring when \(\gamma=0\) or when
\(\gamma\in p^sA\) for the divisibility exponent \(s\) and \(k\le s\),
representing every class modulo \(Pp^{k-1}\);
\item a dead modulo \(p\) descriptor for a class \(a\notin Z_p^{(0)}(1)\);
\item a transverse simple descriptor \((a,d_a,\tau_a\bmod p^{k_0-1})\), with
\(d_a\ne0\), giving the primitive valuation formula
\(v_p(F_a^{(0)}(t))=1+v_p(t-\tau_a)\) to precision \(k_0\), and hence the
original congruence classes by inflation to modulus \(p^{k-1}\);
\item for \(k_0\ge2\), a singular lower-obstruction descriptor for a primitive
singular class with \(T_a^{(0)}/p\not\equiv0\pmod p\), which dies at reduced
precision \(p^2\);
\item a surviving primitive singular descriptor, defined only after
\(T_a^{(0)}/p\equiv0\pmod p\).  If \(k_0=2\), it represents all classes
\(t\bmod p\) above \(a\).  If \(k_0\ge3\), compute \(Q_a\): if \(Q_a\) has no
root, the class dies before reduced precision \(p^3\); if \(r\) is a simple
root, the descriptor
\[
        (a,r,\tau_{a,r}\bmod p^{k_0-2})
\]
represents precisely the classes
\(t\equiv\tau_{a,r}\pmod {p^{k_0-2}}\) modulo \(p^{k_0-1}\), using the valuation
formula \(v_p(F_a^{(0)}(t))=2+v_p(t-\tau_{a,r})\); if \(Q_a\) has a double root
\(r\), the finite quadratic Weierstrass descriptor is
\[
        (a,r,W_{a,r}\bmod p^{k_0-2}),
        \qquad W_{a,r}(Y)=Y^2+b_{a,r}Y+c_{a,r},
\]
where \(W_{a,r}\) is the distinguished factor of
Theorem~\ref{thm:quadratic-singular-classification}; it represents precisely
the classes
\[
        t=\tilde r+Y\pmod {p^{k_0-1}},
        \qquad Y\in p\Z_p,
        \qquad W_{a,r}(Y)\equiv0\pmod {p^{k_0-2}} .
\]
\end{enumerate}
In particular, for this cubic \'etale class the computation of all local zeros
to any fixed precision is unconditional and terminates by finite exact
arithmetic; no global decidability input or \(p\)-adic Schanuel hypothesis is
involved.
\end{corollary}

\begin{proof}
For primitive \(\gamma\), the recursion tests exactly the \(p\) possible next
digits above every surviving class, so induction on \(k\) proves correctness.
The all-solution descriptor is exactly the \(\gamma=0\) or \(k\le s\) case of
Lemma \ref{lem:primitive-reduction}.  When \(k>s\), the same lemma replaces
\(\gamma\) by the primitive coefficient \(\gamma_0\) and lowers the required
precision to \(k_0=k-s\).  The primitive congruence
\(T_n^{(0)}\equiv0\pmod {p^{k_0}}\) fixes the branch parameter only modulo
\(p^{k_0-1}\).  The original modulus for \(t\) is \(p^{k-1}=p^{k_0-1}p^s\), so
the remaining \(s\) digits are free; this is precisely the displayed inflation
rule.

Classes outside the reduced \(Z_p^{(0)}(1)\) die modulo \(p\).  The transverse
simple descriptors are the Hensel factorizations in
Theorem~\ref{thm:general-local-skolem}, applied to the primitive branch.  Under
the primitive-tangent basis hypothesis, a singular class has no first-digit
constraint at precision \(p\).  It is tested against the lower obstruction only
when the requested reduced precision is at least \(p^2\): if
\(T_a^{(0)}/p\not\equiv0\pmod p\), it has no solution modulo \(p^2\); if
\(T_a^{(0)}/p\equiv0\pmod p\), then every \(t\bmod p\) survives modulo \(p^2\).
Only at precision \(p^3\) does Theorem~\ref{thm:quadratic-singular-classification}
apply the quadratic polynomial \(Q_a\) to the first \(t\)-digit.  Its roots give
exactly the reduced \(p^3\) alternatives: no root means death before \(p^3\), a
simple root gives the valuation formula
\(v_p(F_a^{(0)}(t))=2+v_p(t-\tau_{a,r})\), and a double root is represented, for
\(k_0\ge3\), by the corresponding distinguished quadratic factor \(W_{a,r}\)
with coefficients truncated modulo \(p^{k_0-2}\).  Inflation then converts the
reduced primitive output into the complete set of classes modulo \(Pp^{k-1}\)
for the original nonprimitive congruence.
\end{proof}

\section{The codifferent singular line and exact branch census}\label{sec:codifferent-census}

The local branch theorem identifies singular first-order classes by two trace
conditions.  In a cubic algebra generated by the logarithmic tangent, that
singular line has a canonical closed form: it is the codifferent line attached
to the basis \(1,\omega,\omega^2\).  The codifferent and trace-dual basis
formalism is standard algebraic number theory \cite[Chapter~III]{Neukirch1999};
the point here is its use as an exact branch-census coordinate system.  This
gives an exact branch census in full norm-fiber orbits.

\begin{theorem}[Trace-dual basis and the codifferent line]
\label{thm:codifferent-singular-line}
Let \(F\) be a field of characteristic different from \(2\) and \(3\), let
\(B/F\) be a finite \'etale cubic algebra, and let \(\omega\in B\) generate
\(B\) as an \(F\)-algebra.  Let \(m_\omega:B\to B\) be multiplication by \(\omega\), and write
\[
        f_\omega(T)=\det(T\cdot\operatorname{id}_B-m_\omega)
        =T^3+aT^2+bT+d\in F[T].
\]
Because \(\omega\) generates \(B\), this is the monic separable generator
polynomial such that \(B\simeq F[T]/(f_\omega)\) and \(T\mapsto\omega\).  Let \(z_0,z_1,z_2\) be the basis trace-dual to
\(1,\omega,\omega^2\), so
\[
        \Tr(\omega^i z_j)=\delta_{ij}\qquad(0\le i,j\le2).
\]
Then
\[
        z_0=\frac{\omega^2+a\omega+b}{f_\omega'(\omega)},
        \qquad
        z_1=\frac{\omega+a}{f_\omega'(\omega)},
        \qquad
        z_2=\frac1{f_\omega'(\omega)}.
\]
Consequently, for every \(s,t\in F\),
\[
        \{x\in B:\Tr(x)=s,\ \Tr(\omega x)=t\}=s z_0+t z_1+Fz_2.
\]
In particular,
\[
        L_\omega:=\{x\in B:\Tr(x)=0,
        \Tr(\omega x)=0\}=F\cdot f_\omega'(\omega)^{-1}.
\]
Moreover
\[
        \Norm\bigl(f_\omega'(\omega)\bigr)=-\disc(f_\omega),
        \qquad
        \Norm\bigl(z_2\bigr)=-\disc(f_\omega)^{-1}.
\]
\end{theorem}

\begin{proof}
After extending scalars to a separable closure, write the three roots of
\(f_\omega\) as \(\omega_1,\omega_2,\omega_3\).  For every polynomial
\(g(T)\) of degree at most \(2\), Lagrange interpolation gives
\[
        \Tr\left(\frac{g(\omega)}{f_\omega'(\omega)}\right)
        =\sum_{i=1}^3\frac{g(\omega_i)}{f_\omega'(\omega_i)}
        =[T^2]g(T),
\]
where \([T^2]\) denotes the coefficient of \(T^2\).  This proves the formula
for \(z_2\), because \([T^2](1)=0\), \([T^2](T)=0\), and \([T^2](T^2)=1\).
For \(z_1\), take \(g(T)=T+a\).  Then
\([T^2]g=0\), \([T^2](Tg)=1\), and
\[
        T^2g(T)=T^3+aT^2\equiv -bT-d\pmod {f_\omega},
\]
so \([T^2](T^2g\bmod f_\omega)=0\).  Thus
\(z_1=(\omega+a)/f_\omega'(\omega)\).  For \(z_0\), take
\(g(T)=T^2+aT+b\).  Then \([T^2]g=1\), while
\[
        Tg(T)=T^3+aT^2+bT\equiv -d\pmod {f_\omega}
\]
has no \(T^2\)-term.  A second reduction using
\(T^3\equiv-aT^2-bT-d\) gives that \(T^2g(T)\bmod f_\omega\) also has no
\(T^2\)-term.  Hence \(z_0\) is as claimed.

The affine-line formula follows immediately: if
\(x=\lambda_0z_0+\lambda_1z_1+\lambda_2z_2\), then
\(\Tr(x)=\lambda_0\) and \(\Tr(\omega x)=\lambda_1\).  The homogeneous
singular line is the case \(s=t=0\).  Finally,
\[
        \Norm(f_\omega'(\omega))
        =\prod_i\prod_{j\ne i}(\omega_i-\omega_j)
        =-\prod_{i<j}(\omega_i-\omega_j)^2
        =-\disc(f_\omega),
\]
because the degree is \(3\).  Inverting gives the norm of \(z_2\).
\end{proof}

\begin{theorem}[Affine finite-field branch census in full norm fibers]
\label{thm:affine-finite-field-branch-census}
Let \(B/F\) be a finite \'etale cubic algebra over a finite field \(F=\F_q\) of
characteristic different from \(2\) and \(3\).  Let \(\gamma\in B^\times\), let
\(\omega\in B\) generate \(B\), let \(s\in F\), and let \(C\subset F^\times\) be
nonempty.  For \(\delta\in C\), put
\[
        N_\delta=\Norm(\gamma)\delta .
\]
No smoothness assumption is needed for the formal count below: every occurrence
of \(N_B(s,N_\delta)\) denotes the actual affine count of
Definition~\ref{def:NB-global}.  Smooth summands \(s^3\ne27N_\delta\) may be
evaluated immediately by Theorem~\ref{thm:prescribed-trace-norm-s3}.  Nodal
summands are left as actual affine counts in this theorem; their closed
evaluation is not used in the proof of this formal branch-census identity.
If one wants a statement purely in closed elliptic form at this stage, assume
all fibers \(s^3\ne27N_\delta\) are smooth.  Define
\[
        X_{s,C,\gamma}
        =\{h\in B^\times:\Norm(h)\in C,
        \ \Tr(\gamma h)=s\},
\]
and the singular subset
\[
        X_{s,C,\gamma}^{\rm sing}
        =\{h\in X_{s,C,\gamma}:\Tr(\omega\gamma h)=0\}.
\]
Let \(z_0,z_1,z_2\) be the trace-dual basis of
Theorem~\ref{thm:codifferent-singular-line}.  Then
\[
        \#X_{s,C,\gamma}=\sum_{\delta\in C} N_B(s,N_\delta),
\]
where \(N_B(s,n)\) denotes the actual affine count of
Definition~\ref{def:NB-global}.  In smooth summands this count is evaluated by
Theorem~\ref{thm:prescribed-trace-norm-s3}; nodal summands remain formal actual
counts here.  Moreover
\[
        \#X_{s,C,\gamma}^{\rm sing}
        =\sum_{\delta\in C}
        \#\{u\in F:\Norm(s z_0+u z_2)=N_\delta\}.
\]
For a singular point represented by
\[
        x=\gamma h=s z_0+u z_2,
\]
the second trace tangent is exactly
\[
        \Tr(x\omega^2)=u.
\]
Thus \(u\ne0\) is exactly the finite-field nondegeneracy condition for the
second trace tangent.  When \(F=\F_p\) and these reduced data arise from a
finite \'etale \(\Z_p\)-algebra with logarithmic tangent reducing to \(\omega\),
such classes are precisely the nondegenerate quadratic singular classes of
Corollary~\ref{cor:affine-nondegenerate-singular}.  The only possible
degenerate affine singular class in a given norm fiber is the single parameter
value \(u=0\), and it occurs precisely when
\[
        \Norm(s z_0)=N_\delta.
\]
\end{theorem}

\begin{proof}
Multiplication by \(\gamma\) gives a bijection from \(h\) to \(x=\gamma h\).
The condition \(\Norm(h)=\delta\) becomes \(\Norm(x)=N_\delta\), and the target
condition becomes \(\Tr(x)=s\).  This proves the total count by decomposing by
norm and using Definition~\ref{def:NB-global}.  When a summand is smooth, the
explicit elliptic formula of Theorem~\ref{thm:prescribed-trace-norm-s3} may be
inserted immediately; no closed nodal evaluation is required for the formal
identity proved here.

The singular condition adds \(\Tr(\omega x)=0\).  By
Theorem~\ref{thm:codifferent-singular-line}, the simultaneous affine trace
conditions are exactly
\[
        x=s z_0+u z_2,
        \qquad u\in F.
\]
Imposing \(\Norm(x)=N_\delta\) gives the displayed one-variable cubic equation.
Finally, trace-duality gives
\[
        \Tr((s z_0+u z_2)\omega^2)=u,
\]
because \(\Tr(z_0\omega^2)=0\) and \(\Tr(z_2\omega^2)=1\).  Thus \(u\ne0\)
is exactly the finite-field form of nonzero second tangent.  In the special case
where \(F=\F_p\) and the finite-field data are the reduction of local
\(p\)-adic data, this is the hypothesis needed to invoke
Corollary~\ref{cor:affine-nondegenerate-singular}; otherwise the assertion here
is only the stated finite-field nondegeneracy condition.  The isolated
\(u=0\) case is the only possible degenerate affine singular class.
\end{proof}

\begin{corollary}[Homogeneous specialization]
\label{cor:homogeneous-cubic-singular-equation}
In Theorem~\ref{thm:affine-finite-field-branch-census}, take \(s=0\) and a
single norm fiber \(C=\{\delta\}\).  If
\(D_\omega=\disc(f_\omega)\), then the singular count is
\[
        \#\{u\in F^\times:u^3=-\Norm(\gamma)\delta D_\omega\}.
\]
\end{corollary}

\begin{proof}
For \(s=0\), the singular line is \(Fz_2\), and the norm equation is
\[
        u^3\Norm(z_2)=\Norm(\gamma)\delta.
\]
Theorem~\ref{thm:codifferent-singular-line} gives
\(\Norm(z_2)=-D_\omega^{-1}\), hence
\(u^3=-\Norm(\gamma)\delta D_\omega\).  Since the right side is nonzero,
\(u\) is automatically nonzero.
\end{proof}

\begin{theorem}[Finite-field branch census in full norm fibers]\label{thm:finite-field-branch-census}
Let \(B/\F_q\) be a finite \'etale cubic algebra, with
\(\operatorname{char}\F_q\ne2,3\).  Let \(\gamma\in B^\times\), let
\(\omega\in B\) generate \(B\), and let \(C\subset\F_q^\times\) be a nonempty
set.  Define
\[
        X_{C,\gamma}
        =\{h\in B^\times:\Norm(h)\in C,
        \ \Tr(\gamma h)=0\}.
\]
Let \(f_\omega\) be as in Theorem~\ref{thm:codifferent-singular-line}, and put
\(D_\omega=\disc(f_\omega)\).  The singular subset
\[
        X_{C,\gamma}^{\rm sing}
        =\{h\in X_{C,\gamma}:\Tr(\omega\gamma h)=0\}
\]
has cardinality
\[
        \#X_{C,\gamma}^{\rm sing}
        =\sum_{\delta\in C}
        \#\{u\in\F_q^\times:
        u^3=-\Norm(\gamma)\delta D_\omega\}.
\]
The total cardinality is
\[
        \#X_{C,\gamma}=
        \sum_{\delta\in C}N_B\bigl(\Norm(\gamma)\delta\bigr),
\]
where \(N_B(\eps)\) is the trace-zero norm count of
Corollary~\ref{thm:all-etale-counts}.  Hence the transverse count is
\[
        \#X_{C,\gamma}^{\rm tr}
        =\sum_{\delta\in C}N_B\bigl(\Norm(\gamma)\delta\bigr)
        -\sum_{\delta\in C}\#\{u\in\F_q^\times:
        u^3=-\Norm(\gamma)\delta D_\omega\}.
\]
For a single norm fiber \(C=\{\delta\}\), the singular count is
\[
\begin{cases}
1, & q\equiv2\pmod3,\\
3, & q\equiv1\pmod3\text{ and }-\Norm(\gamma)\delta D_\omega
        \in(\F_q^\times)^3,\\
0, & q\equiv1\pmod3\text{ and }-\Norm(\gamma)\delta D_\omega
        \notin(\F_q^\times)^3.
\end{cases}
\]
\end{theorem}

\begin{proof}
Multiplication by \(\gamma\) gives a bijection from \(h\) to
\(x=\gamma h\).  The trace condition becomes \(\Tr(x)=0\), and the norm
condition \(\Norm(h)=\delta\) becomes
\[
        \Norm(x)=\Norm(\gamma)\delta.
\]
Therefore the total count is the sum of the corresponding trace-zero norm
counts.  For the singular subset one imposes also \(\Tr(\omega x)=0\), so
\(x\in L_\omega\).  By Theorem~\ref{thm:codifferent-singular-line}, every
nonzero singular \(x\) has the form
\[
        x=u f_\omega'(\omega)^{-1},\qquad u\in\F_q^\times.
\]
Taking norms gives
\[
        \Norm(x)=u^3\Norm(f_\omega'(\omega)^{-1})
        =-u^3D_\omega^{-1}.
\]
The equation \(\Norm(x)=\Norm(\gamma)\delta\) is therefore
\[
        u^3=-\Norm(\gamma)\delta D_\omega.
\]
This proves the singular formula.  The final alternatives are the elementary
fiber sizes of the cube map on \(\F_q^\times\).
\end{proof}

\begin{corollary}[Exact affine local branch census in a full norm-fiber orbit]
\label{cor:exact-affine-local-branch-census}
Let \(p\ge5\), let \(A\) be a finite \'etale cubic \(\Z_p\)-algebra, and put
\(B=A/pA\).  Let \(\eta\in A^\times\), let \(P=\ord(\bar\eta)\), write
\(\eta^P=1+pU\), and let \(\omega=\bar U\in B\).  Assume that \(\omega\)
generates \(B\), that \(\bar\gamma\in B^\times\), and that the reduced orbit
\(H=\langle\bar\eta\rangle\) is a full union of norm fibers
\[
        H=\{h\in B^\times:\Norm(h)\in C\}
\]
for some nonempty \(C\subset\F_p^\times\).  Let \(c\in\Z_p\) and write
\(s=\bar c\).  For \(\delta\in C\), put
\(N_\delta=\Norm(\bar\gamma)\delta\).  This is a formal count identity.  If all fibers satisfy \(s^3\ne27N_\delta\),
the counts below are evaluated by the smooth formula of
Theorem~\ref{thm:prescribed-trace-norm-s3}.  In the mixed smooth/nodal case,
\(N_B(s,N_\delta)\) is interpreted as the actual affine count from
Definition~\ref{def:compactified-prescribed-trace-norm}; no closed nodal
formula is needed for the proof of this corollary.
Then the number of mod-\(p\) target classes
\[
        \Tr_A(\gamma\eta^n)\equiv c\pmod p
\]
in one period is
\[
        M_{p,c}=\sum_{\delta\in C}N_B(s,N_\delta).
\]
The first-order singular classes have number
\[
        S_{p,c}=\sum_{\delta\in C}
        \#\{u\in\F_p:\Norm(s z_0+u z_2)=N_\delta\},
\]
where \(z_0,z_2\) are the trace-dual elements attached to \(\omega\) in
Theorem~\ref{thm:codifferent-singular-line}.  Exactly \(M_{p,c}-S_{p,c}\)
classes are transverse, and each transverse class lifts to a unique simple
\(p\)-adic branch with the exact valuation formula of
Theorem~\ref{thm:affine-local-branch}.  A singular class corresponding to
\(u\ne0\) is a nondegenerate affine singular class: it must first pass the
lower-obstruction test
\[
        \frac{\Tr_A(\gamma\eta^a)-c}{p}\equiv0\pmod p,
\]
and only after that obstruction vanishes is it governed by the quadratic Hensel
polynomial of Corollary~\ref{cor:affine-nondegenerate-singular}.  A singular class
corresponding to \(u=0\), if it occurs, is the unique degenerate affine singular
class in that norm fiber and is still governed by the finite digit recursion of
Theorem~\ref{thm:affine-local-branch}.
\end{corollary}

\begin{proof}
The mod-\(p\) target classes are the elements \(h=\bar\eta^a\in H\) satisfying
\(\Tr(\bar\gamma h)=s\).  The total and singular counts are exactly
Theorem~\ref{thm:affine-finite-field-branch-census}.  A class outside the
singular subset has nonzero derivative
\(\Tr(\bar\gamma\bar\eta^a\omega)\), so Theorem~\ref{thm:affine-local-branch}
gives the simple branch and valuation formula.  For a singular class, write
\(x=\bar\gamma\bar\eta^a=s z_0+u z_2\).  Theorem~\ref{thm:affine-finite-field-branch-census}
shows that \(\Tr(x\omega^2)=u\).  If \(u\ne0\), the affine nondegenerate
quadratic theorem applies after the lower obstruction in
Corollary~\ref{cor:affine-nondegenerate-singular} has vanished; if that lower
obstruction does not vanish, the class dies modulo \(p^2\).  If \(u=0\), the
class is exactly the degenerate case identified in
Remark~\ref{rem:affine-degenerate-singular}; the general finite recursion still
applies.
\end{proof}

\begin{corollary}[Exact local branch census in a full norm-fiber orbit]
\label{cor:exact-local-branch-census}
Let \(p\ge5\), let \(A\) be a finite \'etale cubic \(\Z_p\)-algebra, and put
\(B=A/pA\).  Let \(\eta\in A^\times\), let \(P=\ord(\bar\eta)\), write
\(\eta^P=1+pU\), and let \(\omega=\bar U\in B\).  Assume that \(\omega\)
generates \(B\), that \(\bar\gamma\in B^\times\), and that the reduced orbit
\(H=\langle\bar\eta\rangle\) is a full union of norm fibers
\[
        H=\{h\in B^\times:\Norm(h)\in C\}
\]
for some nonempty \(C\subset\F_p^\times\).  Then the mod-\(p\) zero classes in
one period have total number
\[
        M_p=\sum_{\delta\in C}N_B\bigl(\Norm(\bar\gamma)\delta\bigr),
\]
and the first-order singular classes have number
\[
        S_p=\sum_{\delta\in C}
        \#\{u\in\F_p^\times:
        u^3=-\Norm(\bar\gamma)\delta\disc(f_\omega)\}.
\]
Exactly \(M_p-S_p\) classes are transverse, and each transverse class lifts to
a unique simple \(p\)-adic branch with the exact valuation formula of
Theorem~\ref{thm:general-local-skolem}.  The remaining \(S_p\) classes are
precisely the primitive first-order singular classes.  Each is governed by
Theorem~\ref{thm:quadratic-singular-classification}: it may die at the lower obstruction modulo \(p^2\), and, if it survives to modulo \(p^2\), it is governed by the quadratic Hensel polynomial of that theorem.  More
explicitly, if the singular reduced point is
\[
        x_a=\bar\gamma\bar\eta^a=u f_\omega'(\omega)^{-1},
\]
then its second tangent coefficient is
\[
        \Delta_a=\Tr_B(x_a\omega^2)=u.
\]
\end{corollary}

\begin{proof}
The mod-\(p\) zero classes are the elements \(h=\bar\eta^a\in H\) satisfying
\(\Tr(\bar\gamma h)=0\).  The total and singular counts are therefore exactly
Theorem~\ref{thm:finite-field-branch-census} with \(q=p\).  A class not in the
singular subset has nonzero derivative
\(\Tr(\bar\gamma\bar\eta^a\omega)\), so Theorem~\ref{thm:general-local-skolem}
gives a unique simple branch and the valuation formula.  A singular class is
nonzero because \(\bar\gamma\) and \(h\) are units.  Since \(\omega\) generates
\(B\), the basis condition \(1,\omega,\omega^2\) holds, and
Theorem~\ref{thm:quadratic-singular-classification} applies with its lower-obstruction and, when applicable, quadratic alternatives.  The formula for
\(\Delta_a\) is the last trace identity in
Theorem~\ref{thm:codifferent-singular-line}, multiplied by the scalar \(u\).
\end{proof}

\begin{corollary}[Supersingular full-fiber census]\label{cor:supersingular-full-fiber-census}
In the situation of Corollary~\ref{cor:exact-local-branch-census}, assume
\(p\equiv2\pmod3\).  Then every norm fiber in \(C\) contributes exactly one
singular class.  In particular, for a full norm-one orbit in the inert field
case \(B=\F_{p^3}\), the total number of mod-\(p\) zero classes is \(p+1\),
exactly one is singular, and exactly \(p\) are transverse.
\end{corollary}

\begin{proof}
When \(p\equiv2\pmod3\), the cube map on \(\F_p^\times\) is bijective, so each
norm fiber contributes one singular class by
Theorem~\ref{thm:finite-field-branch-census}.  In the inert field case,
Corollary~\ref{cor:supersingular} gives \(N_{\F_{p^3}}(\eps)=p+1\) for every
\(\eps\in\F_p^\times\), and the norm-one orbit has \(|C|=1\).
\end{proof}

\section{Uniform subgroup orbits: smooth fibers, trace complexes, and nodal boundary}
\label{sec:proper-subgroups}

The exact branch census above is strongest when the reduced orbit is a full norm
fiber.  We now prove the complementary finite-field theorem for arbitrary
subgroups of the norm-one torus.  The proof is written in sheaf language only
in this section.  Its purpose is concrete: the subgroup character sums are
Frobenius traces of the relative complexes
\[
        R\tau_!\mathcal L_\chi,
\]
and the smooth and nodal estimates become stalk computations for these
complexes.  The local \(p\)-adic branch theory in the preceding sections remains
entirely explicit.

Throughout this section \(F=\F_q\) has characteristic different from \(2\) and
\(3\), and \(B/F\) is a finite \'etale cubic algebra.  Here \(B^\times\) denotes
the algebraic \(F\)-torus \(\operatorname{Res}_{B/F}\Gm\), whose \(F\)-points are
the usual unit group of the algebra \(B\).  We write
\[
        T_B=\ker\bigl(\Norm_{B/F}:\operatorname{Res}_{B/F}\Gm\to \Gm\bigr).
\]
After base change to \(\overline F\), this is the product-one subtorus of
\((\Gm)^3\); in particular \(T_B\) is a connected two-dimensional \(F\)-torus.
For a subgroup \(H\subset T_B(F)\) of index \(m\), let
\[
        H^\perp=\{\chi:T_B(F)\to\C^\times:\chi|_H=1\}
\]
be the annihilator of \(H\).  Since \(T_B(F)\) is a subgroup of a product of
finite-field unit groups, every character considered in this section has order
prime to \(\operatorname{char}F\).  We identify complex-valued finite characters
with \(\overline{\Q}_\ell\)-valued characters after fixing an isomorphism
\(\overline{\Q}_\ell\simeq\C\), where
\(\ell\ne\operatorname{char}F\).

For \(n\in F^\times\), set
\[
        X_n=\{x\in B^\times:\Norm_{B/F}(x)=n\},
        \qquad
        \tau_n:X_n\longrightarrow\mathbb A^1_F,
        \quad x\longmapsto\Tr_{B/F}(x).
\]
Thus \(X_n\) is a torsor under \(T_B\).  Its fiber over \(s\) is the affine
prescribed trace/norm curve
\[
        U_{s,n,B}:\quad \Tr(x)=s,
        \qquad \Norm(x)=n .
\]
If \(\gamma\in B^\times\), multiplication by \(\gamma^{-1}\) identifies
\(X_{\Norm(\gamma)}\) with \(T_B\).

\medskip
\noindent\emph{Proof architecture.}
The smooth theorem uses four inputs: (1) the Lang character-sheaf dictionary on
\(T_B\); (2) the torsor trivialization of \(X_{\Norm(\gamma)}\) by \(\gamma\),
which is the only pullback used in the character sums; (3) a Picard--Kummer
kernel calculation on the smooth trace/norm curve after geometric splitting;
and (4) the Grothendieck--Ogg--Shafarevich and Deligne estimates for the
resulting rank-one sheaf on a genus-one curve with three punctures.  The nodal
theorem repeats the same plan on the normalization, where the smooth
Picard--Kummer kernel degenerates to a cyclic cubic kernel.

\begin{lemma}[Lang/Kummer dictionary for the norm-one torus]
\label{lem:character-sheaf-dictionary}
Let \(\chi:T_B(F)\to\overline{\Q}_\ell^\times\) be a finite character.  Let
\(\operatorname{Frob}_q\) denote arithmetic Frobenius.  The Lang map
\[
        L:T_B\longrightarrow T_B,
        \qquad L(t)=\operatorname{Frob}_q(t)t^{-1},
\]
is finite \'etale, with deck group \(T_B(F)\).  The direct image
\(L_*\overline{\Q}_\ell\) decomposes into rank-one summands indexed by
characters of \(T_B(F)\).  We denote by \(\mathcal L_\chi\) the summand
normalized so that, for every \(h\in T_B(F)\), the trace of geometric Frobenius
on the fiber at \(h\) is \(\chi(h)\).

After base change to \(\overline F\) and after choosing an ordering of the three
geometric embeddings of \(B\), the torus is
\[
        T_{\overline F}: t_1t_2t_3=1 .
\]
If \(N\) is the order of \(\chi\), then \(N\) is prime to
\(\operatorname{char}F\), and the geometric base change of
\(\mathcal L_\chi\) is a finite-order Kummer sheaf represented by an exponent
class
\[
        (a_1,a_2,a_3)
        \in(\Z/N\Z)^3/(\Z/N\Z)(1,1,1).
\]
In split coordinates this is the Kummer class of
\(t_1^{a_1}t_2^{a_2}t_3^{a_3}\).  The exact geometric Kummer order is the order of this exponent class; it divides \(N\), and in the Picard--Kummer kernel propositions we take the modulus to be that exact geometric order unless an auxiliary \(N\)-torsion exponent group is explicitly specified.  The diagonal class is trivial on
\(t_1t_2t_3=1\).  Conversely, if the geometric exponent class is diagonal,
then \(\mathcal L_\chi\) is geometrically constant; with the normalization
above, this forces \(\chi=1\).

For \(\gamma\in B^\times\), define the pullback sheaf on
\(X_{\Norm(\gamma)}\) by
\[
        \mathcal L_{\chi,\gamma}
        =(x\mapsto \gamma^{-1}x)^*\mathcal L_\chi .
\]
Its trace function on \(X_{\Norm(\gamma)}(F)\) is
\[
        x\longmapsto \chi(\gamma^{-1}x).
\]
\end{lemma}

\begin{proof}
Lang's theorem for connected algebraic groups over finite fields gives the
surjectivity and finite \'etaleness of \(L\), and identifies its deck group with
\(T_B(F)\) \cite{Lang1956}.  Since this deck group is finite abelian,
\(L_*\overline{\Q}_\ell\) decomposes into its isotypic rank-one summands.  We
choose the \(\chi\)- or \(\chi^{-1}\)-summand according to the convention that
geometric Frobenius has trace \(\chi(h)\) at \(h\in T_B(F)\).  This is the usual
Lang character-sheaf construction; for the function-sheaf dictionary for
characters of commutative algebraic groups and the Kummer description on tori,
see \cite{KatzSommes,KatzGKM1988,CunninghamRoe2021}.

Over \(\overline F\), the torus is split.  Finite-order tame rank-one local
systems on \((\Gm)^3\) are obtained from the Kummer covers
\(u_i^N=t_i\).  Restricting to the product-one subtorus quotients the exponent
lattice by the diagonal relation, because \((t_1t_2t_3)^a=1\) on
\(T_{\overline F}\).  This gives the displayed exponent class.  If that class is
zero modulo the diagonal, the geometric sheaf is constant.  A rational Lang
sheaf that is geometrically constant has constant Frobenius trace on the
connected torus; evaluating at the identity gives the constant \(\chi(1)=1\),
so \(\chi\) is trivial.  The final statement follows by pulling back along the
torsor trivialization \(x\mapsto\gamma^{-1}x\).
\end{proof}

\begin{definition}[The relative trace complex]
\label{def:relative-trace-complex}
For \(\gamma\in B^\times\) and a character
\(\chi:T_B(F)\to\overline{\Q}_\ell^\times\), define
\[
        K_{\chi,\gamma}
        =R\tau_{\Norm(\gamma),!}\mathcal L_{\chi,\gamma}
        \in D^b_c(\mathbb A^1_F,\overline{\Q}_\ell).
\]
By compact-support base change, its stalk at \(s\in F\) is
\[
        (K_{\chi,\gamma})_{\bar s}
        \simeq R\Gamma_c\bigl(U_{s,\Norm(\gamma),B,\overline F},
        \mathcal L_{\chi,\gamma}\bigr).
\]
\end{definition}

\begin{proposition}[Stalk-trace form of the subgroup character sums]
\label{prop:relative-trace-stalk}
For every \(s\in F\), let \(\operatorname{Fr}_s\) denote geometric Frobenius at
the rational point \(s\) of \(\mathbb A^1_F\).  In the first displayed
formula below, the trace on the derived stalk means the alternating Frobenius
trace on its cohomology sheaves.  Thus
\[
        S_\chi(s;\gamma)
        :=\sum_{\substack{h\in T_B(F)\\ \Tr(\gamma h)=s}}\chi(h)
        =\sum_i(-1)^i\operatorname{Tr}\left(\operatorname{Fr}_s\mid
        H^i((K_{\chi,\gamma})_{\bar s})\right).
\]
Equivalently,
\[
        S_\chi(s;\gamma)
        =\sum_i(-1)^i
        \operatorname{Tr}\left(\operatorname{Fr}_q\mid
        H^i_c\bigl(U_{s,\Norm(\gamma),B,\overline F},
        \mathcal L_{\chi,\gamma}\bigr)\right).
\]
\end{proposition}

\begin{proof}
The first formula is the definition of the stalk of
\(R\tau_{\Norm(\gamma),!}\mathcal L_{\chi,\gamma}\).  The second formula is the
Grothendieck--Lefschetz trace formula for compactly supported cohomology
\cite[Chapter~VI]{MilneEtale1980}.  By
Lemma~\ref{lem:character-sheaf-dictionary}, the trace of
\(\mathcal L_{\chi,\gamma}\) at \(x=\gamma h\) is \(\chi(h)\), giving exactly
the displayed finite sum.
\end{proof}

\begin{proposition}[Character decomposition for norm-torus subgroup cosets]
\label{prop:proper-subgroup-character-decomposition}
Let \(g\in T_B(F)\), let \(\gamma\in B^\times\), and let \(s\in F\).  Define
\[
        N_{gH,B}(s;\gamma)
        =\#\{h\in gH:\Tr_{B/F}(\gamma h)=s\}.
\]
Then
\[
        N_{gH,B}(s;\gamma)
        =\frac1m\sum_{\chi\in H^\perp}\chi(g^{-1})S_\chi(s;\gamma).
\]
The trivial-character term is
\[
        S_1(s;\gamma)=N_B(s,\Norm(\gamma)),
\]
where \(N_B(s,n)\) is the actual affine count of
Definition~\ref{def:NB-global}; the smooth case is evaluated by
Theorem~\ref{thm:prescribed-trace-norm-s3}, and the nodal case by
Proposition~\ref{prop:full-nodal-count}.
\end{proposition}

\begin{proof}
The indicator function of the coset \(gH\) inside \(T_B(F)\) is
\[
        1_{gH}(h)=\frac1m\sum_{\chi\in H^\perp}\chi(g^{-1}h).
\]
Multiplying by the indicator of the trace condition and summing over \(T_B(F)\)
gives the formula.  For \(\chi=1\), multiplication by \(\gamma\) sends
\(T_B(F)\) bijectively to the norm fiber \(\Norm(x)=\Norm(\gamma)\), and the
trace condition becomes \(\Tr(x)=s\).
\end{proof}

\begin{lemma}[Kummer triviality criterion on affine curves]
\label{lem:kummer-triviality-criterion}
Let \(k\) be an algebraically closed field, let \(N\) be prime to
\(\operatorname{char}k\), and let \(U\) be a smooth connected affine curve over
\(k\) with smooth compactification \(C\).  For \(f\in\mathcal O(U)^\times\), let \(Y^N=f\) be the corresponding
\(\mu_N\)-torsor on \(U\).  The torsor is trivial if and only if
\[
        f\in \mathcal O(U)^{\times N}.
\]
Equivalently, any faithful rank-one Kummer summand of exact order \(N\) attached
to this torsor is geometrically trivial if and only if \(f\) is an \(N\)-th
power.  If a nonfaithful character of exact order \(M\mid N\) is used, the same
criterion holds with \(M\) in place of \(N\).  Since \(k^\times=k^{\times N}\),
it is enough in the faithful order-\(N\) case to check that
\(f\in k(C)^{\times N}\); any such rational \(N\)-th root is automatically an
invertible regular function on \(U\), because \(f\) has zero valuation at every
point of \(U\).
\end{lemma}

\begin{proof}
Apply the Kummer exact sequence
\[
        1\longrightarrow \mu_N\longrightarrow \Gm
        \xrightarrow{N}\Gm\longrightarrow1
\]
on the \'etale site of \(U\).  The connecting map sends
\(f\in\mathcal O(U)^\times\) to the torsor \(Y^N=f\) in
\(H^1(U,\mu_N)\).  This torsor is trivial exactly when \(f\) is an \(N\)-th
power in \(\mathcal O(U)^\times\).  If \(f=g^N\) in \(k(C)^\times\), then every
valuation of \(g\) at a point of \(U\) is zero because \(f\) is a unit on \(U\);
therefore \(g\in\mathcal O(U)^\times\).  Constants cause no ambiguity because
\(k\) is algebraically closed.  See \cite[Chapter~III, Section~4]{MilneEtale1980}
for the Kummer sequence.
\end{proof}

\begin{proposition}[Picard--Kummer kernel on smooth trace/norm fibers]
\label{prop:smooth-picard-kummer-kernel}
Let \(n\in F^\times\), let \(s^3\ne27n\), and let \(C_{s,n,B}\) be the smooth
compactification of
\[
        U_{s,n,B}:\quad \Tr(x)=s,
        \qquad \Norm(x)=n .
\]
Let \(\mathcal L\) be a Lang/Kummer character sheaf on \(T_B\) of exact
geometric Kummer order \(N\), with \(N\) prime to \(\operatorname{char}F\).
After base change to \(\overline F\), choose a point
\(x_0\in X_n(\overline F)\) and identify the norm torsor \(X_n\) with
\(T_B\) by \(x\mapsto x_0^{-1}x\).  If the induced pullback of \(\mathcal L\)
to \(U_{s,n,B,\overline F}\) is geometrically trivial, then \(\mathcal L\) is
the trivial character sheaf.  The conclusion is independent of the choice of
\(x_0\), because changing \(x_0\) multiplies the split Kummer functions by
constants, and constants are \(N\)-th powers over \(\overline F\).
Equivalently, every nontrivial character sheaf on \(T_B\) pulls back
geometrically nontrivially to every smooth trace/norm fiber after any geometric
trivialization of the norm torsor.
\end{proposition}

\begin{proof}
The assertion is geometric.  We base change to \(\overline F\), choose an
ordering of the three embeddings of \(B\), and fix a geometric point
\(x_0\in X_n(\overline F)\).  The torus becomes \(t_1t_2t_3=1\), the curve is
\[
        x_1+x_2+x_3=s,
        \qquad x_1x_2x_3=n,
\]
and the Kummer functions obtained from the torsor coordinates
\(t_i=x_i/(x_0)_i\) differ from the functions in the coordinates \(x_i\) only by
nonzero constants.  Since \(\overline F\) is algebraically closed, those
constants are \(N\)-th powers and do not affect Kummer triviality.
Let \(P_1,P_2,P_3\) be the three points at infinity, labeled so that the line
\(x_i=0\) meets the projective cubic only at \(P_i\).  This is the coordinate-zero
labeling used in this proof; relative to the earlier \([X:Y:Z]\) labels it may
permute the names of the three infinity points.  In the projective plane
model, \(x_i=X_i/Z\).  The divisor of \(Z\) on the cubic is
\(P_1+P_2+P_3\).  The line \(X_i=0\) meets the projective cubic with triple
multiplicity at \(P_i\): indeed, on the line \(X_i=0\), the norm equation
restricts to \(0=nZ^3\), so the intersection divisor is \(3P_i\).  Hence
\begin{equation}\label{eq:divisor-xi}
        \operatorname{div}(x_i)=3P_i-(P_1+P_2+P_3)
        =2P_i-P_j-P_k
        \qquad(\{i,j,k\}=\{1,2,3\}).
\end{equation}

Let the geometric character have exact Kummer order \(N\), and represent its
Kummer exponent class by
\[
        \mathbf a=(a_1,a_2,a_3)
        \in(\Z/N\Z)^3/(\Z/N\Z)(1,1,1).
\]
The pullback to the affine curve is the Kummer sheaf of
\[
        f_{\mathbf a}=x_1^{a_1}x_2^{a_2}x_3^{a_3},
\]
viewed as an invertible function on \(U_{s,n,B}\).  By
Lemma~\ref{lem:kummer-triviality-criterion}, and because constants in
\(\overline F\) are \(N\)-th powers, this sheaf is trivial on \(U_{s,n,B}\)
exactly when
\(f_{\mathbf a}\in\overline F(C_{s,n,B})^{\times N}\).  Put
\(S=a_1+a_2+a_3\).  By \eqref{eq:divisor-xi},
\begin{equation}\label{eq:divisor-fa}
        \operatorname{div}(f_{\mathbf a})
        =\sum_{i=1}^3(3a_i-S)P_i .
\end{equation}
If \(f_{\mathbf a}=g^N\), then Lemma~\ref{lem:kummer-triviality-criterion}
also gives \(g\in\mathcal O(U_{s,n,B})^\times\); hence \(\operatorname{div}(g)\)
is supported only at the deleted points at infinity.  Each coefficient in
\eqref{eq:divisor-fa} is divisible by \(N\).  Write
\begin{equation}\label{eq:ki-sum-zero}
        3a_i-S=Nk_i,
        \qquad k_1+k_2+k_3=0 .
\end{equation}
Then \(D=\sum_i k_iP_i=\operatorname{div}(g)\) is principal.

Choose \(P_3\) as the origin of the elliptic curve \(C_{s,n,B}\).  The divisor
relations \eqref{eq:divisor-xi} imply
\[
        P_2-P_3=-(P_1-P_3),
        \qquad 3(P_1-P_3)=0.
\]
The three points at infinity are distinct on the smooth projective cubic, and
the Abel--Jacobi map \(P\mapsto[P-P_3]\) is injective on the points of the
curve.  Hence \(P_1-P_3\ne0\), so \(P_1-P_3\) has exact order \(3\).  Therefore the class of \(D\) in
\(\Pic^0(C_{s,n,B})\) is
\[
        [D]=(k_1-k_2)(P_1-P_3).
\]
Since \(D\) is principal, \(3\mid k_1-k_2\).  Together with
\(k_1+k_2+k_3=0\), this gives \(k_2\equiv k_1\pmod 3\) and
\(k_3\equiv -k_1-k_2\equiv -2k_1\equiv k_1\pmod 3\).  Hence
\(3\mid k_i-k_j\) for all \(i,j\).  From \eqref{eq:ki-sum-zero},
\[
        k_i-k_j=\frac{3(a_i-a_j)}{N}.
\]
Thus \(3(a_i-a_j)/N\) is an integer divisible by \(3\).  Therefore
\((a_i-a_j)/N\) is an integer, i.e. \(N\mid a_i-a_j\).  Hence all \(a_i\) are congruent modulo \(N\), so
\(\mathbf a\) is diagonal and is trivial on \(t_1t_2t_3=1\).  By
Lemma~\ref{lem:character-sheaf-dictionary}, the original rational character
sheaf is trivial.  This is the only point in the proof where the exact
geometric Kummer order \(N\) is used: it converts divisibility of the
boundary divisor into equality of the exponent coordinates modulo \(N\).
\end{proof}

\begin{theorem}[Cohomological form of the smooth subgroup trace theorem]
\label{thm:cohomological-smooth-subgroup}
Let \(\chi\ne1\) be a character of \(T_B(F)\), let \(\gamma\in B^\times\), and
let \(s^3\ne27\Norm(\gamma)\).  Then
\[
        H^i_c\bigl(U_{s,\Norm(\gamma),B,\overline F},
        \mathcal L_{\chi,\gamma}\bigr)=0
        \qquad(i\ne1),
\]
and
\[
        \dim H^1_c\bigl(U_{s,\Norm(\gamma),B,\overline F},
        \mathcal L_{\chi,\gamma}\bigr)=3.
\]
Consequently
\[
        |S_\chi(s;\gamma)|\le3\sqrt q .
\]
\end{theorem}

\begin{proof}
After base change to \(\overline F\) and an ordering of the three geometric
embeddings of \(B\), the smooth compactification of
\(U_{s,\Norm(\gamma),B}\) is the smooth plane cubic
\[
        XY(sZ-X-Y)-\Norm(\gamma)Z^3=0,
\]
hence is geometrically connected of genus one.  The affine curve
\(U_{s,\Norm(\gamma),B}\) is obtained from it by deleting the three geometric
points at infinity.  By Proposition~\ref{prop:smooth-picard-kummer-kernel}, the
rank-one sheaf \(\mathcal L_{\chi,\gamma}\) is geometrically nontrivial on this
affine curve.  It is finite-order, hence pure of weight zero, and tame at the
three punctures.

Geometric nontriviality gives \(H_c^0=0\).  Poincare duality gives
\(H_c^2=0\), because \(H_c^2\) is dual to the geometric invariant sections of
the dual sheaf.  The Grothendieck--Ogg--Shafarevich formula for a tame rank-one
lisse sheaf \(\mathcal L\) on \(U=C\setminus D\) gives
\cite[Expos\'e X]{DeligneSGA41half}
\[
        \chi_c(U_{\overline F},\mathcal L)
        =\operatorname{rank}(\mathcal L)(2-2g(C)-\#D)
        -\sum_{x\in D}\operatorname{Swan}_x(\mathcal L)
        =2-2\cdot1-3=-3.
\]
Since only \(H_c^1\) is nonzero, \(\dim H_c^1=3\).  Deligne's Riemann
Hypothesis for curves \cite{DeligneWeilII} bounds all eigenvalues on
\(H_c^1\) by \(\sqrt q\) in absolute value.  The stalk-trace formula of
Proposition~\ref{prop:relative-trace-stalk} gives
\(|S_\chi(s;\gamma)|\le3\sqrt q\).
\end{proof}

\begin{theorem}[Uniform square-root cancellation for subgroup cosets]
\label{thm:proper-subgroup-square-root}
Let \(B/F\) be any finite \'etale cubic algebra, let \(H\subset T_B(F)\) have
index \(m\), let \(g\in T_B(F)\), let \(\gamma\in B^\times\), and let
\(s\in F\).  Assume
\[
        s^3\ne27\Norm_{B/F}(\gamma).
\]
Then
\[
        N_{gH,B}(s;\gamma)
        =\frac1mN_B\bigl(s,\Norm(\gamma)\bigr)+E_{gH,B}(s;\gamma),
\]
with
\[
        |E_{gH,B}(s;\gamma)|
        \le 3\left(1-\frac1m\right)\sqrt q <3\sqrt q .
\]
\end{theorem}

\begin{proof}
By Proposition~\ref{prop:proper-subgroup-character-decomposition}, the error is
\[
        \frac1m\sum_{\substack{\chi\in H^\perp\\ \chi\ne1}}
        \chi(g^{-1})S_\chi(s;\gamma).
\]
The estimate \(|S_\chi(s;\gamma)|\le3\sqrt q\) for every nontrivial character is
Theorem~\ref{thm:cohomological-smooth-subgroup}.  There are \(m-1\)
nontrivial characters in \(H^\perp\), and division by \(m\) gives the stated
bound.
\end{proof}

\begin{corollary}[Smooth nonemptiness criterion for subgroup cosets]
\label{cor:smooth-coset-nonemptiness}
In the situation of Theorem~\ref{thm:proper-subgroup-square-root}, put
$m=[T_B(F):H]$ and $n=\Norm(\gamma)$.  If
\[
        N_B(s,n)>3(m-1)\sqrt q,
\]
then every coset $gH\subset T_B(F)$ contains at least one element $h$ with
$\Tr(\gamma h)=s$.  More generally,
\[
        N_{gH,B}(s;\gamma)
        \ge \frac{N_B(s,n)-3(m-1)\sqrt q}{m}.
\]
\end{corollary}

\begin{proof}
The lower bound is Theorem~\ref{thm:proper-subgroup-square-root} written as
\[
        N_{gH,B}(s;\gamma)
        \ge \frac1mN_B(s,n)-3\left(1-\frac1m\right)\sqrt q.
\]
The displayed criterion makes the right-hand side positive.
\end{proof}

\begin{corollary}[Smooth distribution in finite quotients]
\label{cor:smooth-quotient-distribution}
Let \(\pi:T_B(F)\to G\) be a surjective homomorphism to a finite abelian group,
and keep the smooth hypotheses of Theorem~\ref{thm:proper-subgroup-square-root}.
For every \(\xi\in G\),
\[
        \#\{h\in T_B(F):\pi(h)=\xi,\ \Tr(\gamma h)=s\}
        =\frac{N_B(s,\Norm\gamma)}{|G|}+E_{\xi,B}(s;\gamma),
\]
where
\[
        |E_{\xi,B}(s;\gamma)|\le
        3\left(1-\frac1{|G|}\right)\sqrt q<3\sqrt q .
\]
Thus every finite quotient of the norm-one torus sees the smooth trace/norm
fiber equidistributed with the same uniform square-root quality.
\end{corollary}

\begin{proof}
Apply Theorem~\ref{thm:proper-subgroup-square-root} to the subgroup
\(H=\ker\pi\) and to any coset \(gH\) with \(\pi(g)=\xi\).  Then
\([T_B(F):H]=|G|\), and the displayed statement follows.
\end{proof}

\begin{corollary}[The three splitting types]
\label{cor:proper-subgroup-all-splitting-types}
Under the hypotheses of Theorem~\ref{thm:proper-subgroup-square-root}, in
particular with \(n=\Norm(\gamma)\) and \(s^3\ne27n\), the theorem applies
uniformly to
\[
        B=F^3,
        \qquad B=F\times F_{q^2},
        \qquad B=F_{q^3}.
\]
In these smooth cases the main term is respectively
\[
\frac1m\bigl(\#E_{s,n}(F)-3\bigr),
\qquad
\frac1m\bigl(2q+1-\#E_{s,n}(F)\bigr),
\qquad
\frac1m\#E_{s,n}(F),
\]
where \(n=\Norm(\gamma)\) and
\[
        E_{s,n}:V^2=s^2U^2-4U^3-4s^3n-27n^2+18sUn .
\]
The same error bound \(<3\sqrt q\) holds in all three rows.
\end{corollary}

\begin{proof}
The main terms are the three rows of
Theorem~\ref{thm:prescribed-trace-norm-s3}.  The error bound is
Theorem~\ref{thm:proper-subgroup-square-root}.
\end{proof}

\begin{proposition}[Picard--Kummer kernel on the nodal normalization]
\label{prop:nodal-picard-kummer-kernel}
Assume \(n\in F^\times\), \(s^3=27n\), and put \(a=s/3\).  Then \(s\ne0\).  Over \(\overline F\), the normalization
of the split nodal prescribed trace/norm curve is parametrized by
\begin{equation}\label{eq:nodal-normalization-param}
\begin{aligned}
        x_1&=a\left(-\frac{t^2}{t+1}\right),
        &
        x_2&=a\left(-\frac1{t(t+1)}\right),
        &
        x_3&=a\left(\frac{(t+1)^2}{t}\right).
\end{aligned}
\end{equation}
The values \(t=0,-1,\infty\) map to the three points at infinity, while the two
roots of \(t^2+t+1\) map to the node \((a,a,a)\).

For a split Kummer exponent class represented modulo \(N\), with \(N\) prime to
\(\operatorname{char}F\),
\[
        \mathbf a=(a_1,a_2,a_3)
        \in(\Z/N\Z)^3/(\Z/N\Z)(1,1,1),
        \qquad S=a_1+a_2+a_3,
\]
the pullback to the normalization is, up to a constant, the Kummer function
\[
        t^{3a_1-S}(t+1)^{3a_3-S}.
\]
Consequently, inside the \(N\)-torsion exponent group, the geometric kernel is trivial if \(3\nmid N\).  If
\(3\mid N\), its nontrivial part is the cyclic order-three subgroup generated by the class of
\((0,N/3,2N/3)\).  These subgroups are compatible as \(N\) varies: in the full
geometric exponent group
\[
        (\Q/\Z)^3/(\Q/\Z)(1,1,1)
\]
they are the images of the single order-three subgroup generated by
\((0,1/3,2/3)\).  Equivalently, the geometric nodal kernel is canonically a
copy of \(\mu_3^\vee\).
\end{proposition}

\begin{proof}
Formula \eqref{eq:nodal-normalization-param} is obtained by projecting the nodal cubic from its node.  Direct
calculation gives
\[
        -\frac{t^2}{t+1}-\frac1{t(t+1)}+\frac{(t+1)^2}{t}=3,
\]
\[
        \left(-\frac{t^2}{t+1}\right)
        \left(-\frac1{t(t+1)}\right)
        \left(\frac{(t+1)^2}{t}\right)=1.
\]
Thus the trace is \(3a=s\) and the norm is \(a^3=n\).  The poles at
\(t=0,-1,\infty\) give the three points at infinity.  The equations
\(x_1=x_2=x_3=a\) reduce to \(t^2+t+1=0\), giving the two branches over the
ordinary node.

Substituting \eqref{eq:nodal-normalization-param} into \(x_1^{a_1}x_2^{a_2}x_3^{a_3}\) gives, up to a nonzero
constant,
\[
        t^{2a_1-a_2-a_3}(t+1)^{-a_1-a_2+2a_3}
        =t^{3a_1-S}(t+1)^{3a_3-S}.
\]
This function is an \(N\)-th power in \(\overline F(t)^\times\) if and only if
both exponents are divisible by \(N\):
\begin{equation}\label{eq:nodal-kernel-congruences}
        3a_1-S\equiv0\pmod N,
        \qquad
        3a_3-S\equiv0\pmod N.
\end{equation}
Modulo diagonal classes, the congruences \eqref{eq:nodal-kernel-congruences} have no nontrivial solution when
\(3\nmid N\).  When \(3\mid N\), their solutions form the cyclic group generated
by \((0,N/3,2N/3)\).

This identification is independent of the auxiliary modulus \(N\).  In the
full geometric exponent group
\[
        (\Q/\Z)^3/(\Q/\Z)(1,1,1),
\]
the kernel is exactly
\[
        \{0,\kappa,2\kappa\},
        \qquad
        \kappa=(0,1/3,2/3)\bmod (\Q/\Z)(1,1,1).
\]
For every \(N\) divisible by \(3\), the image of \(\kappa\) in the
\(N\)-torsion subgroup is the class \((0,N/3,2N/3)\), and these images are
compatible under change of level.  Hence the geometric nodal kernel is the
canonical order-three character group \(\mu_3^\vee\).
\end{proof}

\begin{lemma}[Normalized descent for finite Kummer kernels]
\label{lem:normalized-kummer-kernel-descent}
Let \(p_F=\operatorname{char}F\), and let
\[
        (\Q/\Z)^{(p_F')}=
        \{\alpha\in\Q/\Z:\text{the order of }\alpha
        \text{ is prime to }p_F\}.
\]
Let
\[
        \mathscr K_B=X^*(T_{B,\overline F})\otimes_{\Z}(\Q/\Z)^{(p_F')}
\]
be the prime-to-\(p_F\) geometric Kummer character group of \(T_B\), equipped
with the arithmetic Frobenius action coming from the \(F\)-structure on
\(T_B\) and from \(\zeta\mapsto\zeta^q\) on roots of unity of order prime to
\(p_F\).  Let
\(K\subset\mathscr K_B\) be a finite arithmetic-Frobenius-stable subgroup.  Under
the normalized Lang/Kummer dictionary of Lemma~\ref{lem:character-sheaf-dictionary},
rational characters of \(T_B(F)\) whose geometric Kummer class lies in \(K\) are
canonically identified with the fixed subgroup
\[
        K^{\Frob_q}.
\]
Equivalently, the full normalized Lang character group satisfies the canonical
identification
\[
        \widehat{T_B(F)}\simeq \mathscr K_B^{\Frob_q},
\]
and the displayed assertion is its restriction to \(K\).  In particular, after
imposing the normalization that the trace at the identity of \(T_B(F)\) is
\(1\), there is no additional arithmetic rank-one twist.
\end{lemma}

\begin{proof}
By the prime-to-characteristic definition of \(\mathscr K_B\), choose an integer
\(N\), prime to \(\operatorname{char}F\), that kills \(K\).
The \(N\)-torsion of \(\mathscr K_B\) is
\[
        X^*(T_{B,\overline F})/NX^*(T_{B,\overline F}),
\]
and this is the Kummer group \(H^1(T_{B,\overline F},\mu_N)\) by the Kummer exact
sequence.  In split coordinates it is represented by the classes of monomials
\(t_1^{a_1}t_2^{a_2}t_3^{a_3}\), modulo the diagonal relation
\(t_1t_2t_3=1\).  If the arithmetic Frobenius permutation of the chosen ordered
embeddings is \(\tau_B\), then the convention used here is
\begin{equation}\label{eq:arithmetic-frobenius-exponent-action}
        \Frob_q\bigl([a_1,a_2,a_3]\bigr)
        =q\,[a_{\tau_B^{-1}(1)},a_{\tau_B^{-1}(2)},a_{\tau_B^{-1}(3)}]
        \quad\text{in }(\Z/N\Z)^3/(\Z/N\Z)(1,1,1).
\end{equation}
Equivalently, arithmetic Frobenius acts by pulling back the coordinates through
the descent datum of \(T_B\) and by sending an \(N\)-th root of unity to its
\(q\)-th power.  Formula
\eqref{eq:arithmetic-frobenius-exponent-action} is the Frobenius action used
below.

In these terms the normalized Lang/Kummer correspondence of
Lemma~\ref{lem:character-sheaf-dictionary} gives the concrete finite-level
identification
\[
        \operatorname{Hom}\bigl(T_B(F),\mu_N\bigr)
        \simeq
        H^1(T_{B,\overline F},\mu_N)^{\Frob_q}
        \simeq
        \bigl(X^*(T_{B,\overline F})/NX^*(T_{B,\overline F})\bigr)^{\Frob_q}.
\]
The first map sends a normalized Lang summand to its geometric Kummer class; its
inverse equips a Frobenius-fixed geometric Kummer class with the unique Weil
structure whose trace at the identity of \(T_B(F)\) is \(1\).  Thus this
identification is an identification of normalized character functions, not only
of geometric local systems.

A geometric Kummer class \(\kappa\) can underlie a Weil sheaf over \(F\) only if
\(\Frob_q^*\kappa=\kappa\); otherwise Frobenius sends the geometric local system
to a non-isomorphic one.  If \(\kappa\) is fixed, a Weil structure exists, and
any two such structures differ by a scalar rank-one sheaf pulled back from
\(\Spec F\), because the endomorphism ring of a geometrically irreducible
rank-one local system is the constant field.  Multiplying the Weil structure by
that scalar multiplies the trace function at every \(F\)-rational point by the
same scalar.  Requiring the trace at the identity element of \(T_B(F)\) to be
\(1\) fixes the scalar uniquely.

With this normalization, the group law on \(T_B\) gives the usual
multiplicativity isomorphism
\[
        m^*\mathcal L_\kappa\simeq
        \operatorname{pr}_1^*\mathcal L_\kappa\otimes
        \operatorname{pr}_2^*\mathcal L_\kappa
\]
compatible with the normalized fibers at the identity.  Hence the trace
function on \(T_B(F)\) is a finite character.  Conversely, every finite
character of \(T_B(F)\) occurs as a normalized Lang summand in
Lemma~\ref{lem:character-sheaf-dictionary}; its geometric Kummer class is
Frobenius fixed because the summand is defined over \(F\).  The two
constructions are inverse: two normalized descents with the same geometric class
differ by the scalar just described, and the identity normalization forces that
scalar to be \(1\).  Restricting this bijection to the finite stable subgroup
\(K\) gives the asserted identification with \(K^{\Frob_q}\).  The trivial
geometric class gives only the trivial normalized character, again because the
identity trace is fixed to be \(1\).
\end{proof}

\begin{definition}[The nodal exceptional cubic-character group]
\label{def:nodal-exceptional-group}
Let \(\tau_B\in S_3\) be the arithmetic Frobenius permutation of the three
geometric embeddings of \(B\), and put \(\epsilon_B=\operatorname{sgn}(\tau_B)\).
The exceptional group \(\mathcal E_B\subset\widehat{T_B(F)}\) is the group of
normalized rational Lang characters whose geometric class lies in the nodal
kernel of Proposition~\ref{prop:nodal-picard-kummer-kernel}.  By
Lemma~\ref{lem:normalized-kummer-kernel-descent}, this is exactly the
arithmetic-Frobenius fixed subgroup of that geometric nodal kernel.
\end{definition}

\begin{proposition}[Descent of the nodal cubic kernel]
\label{prop:nodal-kernel-descent}
Arithmetic Frobenius acts on the geometric nodal kernel \(\mu_3^\vee\) by
multiplication by \(q\epsilon_B\).  Hence
\[
        |\mathcal E_B|=
        \begin{cases}
        3, & q\epsilon_B\equiv1\pmod3,\\
        1, & q\epsilon_B\not\equiv1\pmod3.
        \end{cases}
\]
Equivalently,
\[
\begin{array}{c|c|c|c}
\toprule
B & \tau_B & \epsilon_B & \mathcal E_B\ne\{1\}\text{ exactly when}\\
\midrule
F^3 & 1 & +1 & q\equiv1\pmod3\\
F\times F_{q^2} & \text{transposition} & -1 & q\equiv2\pmod3\\
F_{q^3} & \text{three-cycle} & +1 & q\equiv1\pmod3\\
\bottomrule
\end{array}
\]
\end{proposition}

\begin{proof}
Let \(\kappa\) be the generator of the geometric nodal kernel represented by
\((0,1,2)\) modulo diagonal classes.  The coordinate part of the arithmetic
Frobenius action is the permutation part of
\eqref{eq:arithmetic-frobenius-exponent-action}.  Whether one writes this
permutation as \(\tau_B\) or \(\tau_B^{-1}\), its action on the order-three
nodal kernel is the same sign action.  From the parametrization
\eqref{eq:nodal-normalization-param}, solving for the new parameter after
permuting the coordinates gives
\[
        (123):\quad t\longmapsto -\frac1{t+1},
        \qquad
        (23):\quad t\longmapsto -\frac{t}{t+1}.
\]
Substitution of \(t'=-1/(t+1)\) gives \((x_2,x_3,x_1)\), while substitution of
\(t'=-t/(t+1)\) gives \((x_1,x_3,x_2)\).  On exponent classes, the three-cycle
sends
\[
        (0,1,2)\longmapsto(1,2,0)=(0,1,2)+(1,1,1),
\]
so it fixes \(\kappa\).  A transposition sends
\[
        (0,1,2)\longmapsto(0,2,1)=-(0,1,2)
        \quad\text{modulo the diagonal class}.
\]
Thus the coordinate-twist part acts on the cyclic order-three nodal kernel by
the sign character \(\epsilon_B=\operatorname{sgn}(\tau_B)\).

It remains to specify the \(q\)-factor.  In the Kummer convention of
Lemma~\ref{lem:normalized-kummer-kernel-descent}, arithmetic Frobenius acts on
the coefficient group of a cubic Kummer cover by
\[
        \zeta\longmapsto\zeta^q\qquad(\zeta\in\mu_3).
\]
Equivalently, on the additive character group \(\mu_3^\vee\simeq\Z/3\Z\), this
is multiplication by \(q\).  Therefore the full arithmetic Frobenius action on
the geometric nodal kernel is
\[
        \kappa\longmapsto q\,\epsilon_B\,\kappa .
\]
If one rewrites the same calculation with geometric Frobenius, the scalar is
inverted; on an order-three group this gives the same fixed subgroup because
\(q^{-1}\equiv q\pmod3\) and \(\epsilon_B^{-1}=\epsilon_B\).

By Lemma~\ref{lem:normalized-kummer-kernel-descent}, the rational exceptional
characters are exactly the fixed points of this action on
\(\{0,\kappa,2\kappa\}\).  The fixed subgroup has order \(3\) precisely when
\(q\epsilon_B\equiv1\pmod3\), and otherwise has only the identity.  The three
rows are obtained by inserting the signs of the identity, transposition, and
three-cycle Frobenius types.
\end{proof}

\begin{lemma}[Rational preimages of the node]
\label{lem:rational-node-branches}
Let \(F=\F_q\) have characteristic different from \(2\) and \(3\), let \(B/F\)
be a finite \'etale cubic algebra, let \(s\in F^\times\), and put
\(n=s^3/27\).  Let \(\mathcal E_B\) be the nodal exceptional cubic-character
group of Definition~\ref{def:nodal-exceptional-group}.  In the nodal fiber
\(s^3=27n\), the number of \(F\)-rational preimages of the node in the
normalization is \(|\mathcal E_B|-1\).
\end{lemma}

\begin{proof}
Over \(\overline F\), the two preimages of the node are the two roots of
\(t^2+t+1\), i.e. \(\mu_3\setminus\{1\}\).  The same coordinate calculation used
above describes the twisting action on this two-point set.  A three-cycle acts
trivially: under \((123)\) the parameter is transformed by
\(t\mapsto-1/(t+1)\), which fixes both roots of \(t^2+t+1\).  A transposition
swaps the two roots: for example \((23)\) corresponds to
\(t\mapsto -t/(t+1)\), which sends one primitive cube root to the other.  Thus
the \(S_3\)-twist acts on the branch pair through the sign character.  Arithmetic
Frobenius also acts on \(\mu_3\setminus\{1\}\) by \(\zeta\mapsto\zeta^q\).  Hence
the branch pair is fixed pointwise exactly when multiplication by
\(q\epsilon_B\) is the identity on \(\mu_3\).  In that case there are two
rational preimages; otherwise there are none.  This number is exactly
\(|\mathcal E_B|-1\) by Proposition~\ref{prop:nodal-kernel-descent}.
\end{proof}

\begin{proposition}[Full nodal trace/norm count]
\label{prop:full-nodal-count}
Let \(B/F\) be a finite \'etale cubic algebra over a finite field \(F=\F_q\) of
characteristic different from \(2\) and \(3\).  Let \(s\in F^\times\), and put
\(n=s^3/27\).  Let
\(f_B=\#\operatorname{Fix}_{\{1,2,3\}}(\tau_B)\).  Then the affine nodal
trace/norm count is
\[
        N_B^{\rm nod}(s,n)
        :=\#\{x\in B^\times:\Tr(x)=s,\ \Norm(x)=n\}
        =q+3-f_B-|\mathcal E_B|.
\]
Equivalently, the two congruence classes of \(q\) modulo \(3\) give the following closed values:
\begin{center}
\small
\begin{tabular}{@{}c c p{0.28\textwidth} p{0.28\textwidth}@{}}
\toprule
\(B\) & \(f_B\) & \(q\equiv1\pmod3\) & \(q\equiv2\pmod3\)\\
\midrule
\(F^3\) & \(3\) & \(|\mathcal E_B|=3\), \(N_B^{\rm nod}=q-3\) & \(|\mathcal E_B|=1\), \(N_B^{\rm nod}=q-1\)\\
\(F\times F_{q^2}\) & \(1\) & \(|\mathcal E_B|=1\), \(N_B^{\rm nod}=q+1\) & \(|\mathcal E_B|=3\), \(N_B^{\rm nod}=q-1\)\\
\(F_{q^3}\) & \(0\) & \(|\mathcal E_B|=3\), \(N_B^{\rm nod}=q\) & \(|\mathcal E_B|=1\), \(N_B^{\rm nod}=q+2\)\\
\bottomrule
\end{tabular}
\end{center}
\end{proposition}

\begin{proof}
The singular point \((s/3)\cdot1\) is \(F\)-rational.  The split nodal cubic is
geometrically irreducible.  Normalization commutes with separable base change
for this ordinary nodal curve; equivalently, the ordering-torsor twist of the
split normalization is the normalization of the twisted nodal fiber.  Since
twisting preserves geometric irreducibility after base change, the normalization
is a smooth projective geometrically integral genus-zero curve over the finite
field \(F\).  Every
smooth projective geometrically integral genus-zero curve over a finite field
has an \(F\)-rational point, equivalently \(\operatorname{Br}(F)=0\); hence it is
isomorphic to \(\PP^1\).  Thus the normalization has \(q+1\) rational points.
In the split model the three points at infinity are smooth points of the nodal
cubic, so the normalization is an isomorphism above them.  Their preimages are
\(t=0,-1,\infty\) in \eqref{eq:nodal-normalization-param}, and the ordering-torsor
twist permutes these three preimages by the same Frobenius permutation
\(\tau_B\) as the coordinate labels.  Therefore the \(F\)-rational preimages of
the points at infinity are exactly the fixed labels, namely \(f_B\) points.  To obtain the affine nodal curve from the
normalization, remove these rational preimages of the three points at infinity
and remove the rational preimages of the node.  By
Lemma~\ref{lem:rational-node-branches}, the latter number is
\(|\mathcal E_B|-1\).  Then add the rational node itself once.  Therefore
\[
        N_B^{\rm nod}(s,n)
        =q+1-f_B-(|\mathcal E_B|-1)+1
        =q+3-f_B-|\mathcal E_B|.
\]
The table follows by inserting \(f_B=3,1,0\) and the exceptional-group sizes
from Proposition~\ref{prop:nodal-kernel-descent}.
\end{proof}

\begin{corollary}[Closed nodal substitution for branch-census formulae]
\label{cor:closed-nodal-branch-census-substitution}
In the formal branch-census identities of
Theorem~\ref{thm:affine-finite-field-branch-census} and
Corollary~\ref{cor:exact-affine-local-branch-census}, every nodal summand with
\(s^3=27N_\delta\) may now be replaced by
\[
        N_B(s,N_\delta)=N_B^{\rm nod}(s,N_\delta)
        =q+3-f_B-|\mathcal E_B|,
\]
where \(f_B=\#\operatorname{Fix}_{\{1,2,3\}}(\tau_B)\).  Thus the earlier
branch-census statements are formal count identities before
Proposition~\ref{prop:full-nodal-count}, and become closed formulae after this
substitution.
\end{corollary}

\begin{proof}
This is Proposition~\ref{prop:full-nodal-count} applied to the nodal parameter
\(n=N_\delta=s^3/27\).  The quantities \(f_B\) and \(\mathcal E_B\) depend only
on the finite \'etale cubic algebra \(B/F\) and its Frobenius cycle type, not on
which norm-fiber summand produced the nodal equality.
\end{proof}

\begin{lemma}[Rational value of exceptional characters on the nodal fiber]
\label{lem:nodal-exceptional-rational-value}
Let \(B/F\) be a finite \'etale cubic algebra over a finite field of
characteristic different from \(2\) and \(3\).  Let \(\gamma\in B^\times\) and
\(s\in F^\times\) satisfy \(s^3=27\Norm(\gamma)\).  Put \(a=s/3\), and set
\[
        h_*=a\gamma^{-1}\in T_B(F).
\]
If \(\chi\in\mathcal E_B\), then for every \(F\)-rational point
\(h\in T_B(F)\) satisfying \(\Tr(\gamma h)=s\), one has
\[
        \chi(h)=\chi(h_*).
\]
Consequently
\[
        S_\chi(s;\gamma)=\chi(h_*)N_B^{\rm nod}(s,\Norm(\gamma))
        \qquad(\chi\in\mathcal E_B).
\]
\end{lemma}

\begin{proof}
Write \(x=\gamma h\).  On the nodal fiber one has \(x=a r\) with
\(r\in T_B\), and therefore \(h=h_*r\).  It is enough to prove that every
normalized exceptional character has trace value \(1\) on every \(F\)-rational
point \(r\) of the translated nodal fiber.  This is immediate for
\(\chi=1\), so assume first that \(\chi\) is nontrivial.

Work geometrically after choosing an ordering of the three embeddings.  On the
normalization of the translated nodal fiber, away from the three points at
infinity,
\[
        r_1=-\frac{t^2}{t+1},
        \qquad
        r_2=-\frac1{t(t+1)},
        \qquad
        r_3=\frac{(t+1)^2}{t}.
\]
The generator \(\kappa=(0,1,2)\) of the geometric nodal kernel is represented by
the cubic Kummer function
\[
        r_2r_3^2
        =-\left(\frac{t+1}{t}\right)^3
        =\left(-\frac{t+1}{t}\right)^3 .
\]
Thus the pullback of the corresponding geometric Kummer sheaf to the open
normalization is geometrically constant.  The same is true for every power of
\(\kappa\), hence for the geometric class of every \(\chi\in\mathcal E_B\).

The only remaining issue is the Weil scalar of this geometrically constant
restriction.  If \(\chi\ne1\), then \(\mathcal E_B\) has order \(3\).  By
Lemma~\ref{lem:rational-node-branches}, the two preimages of the node in the
normalization are then \(F\)-rational.  The normalization map sends both of
these points to \(r=1\) in the ambient torus.  Since the restricted sheaf is
geometrically constant, its Frobenius scalar may be computed at either of these
rational preimages; the fiber there is the fiber of the normalized Lang sheaf at
the identity of \(T_B(F)\), whose trace is \(\chi(1)=1\).  Hence the pulled-back
constant Weil sheaf has Frobenius scalar \(1\).

Therefore the trace value of \(\chi\) is \(1\) at every rational point of the
open normalization.  Away from the node, the normalization is an isomorphism
onto the smooth locus of the affine nodal fiber, so the same value holds there.
At the rational node itself the translated point is \(r=1\), and the ambient
normalized Lang trace is again \(\chi(1)=1\).  Consequently
\[
        \chi(h/h_*)=\chi(r)=1
\]
for every rational point of the affine nodal trace fiber, and hence
\(\chi(h)=\chi(h_*)\).  Summing this identity over the
\(N_B^{\rm nod}(s,\Norm(\gamma))\) rational points gives the displayed formula
for \(S_\chi(s;\gamma)\).
\end{proof}

\begin{lemma}[Nonexceptional nodal normalization estimate]
\label{lem:nonexceptional-nodal-normalization-estimate}
Let \(B/F\) be a finite \'etale cubic algebra over a finite field of characteristic different from \(2\) and \(3\), let \(\gamma\in B^\times\), and let \(s\in F^\times\) satisfy \(s^3=27\Norm(\gamma)\).  Let \(\chi\notin\mathcal E_B\).  Pull \(\mathcal L_{\chi,\gamma}\) back to the normalization of the nodal trace/norm curve and restrict to the open normalization obtained by deleting the three preimages of the points at infinity.  Then this is a geometrically nontrivial tame rank-one sheaf.  Its compactly supported cohomology satisfies
\[
        H_c^0=H_c^2=0,\qquad \dim H_c^1\le3,
\]
and the corresponding normalization trace sum \(\widetilde S_\chi(s;\gamma)\)
has absolute value at most \(3\sqrt q\).  The finite nodal character sum
\[
        S_\chi(s;\gamma)=
        \sum_{\substack{h\in T_B(F)\\ \Tr(\gamma h)=s}}\chi(h)
\]
differs from \(\widetilde S_\chi(s;\gamma)\) by at most \(3\).
\end{lemma}

\begin{proof}
The statement is geometric after descent.  Over \(\overline F\), choose split
coordinates and the nodal normalization parameter \(t\) from
Proposition~\ref{prop:nodal-picard-kummer-kernel}.  If the exact geometric
Kummer order of the pulled-back character sheaf is \(M\), its pulled-back
Kummer function is, up to a constant,
\[
        t^A(t+1)^B
\]
with exponents read modulo \(M\).  By
Proposition~\ref{prop:nodal-picard-kummer-kernel}, the normalization pullback is
geometrically trivial if and only if \(A\equiv B\equiv0\pmod M\), equivalently
if and only if the geometric exponent class lies in the nodal Picard--Kummer
kernel.  By Definition~\ref{def:nodal-exceptional-group} and
Proposition~\ref{prop:nodal-kernel-descent}, a rational character has this
property exactly when it lies in \(\mathcal E_B\).  Since
\(\chi\notin\mathcal E_B\), not both \(A\) and \(B\) vanish modulo \(M\), and
the sheaf on the open normalization is geometrically nontrivial.  It is tame
and lisse on \(\PP^1\) away from the three points \(0,-1,\infty\), the
preimages of the points at infinity of the nodal cubic.  The two preimages of
the node are the roots of \(t^2+t+1\); at those points both \(t\) and \(t+1\)
are nonzero, so the Kummer function \(t^A(t+1)^B\) has valuation zero and
introduces no additional puncture or ramification.  The ordering torsor
for a nonsplit \(B\) only permutes this geometric picture, so the same
geometric nontriviality and conductor bound hold after descent.

For a geometrically nontrivial rank-one sheaf, \(H_c^0=0\), and \(H_c^2=0\) by
duality because the sheaf has no geometric constant quotient.
Grothendieck--Ogg--Shafarevich on \(\PP^1\) with at most three tame punctures
gives \(\dim H_c^1\le3\).  Deligne's Riemann Hypothesis for curves then bounds
the normalization trace sum \(\widetilde S_\chi(s;\gamma)\) by \(3\sqrt q\).

It remains only to compare finite sets, not to put a lisse sheaf on the
singular curve.  Let \(\nu:\widetilde C\to C\) be the normalization, and let
\(\widetilde U\subset\widetilde C\) be the complement of the three preimages of
the points at infinity.  Away from the node, \(\nu\) identifies
\(\widetilde U\) with the smooth locus of the affine nodal trace/norm curve,
and the trace function is exactly \(h\mapsto\chi(h)\).  The affine nodal curve
has one additional rational point at the node, corresponding to
\(h_*=(s/3)\gamma^{-1}\), whereas \(\widetilde U\) has either zero or two
\(F\)-rational preimages of that node.  The two preimages, when rational, also
map to \(h_*\).  Thus the two sums differ only by the contribution of the node
and its at most two normalization preimages.  Each trace value is a root of
unity, so
\[
        |S_\chi(s;\gamma)-\widetilde S_\chi(s;\gamma)|\le3 .
\]
\end{proof}

\begin{theorem}[Nodal subgroup-orbit formula with secondary main terms]
\label{thm:nodal-subgroup-orbit-formula}
Let \(B/F\) be a finite \'etale cubic algebra, let \(H\subset T_B(F)\) have
index \(m\), let \(g\in T_B(F)\), and let \(\gamma\in B^\times\).  Let
\(s\in F^\times\) satisfy
\[
        s^3=27\Norm(\gamma),
\]
and put
\[
        h_*=(s/3)\gamma^{-1}\in T_B(F).
\]
Then
\[
\begin{aligned}
        N_{gH,B}^{\rm nod}(s;\gamma)
        &:=\#\{h\in gH:\Tr(\gamma h)=s\} \\
        &=\frac{N_B^{\rm nod}(s,\Norm(\gamma))}{m}
        \sum_{\chi\in H^\perp\cap\mathcal E_B}\chi(g^{-1}h_*)
        +R_{gH,B}^{\rm nod}(s;\gamma),
\end{aligned}
\]
where
\[
        \left|R_{gH,B}^{\rm nod}(s;\gamma)\right|
        \le
        \frac{m-|H^\perp\cap\mathcal E_B|}{m}\,(3\sqrt q+3).
\]
Thus the expected main term \(N_B^{\rm nod}/m\) is the whole main term unless
\(H^\perp\) contains a nontrivial exceptional cubic character.  When such
characters occur, the exceptional projection can contribute an order-\(q\) term; it is zero on incompatible cosets and is explicit in all cases.
\end{theorem}

\begin{proof}
Use the character decomposition of
Proposition~\ref{prop:proper-subgroup-character-decomposition}.  For
\(\chi\in\mathcal E_B\), Lemma~\ref{lem:nodal-exceptional-rational-value} gives
\[
        S_\chi(s;\gamma)=\chi(h_*)N_B^{\rm nod}(s,\Norm(\gamma)).
\]
These are exactly the exceptional terms displayed in the theorem, including the
trivial character.

Now let \(\chi\notin\mathcal E_B\).  Lemma~\ref{lem:nonexceptional-nodal-normalization-estimate} gives the normalization trace bound \(3\sqrt q\) and the node/normalization correction of size at most \(3\).  Hence
\[
        |S_\chi(s;\gamma)|\le3\sqrt q+3
        \qquad(\chi\notin\mathcal E_B).
\]
Summing the nonexceptional character contributions and dividing by \(m\) gives
the asserted remainder bound.
\end{proof}

\begin{corollary}[Exceptional quotient concentration on the nodal fiber]
\label{cor:exceptional-quotient-concentration}
Let \(B/F\) be a finite \'etale cubic algebra over a finite field of
characteristic different from \(2\) and \(3\).  Let \(\gamma\in B^\times\) and
\(s\in F^\times\) satisfy \(s^3=27\Norm(\gamma)\), and put
\(h_*=(s/3)\gamma^{-1}\).  Define
\[
        K_B^{\rm exc}=\bigcap_{\chi\in\mathcal E_B}\ker(\chi)
        \subset T_B(F).
\]
Then every \(F\)-rational point of the affine nodal trace fiber
\[
        \{h\in T_B(F):\Tr(\gamma h)=s\}
\]
lies in the single coset
\[
        h_*K_B^{\rm exc}.
\]
If \(|\mathcal E_B|=3\), this is an index-three coset of \(T_B(F)\); if
\(\mathcal E_B=\{1\}\), the assertion is vacuous.  More generally, for a subgroup
\(H\subset T_B(F)\) and a coset \(gH\), the order-\(q\) exceptional main term in
Theorem~\ref{thm:nodal-subgroup-orbit-formula} is zero unless
\[
        gH\cap h_*K_B^{\rm exc}\ne\varnothing.
\]
When this intersection is nonempty, that main term is
\[
        N_B^{\rm nod}(s,\Norm\gamma)\,
        \frac{|H\cap K_B^{\rm exc}|}{|K_B^{\rm exc}|}.
\]
The remaining contribution is the nonexceptional square-root error of
Theorem~\ref{thm:nodal-subgroup-orbit-formula}.
\end{corollary}

\begin{proof}
The equality \(\chi(h)=\chi(h_*)\) for all \(\chi\in\mathcal E_B\), proved in
Lemma~\ref{lem:nodal-exceptional-rational-value}, is exactly the assertion that
\(h/h_*\in K_B^{\rm exc}\).  This proves the concentration statement.

For the main-term reformulation, let \(E=\mathcal E_B\), \(K=K_B^{\rm exc}\),
and let \(\pi:T_B(F)\to T_B(F)/K\) be the quotient.  The group \(E\) is the full
character group of \(T_B(F)/K\).  The subgroup \(E\cap H^\perp\) is therefore
the character group of \((T_B(F)/K)/\pi(H)\).  By finite abelian-group
orthogonality,
\[
        \sum_{\chi\in E\cap H^\perp}\chi(g^{-1}h_*)
\]
is zero unless \(\pi(g^{-1}h_*)\in\pi(H)\), equivalently
\(gH\cap h_*K\ne\varnothing\); in the nonzero case the sum is
\(|E\cap H^\perp|\).  Since
\[
        \frac{|E\cap H^\perp|}{[T_B(F):H]}
        =\frac{|H\cap K|}{|K|},
\]
substitution in Theorem~\ref{thm:nodal-subgroup-orbit-formula} gives the displayed
main term.
\end{proof}

\begin{example}[Split nodal exceptional projection]
\label{ex:split-nodal-secondary}
Assume \(B=F^3\) and \(q\equiv1\pmod3\).  Let \(\rho:F^\times\to\mu_3\) be a
nontrivial cubic character, and define
\[
        \chi_0(t_1,t_2,t_3)=\rho(t_2t_3^2)
        \qquad(t_1t_2t_3=1).
\]
Then \(\mathcal E_B=\{1,\chi_0,\chi_0^2\}\).  If
\(H=\ker\chi_0\), then \(H^\perp=\mathcal E_B\), so the nonexceptional remainder
in Theorem~\ref{thm:nodal-subgroup-orbit-formula} is zero.  Consequently
\[
        N_{gH,B}^{\rm nod}(s;\gamma)=
        \begin{cases}
        N_B^{\rm nod}(s,\Norm\gamma), & gH=h_*H,\\
        0, & gH\ne h_*H.
        \end{cases}
\]
This is the concrete finite-group form of Lemma~\ref{lem:nodal-exceptional-rational-value}:
all rational points of the split nodal fiber lie in the single coset
\(h_*\ker\chi_0\).
\end{example}

\begin{corollary}[Smooth/nodal dichotomy for subgroup orbits]
\label{cor:smooth-nodal-subgroup-dichotomy}
Let \(B/F\) be cubic \'etale over a finite field of characteristic different from \(2\) and \(3\), let \(gH\subset T_B(F)\) be a subgroup coset, and let \(\gamma\in B^\times\).  All character sheaves appearing below have finite order prime to \(\operatorname{char}F\).  For every smooth fiber \(s^3\ne27\Norm(\gamma)\), the coset count has only the equidistributed main term and a square-root error.  On the nodal boundary \(s^3=27\Norm(\gamma)\), necessarily \(s\ne0\), and the nodal fiber has the same square-root cancellation, up to the bounded normalization/node correction in Theorem~\ref{thm:nodal-subgroup-orbit-formula}, after removing the explicit exceptional cubic-character projection.  No other secondary main terms occur.
\end{corollary}

\begin{proof}
The smooth assertion is Theorem~\ref{thm:proper-subgroup-square-root}.  The
nodal assertion is Theorem~\ref{thm:nodal-subgroup-orbit-formula}.  The
exceptional group is precisely the Frobenius-fixed part of the Picard--Kummer
nodal kernel by Proposition~\ref{prop:nodal-kernel-descent}.
\end{proof}

\begin{remark}[Conceptual role of the relative complex]
The smooth subgroup theorem and the nodal formula are two stalk computations of
\(K_{\chi,\gamma}=R\tau_!\mathcal L_{\chi,\gamma}\).  On smooth fibers, the
Picard--Kummer map has no nontrivial kernel, so every nontrivial character
contributes compactly supported cohomology only in degree one, of dimension
three, with weights at most one.  No purity assertion is needed here; boundary
monodromy can leave weight-zero pieces in special finite-order cases.  At the
nodal boundary, the genus-one curve degenerates to a rational nodal curve and
the Picard--Kummer kernel becomes \(\mu_3^\vee\).  The rational fixed part of
that kernel is \(\mathcal E_B\), and it is exactly the source of the order-\(q\)
secondary main terms.  This is the cohomological synthesis behind the explicit
formulae above.
\end{remark}

\section{Singular branch statistics in finite-field and jet families}
\label{sec:singular-statistics}

The codifferent census gives an exact pointwise description of singular
classes.  In the homogeneous full norm-fiber case, for a fixed norm fiber
\(\Norm(h)=\nu\), singular classes are governed by the cube class
\[
        -\Norm_{B/F}(\gamma)\nu\disc(f_\omega)
        \in F^\times/(F^\times)^3 .
\]
We now prove that, when \(|F|\equiv1\pmod3\), these cube classes are
equidistributed in the natural finite-field family where \(\omega\) varies over
generators of a fixed cubic \'etale algebra; when \(|F|\equiv2\pmod3\), the cube
map is bijective and there is only one cube class.
We also record the finite-jet statistic governing the next quadratic Hensel
step: conditional on reaching a nondegenerate singular disk, the two lower
quadratic coefficients are uniformly distributed in the natural lift family.

For a finite field \(F=\F_q\), a finite \'etale cubic \(F\)-algebra \(B\),
and \(\omega\in B\), let \(m_\omega:B\to B\) denote the \(F\)-linear
multiplication operator \(b\mapsto\omega b\), and let
\[
        \Delta_B(\omega)=\disc\bigl(\det(T\cdot\operatorname{id}_B-m_\omega)\bigr).
\]
Write
\[
        B_{\rm gen}=\{\omega\in B:\Delta_B(\omega)\ne0\}.
\]

\begin{lemma}[Generator discriminant locus in cubic finite \texorpdfstring{\'etale}{etale} algebras]
\label{lem:generator-discriminant-count}
Let \(F=\F_q\) have characteristic different from \(2\) and \(3\), and let
\(B/F\) be a finite \'etale cubic algebra.  For every \(\omega\in B\),
\(\Delta_B(\omega)\ne0\) if and only if \(\omega\) generates \(B\) as an
\(F\)-algebra.  The three splitting types have
\[
\#B_{\rm gen}= 
\begin{cases}
q(q-1)(q-2), & B=F^3,\\
q^2(q-1), & B=F\times F_{q^2},\\
q^3-q, & B=F_{q^3}.
\end{cases}
\]
\end{lemma}

\begin{proof}
After base change to \(\overline F\), the algebra becomes \(\overline F^3\), and
the eigenvalues of \(m_\omega\) are the three geometric coordinates of
\(\omega\).  Thus \(\Delta_B(\omega)\ne0\) is equivalent to these three
coordinates being pairwise distinct.  This is equivalent to the three vectors
\(1,\omega,\omega^2\) being linearly independent over \(\overline F\), by the
Vandermonde determinant.  Since both conditions are defined over \(F\), it is
equivalent to \(1,\omega,\omega^2\) forming an \(F\)-basis of \(B\), hence to
\(F[\omega]=B\).

For \(B=F^3\), this says that the three coordinates of \(\omega\) are pairwise
distinct, giving \(q(q-1)(q-2)\) choices.  For
\(B=F\times F_{q^2}\), write \(\omega=(a,b)\).  Generation is equivalent to
\(b\notin F\), so there are \(q(q^2-q)=q^2(q-1)\) choices.  For
\(B=F_{q^3}\), since \(3\) is prime there is no intermediate proper subfield,
so generation is equivalent to \(\omega\notin F\), giving \(q^3-q\) choices.
\end{proof}

\begin{lemma}[Cubic discriminant character sum]
\label{lem:discriminant-character-sum}
Let \(F=\F_q\) be a finite field of characteristic different from \(2\) and \(3\), with \(q\equiv1\pmod3\), and let \(\psi\) be a nontrivial cubic character of
\(F^\times\), extended by \(\psi(0)=0\).  Then, for every finite \'etale cubic
\(F\)-algebra \(B\),
\[
        \left|\sum_{\omega\in B}\psi\bigl(\Delta_B(\omega)\bigr)\right|
        \le q(q-1)\sqrt q < q^{5/2}.
\]
In particular the constant in the estimate is absolute and independent of
\(q\), \(B\), and \(\psi\).
\end{lemma}

\begin{proof}
The discriminant is invariant under translation by scalars and homogeneous of
degree six under scalar multiplication:
\[
        \Delta_B(\omega+\lambda\cdot1)=\Delta_B(\omega),
        \qquad
        \Delta_B(a\omega)=a^6\Delta_B(\omega)
        \quad(a\in F^\times).
\]
Let \(W=B/F\cdot1\), a two-dimensional \(F\)-vector space.  The discriminant
therefore descends to a homogeneous binary sextic \(\widetilde\Delta_B\) on
\(W\).  Since \(\psi(a^6)=1\) for every \(a\in F^\times\), the value
\(\psi(\widetilde\Delta_B(v))\) depends only on the line \([v]\in\PP(W)\) when
\(v\ne0\).  The same observation gives the projective Kummer sheaf precisely:
on two local trivializations of the tautological line over \(\PP(W)\), the
local equations for \(\widetilde\Delta_B\) differ by a sixth power, and this
transition factor is invisible to the cubic character \(\psi\).  Hence these
local Kummer sheaves glue to a well-defined rank-one sheaf on the complement of
the zero divisor of \(\widetilde\Delta_B\).  Therefore
\[
        \sum_{\omega\in B}\psi(\Delta_B(\omega))
        =q(q-1)\sum_{\ell\in\PP(W)(F)}
        \psi(\widetilde\Delta_B(\ell)),
\]
where the notation in the projective sum means the following: choose any
nonzero vector \(v\in\ell\), set
\(\psi(\widetilde\Delta_B(\ell))=\psi(\widetilde\Delta_B(v))\), and interpret the
summand as \(0\) when \(\widetilde\Delta_B\) vanishes on \(\ell\).  This is
independent of the chosen vector because replacing \(v\) by \(av\) multiplies
\(\widetilde\Delta_B(v)\) by \(a^6\), and \(\psi(a^6)=1\).  Equivalently, the
summand is the trace function of the Kummer sheaf just constructed from the
section \(\widetilde\Delta_B\in H^0(\PP(W),\mathcal O(6))\).

After base change to \(\overline F\), the algebra \(B\) becomes
\(\overline F^3\), and \(W\) becomes the quotient of \(\overline F^3\) by the
diagonal line.  The zero divisor of \(\widetilde\Delta_B\) on \(\PP(W)\) is the
three-point divisor given by
\[
        x_1=x_2,
        \qquad x_1=x_3,
        \qquad x_2=x_3,
\]
each with multiplicity two.  Thus the Kummer sheaf associated with
\(\psi(\widetilde\Delta_B)\) is a geometrically nontrivial rank-one tame sheaf
on \(\PP^1\) minus three geometric points: the local exponent is \(2\) modulo
\(3\) at each point, so the local monodromy is nontrivial.  Let \(V\) denote
this three-punctured projective line.  Then
\[
        H_c^0(V_{\overline F},\mathcal L)=H_c^2(V_{\overline F},\mathcal L)=0,
\]
and Grothendieck--Ogg--Shafarevich \cite[Expos\'e X]{DeligneSGA41half} gives
\[
        \chi_c(V_{\overline F},\mathcal L)=2-3=-1,
        \qquad
        \dim H_c^1(V_{\overline F},\mathcal L)=1 .
\]
By Deligne's Riemann Hypothesis for curves \cite{DeligneWeilII},
\[
        \left|\sum_{\ell\in\PP(W)(F)}
        \psi(\widetilde\Delta_B(\ell))\right|\le\sqrt q .
\]
Multiplying by \(q(q-1)\) gives the claimed bound.  The argument is geometric:
the three splitting types of \(B\) only twist the Frobenius action on the three
geometric branch points, and do not change the cohomological dimension.  This
one-dimensional reduction makes the required Betti bound explicit; in a more
general arrangement-sheaf formulation it is a special case of the uniform
Betti-number estimates in \cite{KatzBetti2001}.
\end{proof}

\begin{theorem}[Equidistribution of singular cube classes]
\label{thm:singular-cube-class-statistics}
Let \(F=\F_q\) be a finite field of characteristic different from \(2\) and
\(3\), and let \(B/F\) be a finite \'etale cubic algebra.  Let \(A\in F^\times\).  If \(q\equiv1\pmod3\), then for each
cube class \(\kappa\in F^\times/(F^\times)^3\),
\[
        \#\{\omega\in B_{\rm gen}:A\Delta_B(\omega)\in\kappa\}
        =\frac13\#B_{\rm gen}+O(q^{5/2}),
\]
as \(q=|F|\to\infty\), with an absolute implied constant uniform over the
three finite \'etale cubic \(F\)-algebra types, over \(A\in F^\times\), and over
cube classes \(\kappa\).  If \(q\equiv2\pmod3\), the cube map on
\(F^\times\) is bijective, so there is only one cube class.

Consequently, for fixed \(\gamma\in B^\times\) and \(\delta\in F^\times\), let
\[
        S_\omega(\gamma,\delta)
        =\#\{u\in F^\times:u^3=-\Norm(\gamma)\delta\Delta_B(\omega)\}.
\]
Then
\[
        \sum_{\omega\in B_{\rm gen}}S_\omega(\gamma,\delta)
        =\#B_{\rm gen}+O(q^{5/2})
\]
when \(q\equiv1\pmod3\), while equality holds when \(q\equiv2\pmod3\).  Thus a
fixed norm fiber contributes one singular class on average as \(\omega\) varies
over generators, with square-root relative error.
\end{theorem}

\begin{proof}
When \(q\equiv1\pmod3\), choose any representative of the cube class \(\kappa\), still denoted \(\kappa\).  For \(\omega\in B_{\rm gen}\), orthogonality of the three cubic characters gives
\[
1_{A\Delta_B(\omega)\in\kappa}
=\frac13\sum_{j=0}^2 \psi^j(A\Delta_B(\omega)\kappa^{-1}) .
\]
The expression is independent of the representative because \(\psi\) is cubic.
For non-generators, \(\Delta_B(\omega)=0\), and the nontrivial cubic
characters are extended by zero.  Hence the nontrivial character sums over
\(B_{\rm gen}\) equal the corresponding discriminant-character sums over all
of \(B\), with the zero-discriminant locus contributing zero.  The \(j=0\)
term gives \(\#B_{\rm gen}/3\), and the two nontrivial terms are bounded by
Lemma~\ref{lem:discriminant-character-sum}.  The statement for
\(S_\omega(\gamma,\delta)\) follows because, for \(q\equiv1\pmod3\), the
cubic equation has three roots exactly on one of the three cube classes and no
roots on the other two.  When \(q\equiv2\pmod3\), the cube map is bijective on
\(F^\times\), so every nonzero right-hand side has exactly one cube root.
\end{proof}

\begin{corollary}[Average singular count in full norm-fiber branch families]
\label{cor:average-singular-branch-count}
In the homogeneous full norm-fiber setting of
Theorem~\ref{thm:finite-field-branch-census}, fix \(B/F\), \(\gamma\in B^\times\),
and \(\delta\in F^\times\), and vary the primitive tangent generator \(\omega\)
over \(B_{\rm gen}\).  The average number of primitive first-order singular
classes in the norm fiber \(\Norm(h)=\delta\) is
\[
        1+O(q^{-1/2})
\]
if \(q\equiv1\pmod3\), and is exactly \(1\) if \(q\equiv2\pmod3\).
\end{corollary}

\begin{proof}
This is Theorem~\ref{thm:singular-cube-class-statistics} divided by
\(\#B_{\rm gen}\), using the explicit generator counts in
Lemma~\ref{lem:generator-discriminant-count}, which give \(\#B_{\rm gen}\asymp q^3\).
\end{proof}

\begin{theorem}[Uniform quadratic Hensel alternatives in lift families]
\label{thm:quadratic-jet-statistics}
Let \(p\ge5\), let \(A/\Z_p\) be a finite \'etale cubic algebra, put
\(B=A/pA\), and let \(\omega\in B\) generate \(B\).  Fix a reduced
nondegenerate affine singular class
\[
        x\in B^\times,
        \qquad
        \Tr(x)=s,
        \qquad
        \Tr(\omega x)=0,
        \qquad
        \Delta=\Tr(\omega^2x)\ne0.
\]
Fix a lift \(U\in A\) of \(\omega\), a target \(c\in\Z_p\) with \(\bar c=s\),
and a preliminary lift \(y_0\in A/p^3A\) of \(x\).  The assertions are made for
this fixed lift \(U\); replacing \(U\) by another lift of \(\omega\) may change
the affine offsets in the coefficients below, but it does not change the fiber
sizes or the discriminant frequencies.  Let \(\mathcal Y\) be the affine set of
classes \(y\in A/p^3A\) reducing to \(x\) modulo \(p\).  It is a torsor under
\(pA/p^3A\).  This theorem is a statistic in the full affine lift space
\(\mathcal Y\).  Applying the resulting frequencies to jets coming from a fixed
global recurrence requires a separate argument that those recurrence-induced
jets sample this lift space, for example by equidistribution or another explicit
parametrization.  The associated graded of this translation group is canonically
\[
        pA/p^2A\oplus p^2A/p^3A\simeq B\oplus B .
\]
Choose any splitting of the filtration only to write a translation from \(y_0\)
as
\[
        y=y_0+p\alpha+p^2\beta\pmod {p^3},
        \qquad \alpha,\beta\in B.
\]
The conclusions below are independent of this auxiliary splitting.  Intrinsically,
the survival map has first graded linear part
\(\alpha\mapsto\Tr_{B/\F_p}(\alpha)\), the coefficient \(B_y\) has first graded
linear part \(\alpha\mapsto\Tr_{B/\F_p}(\alpha\omega)\) on the survival
hyperplane, and, after the first graded coordinate is fixed, \(A_y\) has second
graded linear part \(\beta\mapsto\Tr_{B/\F_p}(\beta)\).  A different splitting
only composes these coordinates with a fiber-preserving affine automorphism of
the filtered torsor.  Impose the
single survival condition that the class of \(\Tr(y)-c\) in
\(\Z_p/p^3\Z_p\) lies in the subgroup \(p^2\Z_p/p^3\Z_p\).  This makes
\(\Tr(y)-c\in p^2\Z_p/p^3\Z_p\).  Since \(\bar y=x\), \(\bar U=\omega\), and
\(\Tr(x\omega)=0\), one also has \(\Tr(yU)\in p\Z_p/p^3\Z_p\).  Hence, for each
surviving lift, the quotients
\[
        A_y\equiv\frac{\Tr(y)-c}{p^2}\pmod p,
        \qquad
        B_y\equiv\frac{\Tr(yU)}p\pmod p
\]
are well defined.  If \(U\) is replaced by another lift \(U+pV\), then \(B_y\)
is translated by the affine offset \(\Tr(xV)\) modulo \(p\), while the fiber
sizes and the discriminant frequencies below are unchanged.
For each surviving \(y\), choose any integral lift \(\widetilde y\in A\) of its
class modulo \(p^3\) and attach the order-\(p^3\) branch jet
\[
        F_{y,c}(T)=\Tr_A\bigl(\widetilde y(1+pU)^T\bigr)-c
        \pmod {p^3}.
\]
This truncated branch jet, and therefore its first-digit lifting alternatives,
depend only on \(y\in A/p^3A\), not on the chosen integral lift
\(\widetilde y\).  The alternatives in the table below refer to the possible
values of the first branch digit \(T\bmod p\) for this jet.
Then the map from surviving lifts to \(\F_p^2\),
\[
        y\longmapsto(A_y,B_y),
\]
is uniform: every pair \((A_0,B_0)\in \F_p^2\) occurs the same number of times.
Consequently the reduced quadratic Hensel polynomial
\[
        Q_y(X)=A_y+B_yX+\Delta\binom X2
\]
satisfies, for the branch jet just defined,
\[
        \frac{F_{y,c}(T)}{p^2}\equiv Q_y(T)\pmod p.
\]
Thus the roots of \(Q_y\) are exactly the first branch digits that survive to
precision \(p^3\).  The polynomial \(Q_y\) has the following exact first-digit
statistics, where the frequencies are
conditional frequencies among the surviving lifts in the full affine lift space
\(\mathcal Y\), with denominator
\(\#\{y\in\mathcal Y:\Tr(y)-c\in p^2\Z_p/p^3\Z_p\}\):
\begin{center}
\small
\begin{tabular}{>{\raggedright\arraybackslash}p{0.24\textwidth}
                >{\centering\arraybackslash}p{0.24\textwidth}
                >{\raggedright\arraybackslash}p{0.32\textwidth}}
\toprule
Discriminant of \(Q_y\) & Conditional frequency among surviving lifts & First-digit alternative\\
\midrule
nonsquare & \((p-1)/(2p)\) & dies before \(p^3\)\\
nonzero square & \((p-1)/(2p)\) & two simple branches\\
zero & \(1/p\) & one double first digit; later behavior is governed by the distinguished quadratic\\
\bottomrule
\end{tabular}
\end{center}
\end{theorem}

\begin{proof}
First formulate the two relevant coefficient maps intrinsically on the filtered
torsor \(\mathcal Y\).  Let
\[
        \Lambda_0(y)=\Tr(y)-c\in\Z_p/p^3\Z_p,
        \qquad
        \Lambda_1(y)=\Tr(yU)\in\Z_p/p^3\Z_p .
\]
The translation group of \(\mathcal Y\) is \(pA/p^3A\), with filtration
\(pA/p^3A\supset p^2A/p^3A\), and associated graded pieces canonically
identified with \(B\) and \(B\).  The graded linear part of \(\Lambda_0\) on
\(pA/p^2A\) is
\[
        \alpha\longmapsto p\Tr_{B/\F_p}(\alpha)
        \quad\bmod p^2,
\]
and, after the first graded component has been fixed, its graded linear part on
\(p^2A/p^3A\) is
\[
        \beta\longmapsto p^2\Tr_{B/\F_p}(\beta)
        \quad\bmod p^3.
\]
Similarly, because \(\bar U=\omega\), the graded linear part of
\(p^{-1}\Lambda_1\bmod p\) in the first graded component is
\[
        \alpha\longmapsto \Tr_{B/\F_p}(\alpha\omega).
\]
These graded maps are independent of the auxiliary splitting.  Choosing a
splitting merely writes a point as \(y=y_0+p\alpha+p^2\beta\); changing the
splitting composes these coordinates with a fiber-preserving affine
automorphism, replacing \(\beta\) by \(\beta\) plus an affine function of
\(\alpha\).  The intrinsic maps \(A_y\) and \(B_y\), and hence their fiber sizes,
are unchanged.  Thus it suffices to compute the fibers in one splitting.

For an integral representative \(\widetilde y\) of a surviving class, the expansion
\[
        \Tr_A\bigl(\widetilde y(1+pU)^T\bigr)-c
        =\Tr(\widetilde y)-c+pT\Tr(\widetilde yU)
          +p^2\binom T2\Tr(\widetilde yU^2)\pmod {p^3}
\]
shows, after division by \(p^2\), that the reduced first-digit polynomial is
\[
        A_y+B_yT+\Delta\binom T2,
\]
because \(\widetilde y\equiv x\pmod p\) and \(U\equiv\omega\pmod p\).  Thus
the root pattern of \(Q_y\) is exactly the first-digit pattern for
\(F_{y,c}(T)\equiv0\pmod {p^3}\).

The survival condition fixes the coefficient of \(p\) in \(\Tr(y)-c\).  Its
linear part in the first graded coordinate is the functional
\(\alpha\mapsto\Tr(\alpha)\), which is nonzero because
\(\Tr_{B/\F_p}(1)=3\ne0\); hence it is surjective over \(\F_p\), and the
surviving \(\alpha\)'s form an affine hyperplane in \(B\).  On this hyperplane,
\(B_y\) is an affine function of
\(\alpha\) with linear part
\[
        \alpha\longmapsto\Tr(\alpha\omega).
\]
This functional is not constant on the hyperplane \(\Tr(\alpha)=\text{constant}\):
if it were, then \(\Tr(\alpha\omega)\) would be a scalar multiple of
\(\Tr(\alpha)\), and nondegeneracy of the trace pairing would force
\(\omega\in \F_p\cdot1\), contrary to the hypothesis that \(\omega\) generates the
cubic algebra.  Hence \(B_y\) is uniformly distributed in \(\F_p\) as the
surviving first graded component varies.

After \(\alpha\) is fixed, \(A_y\) is an affine function of the second graded
component \(\beta\) with linear part
\[
        \beta\longmapsto\Tr(\beta),
\]
which is again nonzero because \(\Tr_{B/\F_p}(1)=3\ne0\).  Therefore \(A_y\) is
uniformly distributed in \(\F_p\), independently of the
already chosen value of \(B_y\).  This proves the uniformity of \((A_y,B_y)\).

For fixed \(\Delta\ne0\), the discriminant of
\[
        A+B X+\Delta\binom X2
\]
is
\[
        D=(B-\Delta/2)^2-2\Delta A.
\]
As \((A,B)\) ranges uniformly over \(\F_p^2\), the value \(D\) ranges uniformly
over \(\F_p\): for each \(B\) and each prescribed \(D\), there is a unique
\(A\).  The three frequencies are therefore the numbers of nonsquares, nonzero
squares, and zero in \(\F_p\), divided by \(p\).  These are first-digit alternatives.  In the double-root row, later \(p\)-adic splitting depends on the higher coefficients of the distinguished quadratic in the corresponding residue disk.  The corresponding alternatives are exactly those of
Corollaries~\ref{cor:primitive-singular-alternatives} and
\ref{cor:affine-nondegenerate-singular}.
\end{proof}

\begin{remark}[What the statistics do and do not assert]
Theorem~\ref{thm:singular-cube-class-statistics} is a finite-field family
statement: it explains how often singular classes appear as the reduced tangent
\(\omega\) varies.  Theorem~\ref{thm:quadratic-jet-statistics} is a finite-jet
statement: once a nondegenerate singular class is present, it explains the
unbiased distribution of the next two Hensel coefficients in the natural lift
family.  A fixed global recurrence may impose additional arithmetic constraints
on these jets, but there is no hidden local geometric bias beyond the exact
codifferent census already proved above.
\end{remark}

\section{Further directions and limits of the present results}
\label{sec:further-directions}

The results above are all proved under the stated good-prime and finite-\'etale
hypotheses.  The following problems record natural extensions suggested by the
proofs; none is used in the theorems above.

\begin{problem}[Bad and ramified primes]
Extend the local branch classification to primes at which the reduction of
\(A\) is not finite \'etale, or to the small primes \(p=2,3\).  The first
obstruction is that the trace pairing may become degenerate and the divisor and
Kummer arguments used here no longer have tame prime-to-\(p\) monodromy.
\end{problem}

\begin{problem}[Higher-degree subgroup trace statistics]
For finite \'etale algebras of degree \(d>3\), formulate an analogue of the
smooth subgroup-coset theorem for norm-one tori and trace hyperplane sections.
The local Weierstrass bounds in Appendix~\ref{sec:higher-rank} suggest the
correct local degree, but the finite-field fibers have dimension \(d-2\) rather
than curves, so the Picard--Kummer and nodal arguments of
Section~\ref{sec:proper-subgroups} are not directly available.
\end{problem}

\begin{problem}[Toric Wieferich primes]
For a fixed cubic field \(K\) and norm-one unit \(\eta\), study the distribution
of inert unramified primes satisfying
\[
        \eta^P\equiv1\pmod {p^2},
        \qquad P=\ord_{(\cO_K/p\cO_K)^\times}(\bar\eta).
\]
Appendix~\ref{sec:inert-wieferich} shows that these are exactly the inert primes
where the primitive-tangent basis condition fails; no density statement is
proved here.
\end{problem}

\appendix
\section{Higher-rank optimality and jet versality}\label{sec:higher-rank}

The preceding branch theory is cubic, but the Weierstrass mechanism has a
rank-\(d\) form.  The local section already gave the rank-\(d\) bound and the
tangent-subalgebra refinement.  The following theorem records the complementary
optimality statement: the reduced jet is not merely bounded in degree; every
allowed jet occurs in a split toric branch.

\begin{theorem}[Split jet versality]\label{thm:split-jet-versality}
Let \(p>d\), let
\[
        A=\Z_p^d,
\]
and choose \(\Omega_1,\ldots,\Omega_d\in\Z_p\) with pairwise distinct
reductions modulo \(p\).  Put
\[
        \Omega=(\Omega_1,\ldots,\Omega_d),
        \qquad
        \eta=1+p\Omega.
\]
Then \(P=1\), the logarithmic tangent is
\(\omega=\bar\Omega\in\F_p^d\), and
\(1,\omega,\ldots,\omega^{d-1}\) is an \(\F_p\)-basis of \(A/pA\).

Let
\[
        Q(X)=\sum_{m=0}^{e}c_m\binom Xm\in\F_p[X],
        \qquad 0\le e\le d-1,
        \qquad c_e\ne0.
\]
Then there exists a primitive coefficient \(\gamma\in A\) such that
\[
        F(X)=\Tr_A(\gamma\eta^X)
\]
satisfies
\[
        p^{-e}F(X)\equiv Q(X)\pmod p
\]
as a restricted power series.  Thus every reduced jet of degree at most
\(d-1\) occurs in a split toric branch.
\end{theorem}

\begin{proof}
The Vandermonde determinant of the reductions of the \(\Omega_i\)'s is nonzero,
so \(1,\omega,\ldots,\omega^{d-1}\) is a basis of \(\F_p^d\).  The integral
Vandermonde determinant \(\det(\Omega_i^j)_{1\le i\le d,\,0\le j\le d-1}\) is
therefore a \(p\)-adic unit.  Hence the trace-pairing matrix against
\(1,\Omega,\ldots,\Omega^{d-1}\) is invertible over \(\Z_p\), and the dual
basis lies in \(A=\Z_p^d\).  Let \(z_0,\ldots,z_{d-1}\in A\) be this
\(\Z_p\)-dual basis, so that
\[
        \Tr_A(z_j\Omega^m)=\delta_{jm}\qquad(0\le j,m\le d-1).
\]
Choose lifts of the coefficients \(c_m\) to \(\Z_p\), and set
\[
        \gamma=\sum_{j=0}^{e}p^{e-j}c_jz_j.
\]
Since \(c_e\ne0\), the reduction of \(\gamma\) modulo \(p\) is nonzero, so
\(\gamma\) is primitive.  The binomial expansion gives
\[
        F(X)=\sum_{m\ge0}\binom Xm p^m\Tr_A(\gamma\Omega^m).
\]
For \(0\le m\le e\), the dual-basis construction gives
\[
        \Tr_A(\gamma\Omega^m)=p^{e-m}c_m,
\]
so these terms contribute \(c_m\binom Xm\) to \(p^{-e}F(X)\) modulo \(p\).
For \(e<m<d\), the factor \(p^{m-e}\) kills the term modulo \(p\).  For
\(m\ge d\), the same estimate used in Theorem~\ref{thm:rank-d-weierstrass-bound}
shows that the corresponding monomial coefficients of
\(p^{m-e}\binom Xm\) are divisible by \(p\), because \(e\le d-1<p\).  Hence no
term with \(m>d-1\) contributes modulo \(p\), and the reduction is exactly
\(Q(X)\).
\end{proof}

\begin{corollary}[Optimality of the local degree bound]\label{cor:optimality-degree-bound}
The numerical bound \(d-1\) in Theorem~\ref{thm:rank-d-weierstrass-bound}
cannot be improved, even if one restricts to split finite \'etale algebras and
toric branches with \(\eta\equiv1\pmod p\).  More strongly, every possible
reduced Weierstrass jet of degree \(\le d-1\) is realized by such a branch.
\end{corollary}

\begin{proof}
The first sentence already follows from
Proposition~\ref{prop:rank-d-sharpness}.  The stronger statement is
Theorem~\ref{thm:split-jet-versality}.
\end{proof}

\begin{theorem}[Sharp affine rank-\(d\) Weierstrass bound]
\label{thm:affine-rank-d-bound}
Let \(A\) be a finite \'etale \(\Z_p\)-algebra of rank \(d\), with
\(p>d+1\).  Let \(\eta\in A^\times\), let \(P\) be the order of
\(\bar\eta\in(A/pA)^\times\), and write
\[
        \eta^P=1+pU,
        \qquad \omega=\bar U\in A/pA .
\]
Assume that
\[
        1,
        \omega,
        \ldots,
        \omega^{d-1}
\]
is an \(\F_p\)-basis of \(A/pA\), and assume that \(\omega\) is a unit in
\(A/pA\).  Fix a branch coefficient \(y=\gamma\eta^a\in A\), put
\(x=\bar y\), and let \(c\in\Z_p\).  Suppose
\[
        x\ne0,
        \qquad \Tr_{A/pA/\F_p}(x)=\bar c .
\]
Then the affine branch
\[
        F_c(X)=\Tr_A\bigl(y(\eta^P)^X\bigr)-c
\]
has a distinguished Weierstrass factor of degree at most \(d\).  More
precisely, if
\[
        C_0=\Tr_A(y)-c,
        \qquad C_m=\Tr_A(yU^m)\quad(m\ge1),
\]
and, with the convention \(v_p(0)=+\infty\),
\[
        s=\min\Bigl( v_p(C_0),\ \min_{1\le m\le d}
        \bigl(m+v_p(C_m)\bigr)\Bigr),
\]
then \(s\le d\), and \(p^{-s}F_c(X)\) has nonzero reduction of degree at most
\(d\).  Hence \(F_c\) has at most \(d\) zeros in \(\Z_p\), counted with
Weierstrass multiplicity.

The bound \(d\) is sharp, already for split algebras with \(\eta\equiv1\pmod p\).
\end{theorem}

\begin{proof}
The trace pairing on \(A/pA\) is nondegenerate.  Let
\(z_0,
\ldots,z_{d-1}\) be the trace-dual basis to
\(1,
\omega,
\ldots,
\omega^{d-1}\).  If some
\(\Tr(x\omega^m)\ne0\) for \(1\le m\le d-1\), then
\(m+v_p(C_m)\le d-1\).  Otherwise \(x\) is orthogonal to
\(\omega,
\ldots,
\omega^{d-1}\).  If also \(\Tr(x)=0\), then \(x=0\), contrary to the
hypothesis.  Thus this exceptional case has \(\Tr(x)=\bar c\ne0\), and
\(x=\bar c\,z_0\).

Let \(m_\omega:A/pA\to A/pA\) be multiplication by \(\omega\), and write
\[
        f_\omega(T)=\det(T\cdot\operatorname{id}_{A/pA}-m_\omega)
        =T^d+a_{d-1}T^{d-1}+\cdots+a_0 .
\]
Because \(\omega\) generates \(A/pA\), this is the monic separable generator
polynomial of \(\omega\).  Since \(\omega\) is a unit, \(a_0\ne0\).  The relation \(f_\omega(\omega)=0\) gives
\[
        \Tr(z_0\omega^d)=-a_0 .
\]
Hence, in the exceptional case,
\[
        \Tr(x\omega^d)=-\bar c\,a_0\ne0,
\]
so \(d+v_p(C_d)=d\).  We have proved \(s\le d\).

The binomial expansion gives
\[
        F_c(X)=C_0+
        \sum_{m\ge1}\binom Xm p^m C_m .
\]
After division by \(p^s\), at least one term with \(m\le d\) survives modulo
\(p\), and no term with \(m>d\) survives.  Indeed, for \(m>d\), Legendre's
estimate gives \(v_p(m!)\le(m-1)/(p-1)\), and therefore the coefficient has
valuation at least
\[
        m-s-v_p(m!)\ge m-d-v_p(m!)
        \ge m-d-\frac{m-1}{p-1}>0.
\]
The last expression is minimized at \(m=d+1\), where it is
\(1-d/(p-1)>0\), since \(p>d+1\).  The valuation lower bound is an integer, so
it is at least \(1\).  This is again coefficientwise, using
\(\binom Xm=m!^{-1}\prod_{i=0}^{m-1}(X-i)\).  Thus the reduction of \(p^{-s}F_c\) is a nonzero polynomial
of degree at most \(d\) in the binomial basis.  Weierstrass preparation gives a
distinguished factor of that degree, and the zero bound follows.

For sharpness, since \(p>d+1\), the group \(\F_p^\times\) has
\(p-1>d\) elements.  Choose \(d\) pairwise distinct nonzero residue classes and
lift them to unit elements
\(\Omega_1,
\ldots,
\Omega_d\in\Z_p\) with pairwise distinct reductions.  Put
\(\eta=1+p\Omega\), where \(\Omega=(\Omega_1,
\ldots,
\Omega_d)\), and let \(z_0\) be the trace-dual element to \(1\).  Set
\(y=z_0\) and \(c=1\).  For \(0\le r\le d-1\), the element
\((1+p\Omega)^r\) is a polynomial in \(\Omega\) of degree \(<d\) with constant
term \(1\), so
\[
        \Tr(z_0(1+p\Omega)^r)-1=0 .
\]
On the other hand,
\[
        \Tr(z_0(1+p\Omega)^d)-1
        =p^d\Tr(z_0\Omega^d)=-a_0p^d,
\]
where \(a_0=(-1)^d\Omega_1\cdots\Omega_d\) is a unit.  Thus the branch has the
\(d\) distinct zeros \(0,
1,
\ldots,
d-1\), and the bound cannot be lowered.
\end{proof}

\begin{remark}[Why the affine bound differs from the homogeneous one]
For the homogeneous target \(c=0\), a nonzero reduced class cannot be
orthogonal to all of \(1,
\omega,
\ldots,
\omega^{d-1}\); the first nonzero jet occurs before degree \(d\).  For a
nonzero affine target, the reduced class \(x=\bar c\,z_0\) kills the constant
term after subtracting \(c\) and is orthogonal to
\(\omega,
\ldots,
\omega^{d-1}\).  The first nonzero tangent can then occur at \(\omega^d\),
and the sharp bound is \(d\).
\end{remark}

\section[Inert norm-one tangents and toric Wieferich primes]
{\texorpdfstring{Inert norm-one tangents\\ and toric Wieferich primes}
{Inert norm-one tangents and toric Wieferich primes}}\label{sec:inert-wieferich}

The primitive-tangent basis condition is local and algebraic.  In the inert
norm-one case it has a simple arithmetic interpretation: it fails exactly at a
toric Wieferich congruence.

\begin{theorem}[Scalar tangent equals toric Wieferich]\label{thm:inert-wieferich}
Let \(K/\Q\) be a cubic field, let \(p\ge5\) be inert and unramified in \(K\),
and let \(\eta\in\cO_K^\times\) satisfy
\[
        \Norm_{K/\Q}(\eta)=1.
\]
Let \(P\) be the order of \(\bar\eta\in(\cO_K/p\cO_K)^\times\), and define
\[
        \omega_p=\frac{\eta^P-1}{p}\bmod p
        \in \cO_K/p\cO_K\simeq\F_{p^3}.
\]
Then
\[
        \omega_p\in\F_p
        \quad\Longleftrightarrow\quad
        \omega_p=0
        \quad\Longleftrightarrow\quad
        \eta^P\equiv1\pmod {p^2}.
\]
Consequently, in the inert norm-one case, the primitive-tangent basis condition
\[
        1,\omega_p,\omega_p^2\text{ is an }\F_p\text{-basis of }\F_{p^3}
\]
fails if and only if \(\eta^P\equiv1\pmod {p^2}\).
\end{theorem}

\begin{proof}
If \(\omega_p\in\F_p\), write
\[
        \eta^P\equiv1+pc\pmod {p^2},
        \qquad c\in\F_p.
\]
Since \(\Norm_{K/\Q}(\eta)=1\), also \(\Norm(\eta^P)=1\).  Taking norms modulo
\(p^2\) gives
\[
        1\equiv\Norm(1+pc)\equiv(1+pc)^3\equiv1+3pc\pmod {p^2}.
\]
Because \(p\ge5\), this forces \(c\equiv0\pmod p\).  Hence \(\omega_p=0\).
The equivalence with \(\eta^P\equiv1\pmod {p^2}\) is immediate from the
definition of \(\omega_p\).  Finally, \(\F_{p^3}\) has no proper subfield other
than \(\F_p\); therefore \(1,\omega_p,\omega_p^2\) fails to span \(\F_{p^3}\)
if and only if \(\omega_p\in\F_p\), which we have shown is equivalent to
\(\omega_p=0\).
\end{proof}

\begin{theorem}[Higher tangent restart at inert norm-one primes]
\label{thm:higher-tangent-restart}
Let \(K/\Q\), \(p\), \(\eta\), and \(P\) be as in
Theorem~\ref{thm:inert-wieferich}.  Assume \(\eta^P\ne1\), and let \(r\ge1\) be
maximal such that
\[
        \eta^P\equiv1\pmod {p^r}.
\]
Define the first nonzero higher logarithmic tangent
\[
        \omega_{p}^{(r)}=\frac{\eta^P-1}{p^r}\bmod p
        \in \cO_K/p\cO_K\simeq\F_{p^3}.
\]
Then
\[
        \omega_{p}^{(r)}\notin\F_p.
\]
Consequently \(1,\omega_p^{(r)},(\omega_p^{(r)})^2\) is an \(\F_p\)-basis of
\(\F_{p^3}\).  Thus an inert norm-one toric Wieferich prime does not create an
unresolved local exception: it only shifts the first nonzero tangent from order
\(p\) to order \(p^r\).
\end{theorem}

\begin{proof}
Write
\[
        \eta^P=1+p^r V,
        \qquad \omega_p^{(r)}=\bar V.
\]
Suppose \(\omega_p^{(r)}=c\in\F_p\).  Since \(\Norm(\eta)=1\), we have
\(\Norm(\eta^P)=1\).  Reducing the norm modulo \(p^{r+1}\) gives
\[
        1=\Norm(1+p^rV)
        \equiv \Norm(1+p^r c)
        \equiv (1+p^r c)^3
        \equiv 1+3p^r c\pmod {p^{r+1}}.
\]
Because \(p\ge5\), this forces \(c=0\).  But \(c=0\) means
\(\eta^P\equiv1\pmod {p^{r+1}}\), contradicting the maximality of \(r\).  Hence
\(\omega_p^{(r)}\notin\F_p\).  The field \(\F_{p^3}\) has no intermediate
subfield other than \(\F_p\), so a non-scalar element generates it as an
\(\F_p\)-algebra.
\end{proof}

\begin{proposition}[Higher-order transverse lifting]
\label{prop:higher-order-transverse-lifting}
Let \(A/\Z_p\) be finite \'etale and let \(p\ge5\).  Let
\(\eta\in A^\times\), \(\gamma\in A\), \(a\in\Z\), and \(P\ge1\) be such that
\[
        \eta^P=1+p^rU,
        \qquad r\ge1,
        \qquad U\in A.
\]
Here \(P\) may be the order of \(\bar\eta\) or any exponent satisfying the displayed congruence.  Put
\[
        \omega^{(r)}=\bar U\in A/pA.
\]
For an affine target \(c\in\Z_p\), set
\[
        F_{a,c}(t)=\Tr_A\bigl(\gamma\eta^a(\eta^P)^t\bigr)-c.
\]
Assume \(F_{a,c}(0)\in p^r\Z_p\) and
\[
        d_a^{(r)}=\Tr_{A/pA/\F_p}(\bar\gamma\bar\eta^a\omega^{(r)})\ne0.
\]
Then there is a unique \(\tau\in\Z_p\) with \(F_{a,c}(\tau)=0\), and
\[
        v_p(F_{a,c}(t))=r+v_p(t-\tau)
        \qquad(t\in\Z_p).
\]
\end{proposition}

\begin{proof}
Write \(y=\gamma\eta^a\).  The binomial expansion gives, in
\(\Z_p\langle t\rangle\),
\[
        F_{a,c}(t)=F_{a,c}(0)+
        \sum_{m\ge1}\binom tm p^{rm}\Tr_A(yU^m).
\]
The hypothesis \(F_{a,c}(0)\in p^r\Z_p\) makes the constant term divisible by
\(p^r\).  For \(m\ge1\), coefficientwise divisibility after division by \(p^r\)
follows from
\[
        rm-r-v_p(m!)=r(m-1)-v_p(m!)\ge0.
\]
Indeed this is equality for \(m=1\), while for \(m\ge2\) it follows from
\(r\ge1\), \(p\ge5\), and \(v_p(m!)\le(m-1)/(p-1)<m-1\).  Thus
\(F_{a,c}(t)/p^r\in\Z_p\langle t\rangle\).  Moreover, for every \(m\ge2\), the
same estimate gives
\[
        rm-r-v_p(m!)\ge1,
\]
so all terms of degree at least two vanish modulo \(p\) after division by
\(p^r\).  Therefore
\[
        \frac{F_{a,c}(t)}{p^r}\equiv
        \frac{F_{a,c}(0)}{p^r}+t d_a^{(r)}\pmod p.
\]
Hensel's lemma applied to this restricted \(\Z_p\)-power series gives a unique
zero \(\tau\), and division by \(t-\tau\) in the restricted power-series ring
gives the valuation formula.
\end{proof}

\begin{remark}[Exceptional primes and distribution]\label{rem:wieferich-distribution}
The theorem is local: it identifies which inert primes can violate the
primitive-tangent hypothesis.  Estimating the density of those toric Wieferich
primes is a separate global problem and is not used in the local theory.

\end{remark}

\end{document}